\documentclass[11pt]{article}
\usepackage[english]{babel}

\usepackage[letterpaper,top=2cm,bottom=2cm,left=3cm,right=3cm,marginparwidth=1.75cm]{geometry}
\usepackage[titletoc,title]{appendix}
\usepackage{amsmath,amssymb,mathrsfs,amssymb,amsthm}
\usepackage{enumerate}
\usepackage{cases}
\usepackage{graphicx}
\usepackage[colorlinks=true, allcolors=blue]{hyperref}
\usepackage{color}
\newcommand{\vep}{\varepsilon}
\newcommand{\R}{\mathbb R}
\newcommand{\CC}{\mathbb C}
\definecolor{HW}{rgb}{1,0,0}
\definecolor{HW1}{rgb}{0,0,1}

\numberwithin{equation}{section}
\numberwithin{figure}{section}
\numberwithin{table}{section}

\newtheorem{theorem}{Theorem}[section]
\newtheorem{lemma}{Lemma}[section]
\newtheorem{remark}{Remark}[section]
\newtheorem{corollary}{Corollary}[section]
\newtheorem{definition}{Definition}

\title{High Contrast Transmission and Fabry-P\'erot-type Resonances}
\author{Long Li \thanks{{RICAM, Austrian Academy of Sciences, A-4040, Linz, Austria (long.li@ricam.oeaw.ac.at)}} \; and Mourad Sini \thanks{{RICAM, Austrian Academy of Sciences, A-4040, Linz, Austria (mourad.sini@oeaw.ac.at)}}}

\date{}
\begin{document}
\maketitle
\begin{abstract}

\noindent It is well known, in the acoustic model, that highly contrasting transmission leads to the so-called Minnaert subwavelength resonance. In this work, we show that such highly contrasting transmissions create not only one resonance but a family of infinite resonances located near the real axis where the first one (i.e. the smallest) is indeed the Minnaert one. This family of resonances are the shifts (in the lower complex plan) of the Neumann eigenvalues of the Laplacian. The well known Minneart resonance is nothing but the shift of the trivial (zero) Neumann eigenvalue of the bubble. These resonances, other than the Minnaert ones, are Fabry-P\'erot-type resonances as the generated total fields, in the bubble, are dominated by a linear combination of the Neumann eigenfunctions which, in particular, might create interferences. In addition, we establish the following properties.
\begin{enumerate}
\item We derive the asymptotic expansions, at the second order, of this family of resonances in terms of the contrasting coefficient.

\item In the time-harmonic regime, we derive the  resolvent estimates of the related Hamiltonian and the asymptotics of scattered fields that are uniform in the whole space, highlighting the contributions from this sequence of resonances. 

\item In the time domain regime, we derive the time behavior of the acoustic microresonator at large time-scales inversely proportional to powers of microresonator's radius.

\item The analysis shows that near Fabry-P\'erot resonances, the mircoresonator exhibits pronounced anisotropy. We believe that such a feature may pave the way for designing anisotropic metamaterials from simple configurations of a single microresonator.

\end{enumerate}

\end{abstract}

\tableofcontents

\section{Introduction}

Let $\Omega \subset \mathbb R^3$ be an open bounded and connected domain. Given a positive real number $\tau$, consider the following Hamiltonian $H_{\rho_\tau,k_\tau}(\Omega)$
\begin{align} 
\left(H_{\rho_\tau, k_\tau}(\Omega)\right) u := -{k_\tau}\nabla \cdot \frac{1}{\rho_\tau} \nabla u \notag
\end{align}
with the domain 
\begin{align}
D(H_{\rho_\tau,k_\tau}(\Omega)):= \bigg\{u \in H^1(\R^3),\; k_\tau \nabla\cdot \frac{1}{\rho_\tau}\nabla u \in L^2(\R^3) \bigg\}, \notag
\end{align}
where $\rho_\tau$ and $k_\tau$ are defined by
\begin{align} \label{eq:0} 
\rho_\tau(x) := 
\begin{cases}
\rho_0,    &  x \in \R^3 \backslash \Omega, \\
{\rho_1}\tau,  & x \in \Omega,
\end{cases}
\quad {\rm{and}} \quad
k_\tau(x):=
\begin{cases}
k_0,       &  x \in \R^3 \backslash \Omega, \\
k_1\tau,   & x \in \Omega,
\end{cases} 
\end{align}
respectively. Here, $\rho_0, k_0, \rho_1$ and $k_1$ are all positive real numbers. 
It is well-established that, for each fixed $\tau$, the Hamiltonian $H_{\rho_\tau, k_\tau}(\Omega): \mathcal H \rightarrow \mathcal H$ with a domain $\mathcal D(H_{\rho_\tau,k_\tau}(\Omega)) \subset \mathcal H$ is a black box Hamiltonian (see Lemma 2.3 and Remark 2.4 in \cite{LSW} for more details). Here, $\mathcal H$ is defined by
\begin{align*}
&\mathcal H := \left\{u \in L^2(\R^3): \int_{\R^3} (k_\tau (x))^{-1} |u(x)|^2 dx < +\infty\right\}.
\end{align*}
For simplicity, we denote $H_{\rho_\tau,k_\tau}(\Omega)$ by $H_{\rho_\tau,k_\tau}$. Therefore, the resolvent of the Hamiltonian $H_{\rho_\tau, k_\tau}$, defined as 
\begin{align}
R_{H_{\rho_\tau, k_\tau}}(z):= \left(H_{\rho_\tau,k_\tau} - z^2 \right)^{-1}, \notag
\end{align}
is a meromorphic family of operators mapping from $\mathcal H_{\rm{comp}}$
to $\mathcal D_{\rm{loc}}(H_{\rho_\tau,k_\tau})$ (see Theorem 4.4 in \cite{DM}), where the spaces $\mathcal H_{\rm{comp}}$ and $\mathcal D_{\rm{loc}}(H_{\rho_\tau,k_\tau})$ are defined by \eqref{eq:164} and \eqref{eq:165}, respectively. This naturally leads to the definition of the scattering resonance, defined as a pole of the meromorphic extension of $R_{H_{\rho_\tau,k_\tau}}(z)$. The above Hamiltonian $H_{\rho_\tau, k_\tau}$ is associated with the acoustic transmission problem, in which the wave speeds inside and outside the penetrable obstacle $\Omega$ are denoted by 
\begin{align} \label{eq:41}
c_1:=\sqrt{k_1/\rho_1}\; \mathrm{and}\; c_0:=\sqrt{k_0/\rho_0},
\end{align}
respectively. When the acoustic frequency approaches the set of scattering resonances, $\Omega$ acts as an effective resonator.

There is a considerable literature on the distribution of the scattering resonance of the Hamiltonian $H_{\rho_\tau, k_\vep}$. In \cite{PV-99}, the authors proved that, when $c_1 > c_0$ and $\Omega$ is $C^{\infty}$ and convex with strictly positive curvature, there exists a resonance-free region near the real axis. This result subsequently was improved in \cite{CPV-99} to a resonance-free strip beneath the real axis. In contrast, when $c_1 < c_0$, the authors of \cite{P-V:1919} showed that there exists a sequence of scattering resonances super-algebraically close to the real axis. Further results on the location and asymptotics of the resonances, in the case where $\Omega$ is $C^{\infty}$ and convex with strictly positive curvature and for both wave speed configurations, were obtained in \cite{CPV-01}. Later, sharp bounds on the location of the resonances under the same geometric assumption were established in \cite{J19}. For the case when $\Omega$ is star-shaped Lipschitz, \cite{M-S:19} verified the existence of resonance-free strip beneath the real axis, assuming certain natural conditions on the ratio of wave speeds.

The current work focuses on the regime in which the coefficients of the Hamiltonian $H_{\rho_\tau, k_\tau}$, namely $\rho_\tau$ and $k_\tau$, exhibit high-contrast property (specifically, $\tau$ is sufficiently small). Our work is motivated by the application involving acoustically driven gas bubbles in liquids, where the mass density and bulk modulus inside the gas bubbles display high contrast compared to the surrounding fluids. It is widely known that such high contrasts give rise to a very low frequency (subwavelength) resonance called the Minnaert resonance, present across all configurations of wave speeds. This resonant phenomenon was first observed by Minnaert in \cite{Mi}, while its rigorous mathematical foundation - including the existence and asymptotic behavior for arbitrarily shaped bubbles - was established in \cite{AZ18}. For extensions to configurations involving multiple bubbles, we refer to \cite{ADHO, ACCS,FH}. Owing to their resonant behavior at subwavelength scales, gas bubbles can be used as agents to facilitate imaging \cite{C-06, DGS-21,FSLT, SS-2024} and also serve as fundamental building blocks in the design of dispersive metamaterials \cite{AFGLZ-17, AFZ-17, AZ-17, LPLL, LBFT, MS-241,MS-242} that support wave phenomena such as band-gap formation, negative refraction, and subwavelength focusing. In the previous works \cite{ADHO, ACCS, AZ18, FH}, the Minnaert resonance is shown to be a frequency for which the corresponding system of integral equations is non injective. A natural question was whether this frequency is indeed a mathematical resonance in the sense that it is a pole of the resolvent of its natural Hamiltonian. A first attempt in this direction was proposed in \cite{MPS}, where the authors stated the acoustic propagator as a frequency-dependent Schr\"{o}dinger operator for which they prove that the related resolvents display a drastically different behavior when the frequency is near or away from the Minnaert one. Recently the authors of \cite{LS-04} showed that the Minnaert resonance is indeed one of the scattering resonance of the natural Hamiltonian $H_{\rho_\tau,k_\tau}$ confirming the resonant character of the Minnaert frequency both physically (i.e., as a frequency near which enhancement of the generated fields) and mathematically (as a pole of the related Hamiltonian).
\bigskip

The main contributions of the current work are outlined as follows:
\begin{enumerate}
\item
Our first key result demonstrates that high contrast gives rise not only to the Minnaert resonance but also to a family of scattering resonances of the Hamiltonian $H_{\rho_\tau, k_\tau}$ near the real axis, uniformly across all wave speed configurations. These scattering resonances, apart from the Minnaert ones, can be regarded as Fabry-P\'erot-type resonances. \footnote{This is a natural designation as the generated total fields, in the bubble, are dominated by a linear combination of the Neumann eigenfunctions which hence might create interferences, see formula (\ref{eq:115}) in Theorem \ref{th:5}, a reminiscent property of the classical Fabry-P\'erot resonators modeled by two reflecting mirrors, see \cite{BW-99, r-17, Vaughan-1989} and references therein.} Specifically, we show that, in the high-contrast regime, each Neumann eigenvalue of the Laplacian on $\Omega$, where $\Gamma$ is $C^{1,1}\text{-smooth}$, splits into at most as many scattering resonances as the dimension of its associated Neumann eigenspace (see Theorem \ref{th:1}). Notably, the Minnaert resonance is associated with the first Neumann eigenvalue (which is $0$). Moreover, we derive the asymptotic expansion of all  Fabry-P\'erot resonances, which are the shifts of the original Neumann eigenvaule into the lower complex plane, with imaginary parts asymptotically proportional to $\tau$, that is $\Theta (\tau)$ (see formula \eqref{eq:38} in Theorem \ref{th:2}). See also Figure \ref{fi1} for the distribution of the  Fabry-P\'erot resonances. The aforementioned results are established by combining the Lippmann-Schwinger equation and the variational method, where the former provides the existence, uniqueness and estimates on the number of the  Fabry-P\'erot resonances, while the latter contributes to their asymptotic analysis.
It should be noted that Fabry-P\'erot resonances are present in all wave speed configurations. This does not contradict the findings of \cite{CPV-99}, where a resonance-free beneath the real axis was identified in the case of $c_1 > c_0$; rather, our result suggests that the width of this strip depends on the contrast, and in the high-contrast regime, it is $O(\tau)$. 

\begin{figure}[htbp]
\centering
\includegraphics{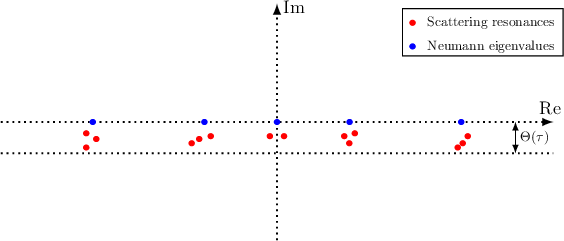}
\caption{Neumann eigenvalues and their associated scattering resonances}\label{fi1}
\end{figure}

\item Our second main result concerns the uniform asymptotics in the whole space for the resolvent estimates of the Hamiltonian $H_{\rho_\tau, k_\tau}$ and for the scattered fields of the associated scattering problem in the time-harmonic regime, emphasizing the contributions from the  Fabry-P\'erot resonances. More precisely, we prove that the resolvent exhibits an enhancement of $\tau^{-1}$ when the real frequency is close to each nonzero Neumann eigenvalue (see Theorem \ref{th:3}). Furthermore, we prove that the scattered field undergoes a transition from an asymptotically trivial (i.e., simpler) limit to a nontrivial one as $\tau \rightarrow 0$ when the incident frequency approaches the nonzero Neumann eigenvalue (see Theorem \ref{th:5}). Here, we consider a single resonator. Handling multiple resonators is interesting with important applications in mind, For this, it should be noted that since the eigenmodes remain approximately constant within the resonators, the calculation of Minnaert resonances for a system of multiple resonators can be performed using the capacitance matrix formalism, see \cite{ADHO, FH}. This provides a discrete approximation of the problem. However, for Fabry-P\'erot resonances, the eigenfunctions exhibit oscillatory behavior within the resonators, see Theorem \ref{th:4} and Theorem \ref{th:6}, making it difficult to extend such a discrete approximation. Such properties are related, in particular, to the high frequency homogenization issues that we intend to consider next.

\item Finally, we consider the microresonator regime, where the size of the resonator $\Omega$ is small while the contrast remains high. In this regime considered here, all  Fabry-P\'erot resonances associated with the nonzero Neumann eigenvalues only contribute at high frequencies, and are therefore referred to as high frequency resonances. Based on scaling transformations and the results previously obtained for fixed-size resonators, we characterize the high-frequency asymptotics of the resolvent and the associated scattered fields corresponding to the microresonator, uniformly throughout the whole space, when both the size of the resonator and the contrast are small (see Theorems \ref{th:6} and \ref{th:4}). The dominating parts, for both the resolvent operator and the scattered field, are characteristic of the highly oscillatory point-scatterer supported at the location of the microresonator, modulated by the far-field pattern of the exterior Dirichlet scattering problem (see formulas \eqref{eq:45} and \eqref{eq:5}). This reveals that, unlike in the Minnaert frequency case - where the asymptotic behavior is isotropic \cite{LS-04,MPS} -each mircoresonator at high frequencies exhibits pronounced anisotropy, see Remark \ref{Anisotropic-small-scatterers} for more details on this notion of anisotropic small scatterers. We believe that such a feature may pave the way for designing anisotropic metamaterials from simple configurations of a single microresonator. Moreover, we derive the time behavior of the acoustic microresonator driven by a causal source with compact support in both time and space (see Theorem \ref{th:7}). Notably, all  Fabry-P\'erot resonances, other than the Minnaert resonances, must tend to infinity as the microresonator's size shrinks to zero. Therefore, for large times that remain within the Minnaert resonance lifetime, the wave evolution in this setting is asymptotically dominated by the Minnaert resonant oscillation, while the amplitude is governed by the associated projection. The precise asymptotic characterization of the Minnaert resonant projection (see Lemma \ref{le:6}) is highly nontrivial due to the fact that the Hamiltonian $H_{\rho_\tau, k_\tau}$ is an unbounded operator. To overcome this difficulty, we show that the Minnaert resonance is simple and make use of the self-adjointness of $H_{\rho_\tau, k_\tau}$ in the space $\mathcal H$. Let us cite the series of works \cite{APL-22,BMV-21, LS-042,MP,Mukh-Si:2023} on wave dynamics of various type of microresonators, where a similar phenomenon-namely, the dominance of subwavelength resonances- was also observed. In comparison, our contribution differs in two key aspects. First, our estimate can be valid for arbitrarily large time scales, specifically of order inverse powers of the microresonator's size, with the exponent depending on the smoothness of the source. This extends beyond the time regimes considered in \cite{LS-042,MP,Mukh-Si:2023}). Second, in contrast to \cite{APL-22,BMV-21}, we analyze the original wave field governed by the wave equation, rather than the truncated inverse Fourier transform of the wave field.
\end{enumerate}

The remaining part of this work is divided as follows. In section \ref{sec:2}, we introduce the basic notations and preliminary results regarding the properties of the scattering resonances. In section \ref{sec:3}, we establish existence and uniqueness of the  Fabry-P\'erot resonances. In section \ref{sec:4} and \ref{sec:5}, we provide a detailed asymptotic analysis for the extended resonator and the microresonator, respectively. Appendix \ref{sec:a} contains auxiliary results and the proofs of certain technical lemmas supporting the main theorems.

\section{Preliminaries} \label{sec:2}
In this section, we begin with basic definitions, followed by the key results on the properties of scattering resonances, both of which will be used throughout the paper.

\subsection{Notations}
Given a domain $\Omega$ whose boundary $\Gamma$ is $C^{1,1}\text{-smooth}$, for each $z\in \CC$, we define the operators
\begin{align*} 
&SL_{\Gamma, z}: H^{-\frac12}(\Gamma) \rightarrow H_{\textrm{loc}}^{2}(\R^3 \backslash \Gamma),  \;\; \left(SL_{\Gamma, z}\phi\right)(x) := \int_{\Gamma} \frac{e^{iz|x-y|}}{4\pi|x-y|}\phi(y)d\sigma(y), \;\; x \in \R^3\backslash \Gamma,\\
&DL_{\Gamma, z}: H^{\frac12}(\Gamma) \rightarrow H_{\textrm{loc}}^{2}(\R^3 \backslash \Gamma),  \;\; \left(SL_{\Gamma, z}\phi\right)(x) := \int_{\Gamma} \partial_{\nu_y}\frac{e^{iz|x-y|}}{4\pi|x-y|}\phi(y)d\sigma(y), \;\; x \in \R^3\backslash \Gamma,\\
&S_{\Gamma, z}: H^{-\frac 12}(\partial  \Omega)\rightarrow H^{\frac12}(\partial  \Omega), \quad \left(S_{\Gamma, z}\phi\right)(x) := \int_{\Gamma} \frac{e^{iz|x-y|}}{4\pi|x-y|}\phi(y)d \sigma(y),\quad x\in \Gamma, \\
&N_{\Omega,z}: L^2(\Omega) \rightarrow H_{\textrm{loc}}^2(\R^3), \quad \left(N_{\Omega, z}\phi\right)(x) := \int_{\Omega}\frac{e^{iz|x-y|}}{4\pi|x-y|}\phi(y)dy, \quad x\in \R^3,\\
&K_{\Gamma,z}^*: H^{-\frac12}(\Gamma) \rightarrow H^{\frac 12}(\Gamma), \quad \left(K_{\Gamma, z}^{*}\phi\right)(x) := \partial_{\nu_x}\int_{\Gamma} \frac{e^{iz|x-y|}}{4\pi|x-y|}\phi(y)d\sigma(y),\quad x\in \Gamma.
\end{align*}
Here, $\nu$ is the outward unit normal to the boundary $\Gamma:= \partial \Omega$. 
For two Banach spaces $X$ and $Y$, denote the space of all linear bounded mapping from $X$ to $Y$ by $\mathcal L(X,Y)$. For simplicity, $\mathcal L(X,X)$ is also denoted by $\mathcal L(X)$. The kernel space and the range of the operator $B \in \mathcal L(X,Y)$ are defined by $\rm{Ker}(B)$ and $\rm{Ran}(B)$, respectively. A linear operator $\mathcal P: X \rightarrow X$ is called a projection if it satisfies: $P^2 = P$, where $\mathcal P^2$ denotes the composition of $P$ with itself. We denote the rank of the projection $P$ by $\rm{dim}\left(\rm{Ran}\left(\mathcal P\right)\right)$. Furthermore, we denote by $\gamma$ the operator that maps a function onto its Dirichlet trace. It is well established that the trace operator $\gamma$ satisfies, {up to a positive bound $C_\Omega$,}
\begin{align*}
\|\gamma \phi\|_{H^{s-\frac 12}(\Gamma)} \le C_{\Omega}\|\phi\|_{H^s(\Omega)}, \quad s>\frac 12.
\end{align*}
Define
\begin{align} 
&\mathcal H_{\textrm{comp}}:= \{\phi\in \mathcal H: \phi|_{\R^3\backslash B_{r_0}} \in  L^2_{\textrm{comp}}(\R^3 \backslash B_{r_0})\}, \label{eq:164}\\
&\textrm{and}\;\mathcal D_{\rm{loc}}(H_{\rho_\tau,k_\tau}):=\notag\\
&\{\phi \in \mathcal H: \phi |_{\R^3\backslash B_{r_0}} \in  L^2_{\textrm{loc}}(\R^3 \backslash B_{r_0})\; \textrm{and}\; \chi \phi \in \mathcal D(H_{\rho_\tau,k_\tau}) \; \textrm{if}\; \chi \in C_c^{\infty}(\R^3)\; \textrm{and}\; \chi|_{B_{r_0}} =1\}. \label{eq:165}
\end{align}
Here $B_{r_0}:=\{x\in \R^3: |x|< r_0\}$ with $r_0>0$ chosen to be large enough such that $\overline{\Omega} \subset B_{r_0}$. 
For $z\in \CC$, define
\begin{align}\label{eq:117}
R_z: L_{\rm{comp}}^2(\R^3) \rightarrow L_{\rm{loc}}^2(\R^3), \;\; \left(R_{z}\phi\right)(x):= \int_{\R^3} \frac{e^{iz|x-y|/c_0}}{4\pi|x-y|} \phi(y) dy, \;\; x\in \R^3,
\end{align}
where $L^2_{\textrm{comp}}(\R^3)$ is defined by
\begin{align*}
L^2_{\textrm{comp}}(\R^3):=\{u\in L^2(\R^3): \exists r>0, |u(x)| = 0\; \textrm{for}\; |x|>r\}.
\end{align*}
Denote $\mathbb L^2(\Omega)$ as $ \left(L^2(\Omega)\right)^3$ with the inner product given by the integral of the dot product of two functions over the domain $\Omega$. Given $\Omega$, for each $z \in \overline{\mathbb C_+}$, let $D_N(z)$ be the exterior Dirichlet to Neumann map at $z$, defined as $\gamma u \rightarrow \partial_\nu u$, where $u$ solves
\begin{align*}
&\Delta u + z^2 u = 0 \quad \textrm{in}\; \R^3 \backslash \Omega, \\
& u\; \textrm{is}\; z-\textrm{outgoing}.
\end{align*}
Here, the condition $z-\rm{outgoing}$ means that there exists $g\in L_{\textrm{comp}}^2(\R^3)$ and $r>0$ such that $u = R_{z} g$ outside $B_r$. 
Furthermore, it is known that $D_N(z)$ can be extended as a meromorphic family of operators mapping from $H^{1/2}(\Gamma)$ to $H^{-1/2}(\Gamma)$ in the whole complex plane. Let $\vep_1,\vep_2$ be two positive constants, we write $\vep_1 = O(\vep_2)$ as $\vep_2\rightarrow 0$, if there exists a constant $C$ such that $\vep_1 <C\vep_2$ for sufficiently small $\vep_2$.
From now on, $\mathbb I$ denotes an identity operator in various spaces, and the constants may be different at different places.

\subsection{Scattering resonances}
In this subsection, we establish that the scattering resonance of the Hamiltonian $H_{\rho_\tau,k_\tau}$ is also a point for which the related system of integral equations is not injective. We recall that the  scattering resonance of the Hamiltonian $H_{\rho_\tau,k_\tau}$ is defined as a pole of the meromorphic extension of $R_{H_{\rho_\tau,k_\tau}}(z)$.
It is worth mentioning that \cite{LS-04} provided another equivalent definition of the scattering resonances of the same type Hamiltonian as $H_{\rho_\tau,k_\tau}$,  as a non-injective point of the related integral equations. Therefore, by using same arguments as employed in the equivalence proof from \cite[section 5]{LS-04}, we immediately obtain the following result.
\begin{lemma}\label{le:5}
A point $z$ is a scattering resonance of the Hamiltonian $H_{\rho_\tau,k_\tau}$ for each fixed $\tau >0$, if and only if it is a point at which the operator $\mathcal A^{H}(z,\tau)$ is not injective. Here, for each $\tau >0$ and $z \in \CC$, the operator $\mathcal A^{H}(z,\tau) \in \mathcal L\left(L^2(\Omega)\times L^{2}(\Gamma)\right)$ is defined by 
\begin{align*}
\mathcal A^{H}(z,\tau):=\begin{bmatrix}
&\mathbb I - \left(\frac{1}{c^2_1}-\frac{1}{c^2_0}\right)z^2~ N_{\Omega, z/c_0} & \left(\frac{\rho_0}{\rho_1\tau}-1\right) SL_{\Gamma, z/c_0}\\
&\left(\frac{1}{c^2_0}-\frac{1}{c^2_1}\right)z^2 \partial_\nu N_{\Omega, z/c_0} &\quad\frac{\rho_0}{\rho_1 \tau}\left(\frac 12\left(1 + \frac {\rho_1\tau}{\rho_0}\right) \mathbb I + \left(1- \frac {\rho_1\tau}{\rho_0}\right)K_{\Gamma, z/c_0}^*\right)
\end{bmatrix},
\end{align*}
where $c_0$ and $c_1$ are specified in \eqref{eq:41}.
\end{lemma}

In order to characterize the above mentioned resonances, with the aid of Theorem \ref{le:5}, 
it suffices to find $(\phi_z, \psi_z) \in L^2(\Omega) \times L^2(\Gamma)$ such that 
\begin{align}
\mathcal A^{H} (z,\tau)\begin{bmatrix}
 \phi_z \\
\psi_z 
\end{bmatrix} = \begin{bmatrix}
 0\\
 0
\end{bmatrix}. \notag
\end{align}
Clearly, the above equation is equivalent to the subsequent equation 
\begin{align}
\mathcal A(z,\tau)
\begin{bmatrix}
\phi_z\\
{\rho_0}{\tau^{-1}}/\rho_1 \psi_z 
\end{bmatrix} 
= \begin{bmatrix}
 0\\
 0
\end{bmatrix}. \notag
\end{align}
Here, $\mathcal A(z,\tau) \in \mathcal L\left(L^2(\Omega)\times L^{2}(\Gamma)\right)$ is defined by 
\begin{align}\label{eq:54}
\mathcal A(z,\tau):=
\begin{bmatrix}
&\mathbb I - \left(\frac{1}{c^2_1}-\frac{1}{c^2_0}\right)z^2~ N_{\Omega, z/c_0} & \left(1-\frac{\rho_1\tau}{\rho_0}\right) SL_{\Gamma, z/c_0}\\
&\left(\frac{1}{c^2_0}-\frac{1}{c^2_1}\right)z^2 \partial_\nu N_{\Omega, z/c_0} &\quad \frac 12\left(1 + \frac {\rho_1\tau}{\rho_0}\right) \mathbb I + \left(1-\frac {\rho_1\tau}{\rho_0}\right)K_{\Gamma, z/c_0}^*
\end{bmatrix}.
\end{align}

Then, as a by-product of Theorem \ref{le:5}, we obtain the following corollary.

\begin{corollary} \label{co:1}
The scattering resonance of the Hamiltonian $H_{\rho_\tau,k_\tau}$ for each fixed $\tau >0$ if and only if it is a point at which the operator $\mathcal A(z,\tau)$, specified in \eqref{eq:54}, is not injective. 
\end{corollary}

\begin{remark} 
If $z$ is not scattering resonance of the Hamiltonian $H_{\rho_\tau,k_\tau}$, for each $f\in L^2_{\textrm{comp}}(\R^3)$, we have
\begin{align} 
R_{H_{\rho_\tau, k_\tau}}(z) f = & v^f_z(x) +
\begin{bmatrix}
\left(\frac{1}{c^2_0}-\frac{1}{c^2_1}\right)z^2~ N_{\Omega, z/c_0} & \left(1-\frac{\rho_1 \tau}{\rho_0}\right)SL_{\Gamma, z/c_0}\\
\left(\frac{1}{c^2_0}-\frac{1}{c^2_1}\right)z^2 \partial_\nu N_{\Omega, z/c_0} & \left(1-\frac{\rho_1 \tau}{\rho_0}\right)K_{\Gamma, z/c_0}^*
\end{bmatrix} &\mathcal A^{-1}(z,\tau)
\begin{pmatrix}
& g_z^{f}\\
& \partial_\nu g^f_z
\end{pmatrix} \notag\\
& + \left({c^{-2}_1}-{c^{-2}_0}\right) N_{\Omega, z/c_0} f. \label{eq:62}
\end{align}
Here, $v^f_z$ is denoted as
\begin{align} \label{eq:81}
v^f_{z}:= -\frac{1}{c_0^2} R_{z/c_0} f
\end{align}
with the operator $R_{z/c_0}$ defined by \eqref{eq:117}
and $g_z^f:= v_z^{f} + \left({c^{-2}_1}-{c^{-2}_0}\right) N_{\Omega, z/c_0} f$. We note that the inverse of $\mathcal A(z,\tau)$ exists by Corollary \ref{co:1}. Formula \eqref{eq:62} is easily derived from the subsequent Lippmann-Schwinger equation
\begin{align}
& u^f_{z,\tau}(x)= v^f_z(x) + \left(\frac{1}{c^2_1} - \frac{1}{c^2_0} \right)z^2\int_{\Omega} \frac{e^{i{z}|x-y|/c_0}}{4\pi|x-y|}u^f_{z,\tau}(y)dy + \left(\frac{1}{c^2_1}-\frac{1}{c^2_0}\right)\int_{\Omega}\frac{e^{i{z}|x-y|/c_0}}{4\pi|x-y|}f(y)dy \notag\\
&\quad\quad \quad \;-\left(\frac{\rho_0}{\rho_1\tau}-1\right)\int_{\Gamma}\frac{e^{i{z}{|x-y|}/c_0}}{4\pi|x-y|} \partial_\nu u^f_{z,\tau}(y)d\sigma(y), \quad x\in \R^3 \backslash \Gamma, \label{eq:75}
 \end{align}
where  
\begin{align} \label{eq:72}
u^f_{z,\tau}:= R_{H_{\rho_\tau, k_\tau}}(z) f,
\end{align}
with its value within $\Omega$ and its interior normal derivative $\partial_\nu u^f_{z, \tau}$ on $\Gamma$ solving
\begin{align}
\mathcal A(z,\tau) 
\begin{pmatrix}
& u^f_{z,\tau} \\
&\frac{\rho_0}{\rho_1 \tau}\partial_\nu u^f_{z,\tau}
\end{pmatrix}
=
\begin{pmatrix}
& g_z^{f}\\
& \partial_\nu g^f_z
\end{pmatrix}. \notag
\end{align}
The above representation of the resolvent, based on the Lippmann-Schwinger equation \eqref{eq:75}, is instrumental in deriving its asymptotics in section \ref{sec:4}.
\end{remark}

It is well established that $H_{\rho_\tau,k_\tau}$ admits two Minnaert resonances, whose asymptotics are presented as follows.

\begin{lemma}[Statement (b) of Lemma 5.1 in \cite{LS-04}] \label{le:10}
There exhibit two scattering resonances $z_{\pm}(\tau)$ of of $H_{\rho_\tau, k_\tau}$, denoted by $z_+(\tau)$ and $z_-(\tau)$ around $0$. Furthermore, these two resonances satisfy the following asymptotic expansion 
\begin{align}\label{eq:65}
z_{\pm}(\tau) = \pm \omega_M -i\frac{\omega^2_M \mathcal C_\Omega }{8\pi c_0} + O(\tau^{3/2}), \quad \textrm{as}\; \tau \rightarrow 0,
\end{align}
where $c_0:=\sqrt{k_0/\rho_0}$, and
\begin{align} \label{eq:70}
\omega_M:= \sqrt{\frac {\mathcal C_\Omega k_1}{|\Omega|\rho_0}}\sqrt \tau
\end{align}
denotes the related Minnaert frequency generated by the micro-bubble and $\mathcal C_\Omega$, defined by 
\begin{align}
\mathcal C_\Omega:= \int_\Gamma \left(S^{-1}_{\Gamma, 0}1\right)(x) d\sigma(x), \notag
\end{align}
represents the capacitance of $\Omega$.
\end{lemma}

We proceed to introduce the resonant projection $\Pi_{H_{\rho_\tau,k_\tau}}(z) \subset \mathcal L \left(\mathcal H_{\textrm{comp}}, \mathcal D_{\textrm{loc}}(H_{\rho_\tau,k_\tau})\right)$, defined as:
\begin{align} \label{eq:64}
\Pi_{H_{\rho_\tau,k_\tau}}(z) = -\frac{1}{2\pi i}\int_{B(z)} R_{H_{\rho_\tau, k_\tau}}(\lambda) 2\lambda d\lambda.
\end{align}
Here, $B(z)$ is a circle in the complex plane, which contains $z$ while excludes other resonances.
The following lemma establishes the connection between the resonant projection and the resolvent $R_{H_{\rho_\tau,k_\tau}}$.

\begin{lemma}[{\cite[Theorem 4.7]{DM}}] \label{le:3}
Let $\tau > 0$ be fixed. Assume that $\lambda_\tau$ is any resonance of the Hamiltonian $H_{\rho_\tau,k_\tau}$. There exists a positive integer $M_{\lambda_\tau}\le m_R(\lambda_\tau)$ such that the resolvent $R_{H_{\rho_\tau,k_\tau}}$ has the expansion near $\lambda_\tau$
\begin{align*}
R_{H_{\rho_\tau,k_\tau}}(z) = - \sum^{M_{\lambda_\tau}}_{l=1} \frac{(H_{\rho_\tau, k_{\tau}}-\lambda_\tau^2)^{l-1}}{\left(z^2-\lambda_\tau^2\right)^{l}} \Pi_{H_{\rho_\tau,k_\tau}}({\lambda_\tau})  + B(z, \lambda_\tau),
\end{align*}
where $z \rightarrow B(z,\lambda_\tau)$ is holomorphic near $\lambda_{\tau}$ and  $\Pi_{{H_{\rho_\tau,k_\tau}}}({\lambda_\tau})$ satisfies
\begin{align*}
\left(H_{\rho_\tau, k_\tau} - \lambda_\tau^2 \right)^{M_{\lambda_\tau}} \Pi_{{H_{\rho_\tau,k_\tau}}}({\lambda_\tau}) = 0.
\end{align*}
Here, $\Pi_{{H_{\rho_\tau,k_\tau}}}({\lambda_\tau})$ is given by \eqref{eq:64}, and 
\begin{align}\label{eq:118}
m_R(\lambda_\tau) = {\rm{dim}}\left({\rm{Ran}}\left(\Pi_{{H_{\rho_\tau,k_\tau}}}({\lambda_\tau})\right)\right). 
\end{align}
\end{lemma}

We conclude this subsection by introducing the definition of resonant states and generalized resonant states. 

\begin{definition}
Let $\tau>0$ be fixed. Assume that $\lambda_\tau$ is a resonance of the Hamiltonian $H_{\rho_\tau,k_\tau}$.
We call an element $u_{\lambda_\tau}$ is a resonant state corresponding to the resonance $\lambda_\tau \ne 0$ if $u_{\lambda_\tau} \in \Pi_{{H_{\rho_\tau,k_\tau}}}({\lambda_\tau})\mathcal H_{\rm{comp}}$ and $H_{\rho_\tau, k_\tau} u_{\lambda_\tau} - \lambda_\tau^2 u_{\lambda_\tau} = 0$. Furthermore, we call an element $v_{\lambda_\tau}$ is a generalized resonant state corresponding to the resonance $\lambda_\tau \ne 0$ if $v_{\lambda_\tau} \in \Pi_{{H_{\rho_\tau,k_\tau}}}({\lambda_\tau}) \mathcal H_{\rm{comp}}$ but $v_{\lambda_\tau}$ is not a resonant state.
\end{definition}

\begin{remark}[Characterization of the resonant states] \label{re:0}
Let $\lambda_\tau$ with $\tau>0$ be any resonance of the Hamiltonian $H_{\rho_\tau,k_\tau}$. With the aid of \cite[Theorem 4.9]{DM}, we obtain that  $u_{\lambda_\tau} \in {\rm{Ran}}\left(\Pi_{H_{\rho_\tau,k_\tau}}({\lambda_\tau})\right)$ is a resonant state if and only if 
\begin{align*}
H_{\rho_\tau, k_\tau} u_{\lambda_\tau} - \lambda_\tau^2 u_{\lambda_\tau} = 0
\end{align*}
and there exist $g_{\lambda_\tau}\in L^2_{{\rm{comp}}}(\R^3)$ and $r > 0$ such that 
\begin{align*}
u_{\lambda_\tau}|_{\R^3\backslash B_{r}} = R_{\lambda_\tau/c_0} g_{\lambda_\tau}|_{\R^3\backslash B_{r}},
\end{align*}
i.e.,
\begin{align} 
&\Delta u_{\lambda_\tau} + \lambda_\tau^2c_0^{-2} u_{\lambda_\tau} = 0 \quad {\rm{in}}\; \R^3 \backslash \Omega, \label{eq:141}\\
&\Delta u_{\lambda_\tau} + \lambda_\tau^2 c_1^{-2} u_{\lambda_\tau} = 0 \quad {\rm{in}}\; \Omega,\\
& \gamma_+ u_{\lambda_\tau} = \gamma_- u_{\lambda_\tau}, \quad \partial^+_{\nu} u_{\lambda_\tau} = \frac{\rho_0}{\rho_1\tau}\partial^-_{\nu} u_{\lambda_\tau} \quad {\rm{on}}\; \Gamma, \\
& u_{\lambda_\tau}\; {\rm{is}}\; \lambda_\tau/c_0-{\rm{outgoing}}. \label{eq:142}
\end{align}
Here, $\gamma_\pm$ and $\partial_\nu^\pm$ denote the Dirichlet trace and normal derivative operators taken from the exterior and interior of $\Omega$, respectively.
Moreover, when $u_{\lambda_\tau}$ is a resonant state, utilizing \eqref{eq:141}--\eqref{eq:142} and Green formulas, we arrive at 
\begin{align}
& u_{\lambda_\tau}(x) = \left(\frac{1}{c^2_1} - \frac{1}{c^2_0}\right)\left(\lambda_\tau\right)^2 \left(N_{\Omega, \lambda_\tau /c_0} u_{\lambda_\tau}\right)(x) \notag \\
&\quad\quad \quad -\left(\frac{\rho_0}{\rho_1\tau}-1\right)\left(SL_{\Gamma,\lambda_\tau/c_0}\partial_\nu u_{\lambda_\tau}\right)(x), \quad x\in \R^3 \backslash \Gamma.  \label{eq:77}
 \end{align}
Thus, the value of $u_{\lambda_\tau}$ within $\Omega$ and the multiplication of the constant $\rho_0/(\rho_1\tau)$ and its normal derivative on $\Gamma$ constitute a nonzero solution of $\mathcal A(\lambda_\tau,\tau)$.
\end{remark}

\section{Existence and uniqueness of  Fabry-P\'erot resonances} \label{sec:3}

This section is devoted to investigating the existence and uniqueness of the  Fabry-P\'erot resonances. With the aid of Corollary \ref{co:1}, it is sufficient to study the non-injective points of $\mathcal A(z,\tau)$, specified in \eqref{eq:54}. For this aim, we observe that 
\begin{align}
&\mathcal A(z,\tau) = \begin{bmatrix}
&\mathbb I - \left(\frac{1}{c^2_1}-\frac{1}{c^2_0}\right)z^2~ N_{\Omega, z/c_0} & SL_{\Gamma, z/c_0} \notag \\
&\left(\frac{1}{c^2_0}-\frac{1}{c^2_1}\right)z^2 \partial_\nu N_{\Omega, z/c_0} &\quad \frac{\mathbb I}2 + K_{\Gamma, z/c_0}^*
\end{bmatrix}  + \frac{\rho_1\tau}{\rho_0}\begin{bmatrix}
&0 & -SL_{\Gamma, z/c_0}\\
&0 &\quad \frac 12 \mathbb I - K_{\Gamma, z/c_0}^*
\end{bmatrix} \notag \\
&\qquad\quad\;=:\mathcal A_0(z) + \frac{\rho_1\tau}{\rho_0}\mathcal A_1(z). \label{eq:40}
\end{align}

The following result demonstrates the injective properties of $\mathcal A_0(z)$ and $\mathcal A(z,\tau)$.

\begin{lemma} \label{le:0}
Let $\tau>0$ and $z\in \CC$. Assume that the operators $\mathcal A_0(z)$ and $\mathcal A(z,\tau)$ are given by \eqref{eq:40} and \eqref{eq:54}, respectively. The following arguments hold true. 
\begin{enumerate}[(a)]
\item \label{f1} 
The operator $\mathcal A_0(z)$ is not injective at $z_0 \in \R$ if and only if $z^2_0/c^2_1$ is a Neumann eigenvalue of the domain $\Omega$.

\item \label{f2}
Assume that $z_0 \in \mathbb R$ such that $z^2_0/c^2_1$ is a Neumann eigenvalue of the domain $\Omega$. There exists $\delta_{z_0}> 0$ such that when $\tau \in (0,\delta_{z_0})$, we can identify a finite set of points, whose the total number is independent of $\tau$, in the neighborhood of $z_0$ such that $\mathcal A(z,\tau)$ fails to be injective. Furthermore, these points all tend to $z_0$ when $\tau \rightarrow 0$.

\item \label{f3}
Let $I \subset \R$ be a bounded interval. There exists a constant $\delta_I> 0$ independent of $\tau$, and a constant $\tau_{\delta_I} > 0$ dependent on $\delta_I$, such that when $\tau \in (0, \tau_{\delta_I})$, $\mathcal A(z,\tau)$ is invertible in $ \mathcal I_I(\delta_I)$, which is defined by
\begin{align}
\mathcal I_I(\delta_I):= \left\{z\in \mathbb C: {\rm{Re}}(z) \in I,\; |{\rm{Im}}(z)| <\delta_I\right\} \backslash \bigcup_{z_0\in I \cap \Lambda_N} B_{\delta_I}(z_0). \label{eq:74}
\end{align}
Here, 
\begin{align} \label{eq:82} 
\Lambda_{N}:=\{z_0\in \mathbb R: z^2_0/c^2_1\; \textrm{is a Neumann eigenvalue of the domain}\; \Omega\}.
\end{align}
\end{enumerate}

\end{lemma}

\begin{proof}
\eqref{f1} Assume that $z_0 \in \R$ is a point where $A_0(z_0)$ fails to be injective. Then, there exists $(\phi_{z_0}, \psi_{z_0}) \in L^2(\Omega) \times L^2(\Gamma)$ such that 
\begin{align} \label{eq:3}
\mathcal A_0(z_0) \begin{bmatrix}
\phi_{z_0}\\
\psi_{z_0} 
\end{bmatrix} = \begin{bmatrix}
0\\
0
\end{bmatrix},
\end{align}
which implies 
\begin{align*}
&\left(\frac{1}{c_0^{2}} - \frac{1}{c_1^{2}}\right) z_0^2 N_{\Omega,z_0/c_0} \phi_{z_0}+ SL_{\Gamma, z_0/c_0} \psi_{z_0} = -\phi_{z_0} \quad \textrm{in}\; \Omega,\\
& \partial_\nu \left(\left({c_0^{-2}} - {c_1^{-2}}\right) z_0^2 N_{\Omega,z_0/c_0} \phi_{z_0}+ SL_{\Gamma, z_0/c_0} \psi_{z_0} \right) = 0  \quad \textrm{on}\; \Gamma.
\end{align*}
Therefore, we easily obtain
\begin{align}
&\Delta \phi_{z_0} + z_0^2/c^2_0\phi_{z_0} = \left(\frac{1}{c_0^{2}} - \frac{1}{c_1^{2}}\right) z_0^2\phi_{z_0} \quad \textrm{in}\; \Omega, \label{eq:67}\\
& \partial_\nu \phi_{z_0} = 0 \quad \textrm{on}\; \Gamma. \label{eq:68}
\end{align}
This means that $z^2_0/c^2_1$ is a Neumann eigenvalue of the domain $\Omega$.

Conversely, given any $z_0 \in \R$ such that $z^2_0/c^2_1$ is a Neumann eigenvalue of the domain $\Omega$, we proceed to construct a nonzero solution such that \eqref{eq:3} holds. Since \eqref{eq:3} is equivalent to the following equation
\begin{align}
\mathcal A_0(-z_0) \begin{bmatrix}
\overline{\phi_{z_0}}\\
\overline{\psi_{z_0}}
\end{bmatrix} = \begin{bmatrix}
0\\
0
\end{bmatrix}, \notag
\end{align}
it sufficient to focus on the case of $z_0 \in \R_+$. To do so, we aim to construct $\psi_{z_0} \in H^{-1/2}(\Gamma)$ satisfying
\begin{align} 
SL_{\Gamma,z_0/c_0}\psi_{z_0} + (\phi_{z_0} + \left({c_0^{-2}} - {c_1^{-2}}\right) z_0^2 N_{\Omega,z_0/c_0}\phi_{z_0}) = 0 \quad \textrm{in}\; \Omega. \label{eq:11}
\end{align}
Here, $\phi_{z_0}$ is any nonzero Neumann eigenfunction corresponding to $z^2_0/c^2_1$. By the regularity of the boundary $\Gamma$, it is known that $\phi_{z_0} \in H^2(\Omega)$. Once \eqref{eq:11} is verified, we have that 
\begin{align}\label{eq:6}
\left(\frac{1}2 + K^*_{\Gamma,z_0/c_0}\right)\psi_{z_0} = \partial_\nu (-\left({c_0^{-2}} - {c_1^{-2}}\right) z_0^2 N_{\Omega,z_0/c_0}\phi_{z_0}) \quad \textrm{on}\; \Gamma
\end{align}
and that
\begin{align} \label{eq:59}
\left(-\frac{1}2 + K^*_{\Gamma,z_0/c_0}\right)\psi_{z_0} = D_N(z_0/c_0)(-\gamma\left(\phi_{z_0}\right) - \left({c_0^{-2}} - {c_1^{-2}}\right) z_0^2 \partial_\nu N_{\Omega,z_0/c_0}\phi_{z_0}) \quad \textrm{on}\; \Gamma,
\end{align}
due to the well-posedness of the exterior Dirichlet obstacle scattering problem and the jump relations of $S_{\Gamma,\lambda}$. Furthermore, subtracting \eqref{eq:6} and \eqref{eq:59}, we have 
\begin{align} \label{eq:84}
\psi_{z_0} = D_N(z_0/c_0) \gamma\left(\phi_{z_0}\right).
\end{align}
In conjunction with $\phi_{z_0} \in H^{2}(\Omega)$, we obtain $\psi_{z_0} \in L^2(\Gamma)$. Therefore, building upon \eqref{eq:11} and \eqref{eq:6}, $(\phi_{z_0}, \psi_{z_0}) \in L^2(\Omega) \times L^2(\Gamma)$ solves \eqref{eq:3}. In order to verify \eqref{eq:11}, we distinguish between the following two cases.

\textbf{Case 1}: 
If $z^2_0/c^2_0$ is not a Dirichlet eigenvalue of $\Omega$, by using the well-known fact that $S_{\Gamma,z_0/c_0}$ is invertible from $H^{-1/2}(\Gamma)$ to $H^{1/2}(\Gamma)$, we can find a unique $\psi_{z_0} \in H^{-1/2}(\Gamma)$ such that 
\begin{align}
S_{\Gamma,z_0/c_0}\psi_{z_0} = \gamma(-\phi_{z_0} - \left({c_0^{-2}} - {c_1^{-2}}\right) z_0^2 N_{\Omega,z_0/c_0}\phi_{z_0}) \quad \textrm{on}\; \Gamma, \notag
\end{align}
which directly yields \eqref{eq:11}.

\textbf{Case 2}:
Now we focus on the case when $z^2_0/c^2_0$ is a Dirichlet eigenvalue of $\Omega$.

First, we recall that $\lambda^2$ is a Dirichlet eigenvalue of $\Omega$ if there exists a nonzero solution $w^D \in H^1(\Omega)$ satisfying
\begin{align*}
&\Delta w^{D} + \lambda^2 w^{D} = 0 \quad \textrm{in}\; \Omega,\\
& w^{D} = 0 \quad \textrm{on}\; \Gamma.
\end{align*}
By Green formulas, we easily obtain
\begin{align} \label{eq:14}
w^{D}(x) = \int_{\Gamma}\frac{e^{i\lambda|x-y|}}{4\pi|x-y|}\partial_{\nu(y)} w^{D}(y) d\sigma(y) \quad x\in \Omega.
\end{align}
This means that each Dirichlet eigenfunction in $\Omega$ can be represented by corresponding single-layer potential. Let $\left\{w_1^{D},\ldots, w_N^{D}\right\}$ denote the orthogonal basis of the Dirichlet eigenfunction space of dimension $N \in \mathbb N$ on $\Omega$ corresponding to the Dirichlet eigenvalue $\lambda^2$. 

Second, we assert that the kernel of $S_{\Gamma,\lambda}$ can be characterized by
\begin{align} \label{eq:56}
{\rm{Ker}}\left(S_{\Gamma,\lambda}\right) = \left\{ h \in H^{-\frac{1}2}(\Gamma) : h =\sum_{l=1}^{N}h_l\partial_\nu w^D_l\;\textrm{with}\; h_l \in \mathbb C\right\}.
\end{align}
To do so, it suffices to prove that for any $\psi \in \textrm{Ker}\left(S_{\Gamma, \lambda}\right)$, 
\begin{align} \label{eq:39} 
\psi(x) = \partial_\nu g_\psi(x), \quad \textrm{for}\; x\in \Gamma,
\end{align}
where $g_\psi(x): = \left(SL_{\Gamma,\lambda} \psi\right)(x)$ for $x \in \Omega$. We note that $g_\psi \in \textrm{Span}\left\{w_1^{D},\ldots, w_N^{D}\right\}$.
By Green's integral theorems, we easily obtain
\begin{align} 
SL_{\Gamma,\lambda} \partial_\nu g_\psi = g_\psi \quad &\rm{in}\; \Omega, \label{eq:48}\\
S_{\Gamma,\lambda} \partial_\nu g_\psi = 0 \quad &\rm{on}\; \Gamma. \label{eq:50}
\end{align}
We define
\begin{align*}
W_\psi(x):=\left(SL_{\Gamma,\lambda} \partial_\nu g_\psi\right)(x) - \left(SL_{\Gamma,\lambda}\psi\right)(x) \quad \mathrm{for}\; x\in \R^3 \backslash \Gamma.
\end{align*}
It immediately follows from \eqref{eq:48}--\eqref{eq:50} that $W_\psi = 0 $ in $\overline\Omega$. Since $W_\psi$ solves Helmholtz equation with the Sommerfeld radiation condition, we also have $W_\psi =0 $ in $\R^3 \backslash \overline{\Omega}$.
Therefore, by jump relations of single layer potentials, we obtain 
\begin{align*}
\partial^+_\nu W_\psi - \partial^-_\nu W_\psi = \psi - \partial_\nu g_\psi = 0 \quad \rm{on}\; \Gamma,
\end{align*}
whence \eqref{eq:39} follows. 

Third, we prove that 
\begin{align}\label{eq:58}
S_{\Gamma,\lambda}: H^{-1/2}(\Gamma) \backslash \rm{Ker}\left(S_{\Gamma,\lambda}\right)\rightarrow  \rm{Ker}^{\perp}\left(S_{\Gamma,\lambda}\right)\;\;\textrm{is invertible}. 
\end{align}
Here, we note that the space $\rm{Ker}^{\perp}\left(S_{\Gamma,\lambda}\right)$ can be represented by 
\begin{align*}
\textrm{Ker}^{\perp}\left(S_{\Gamma,\lambda}\right):= \left\{\psi \in H^{\frac 12}(\Gamma): \int_{\Gamma} \psi(x) \partial_\nu w^D_l(x) d\sigma(x) = 0,\; l=1,\ldots,N\right\}.
\end{align*}
It easily follows from the definition of $\rm{Ker}\left(S_{\Gamma,\lambda}\right)$ that
\begin{align}
S_{\Gamma,\lambda}: H^{-1/2}(\Gamma) \backslash \textrm{Ker}\left(S_{\Gamma,\lambda}\right) \rightarrow \textrm{Ran}\left(S_{\Gamma,\lambda}\right)\;\; \textrm{is bijective} \notag.
\end{align}
Therefore, it remains to derive
\begin{align}\label{eq:53}
\left(\textrm{Ker}\left(S_{\Gamma,\lambda}\right)\right)^{\perp} = \textrm{Ran}\left(S_{\Gamma,\lambda}\right).
\end{align}
A straightforward calculation gives 
\begin{align*}
\langle S_{\Gamma,-\lambda} \psi_1, \psi_2 \rangle_{L^2(\Gamma)} = \langle \psi_1, S_{\Gamma,\lambda}\psi_2 \rangle_{L^2(\Gamma)}, \quad \psi_1, \psi_2 \in H^{-\frac12}(\Gamma).
\end{align*}
From this, we arrive at 
\begin{align}\label{eq:57}
\textrm{Ker}\left(S_{\Gamma,-\lambda}\right) =  \left(\textrm{Ran}\left(S_{\Gamma,\lambda}\right)\right)^{\perp}.
\end{align}
Furthermore, we note that $S_{\Gamma,\lambda}$ is a Fredholm operator with index zero mapping from $H^{-1/2}(\Gamma)$ to $H^{1/2}(\Gamma)$ due to the decomposition 
\begin{align*}
S_{\Gamma,\lambda} = S_{\Gamma,0} + (S_{\Gamma,\lambda}-S_{\Gamma,0}).
\end{align*}
Here, $S_{\Gamma,0}$ is invertible from $H^{-1/2}(\Gamma)$ to $H^{1/2}(\Gamma)$ and $S_{\Gamma,\lambda}-S_{\Gamma,0}$ is a compact operator mapping from $H^{-1/2}(\Gamma)$ to $H^{1/2}(\Gamma)$ due to the smoothness of its kernel.
The Fredholm properties of $S_{\Gamma,\lambda}$ yields that the codimension of the rank of $S_{\Gamma,\lambda}$ is finite, whence the closeness of the range of $S_{\Gamma,\lambda}$ follows (see \cite[Corollary 2.17]{AA-02}). 
Utilizing the identity $\rm{Ker}\left(S_{\Gamma,-\lambda}\right) = \rm{Ker}\left(S_{\Gamma,\lambda}\right)$, and adopting the closeness of $\rm{Ran}\left(S_{\Gamma,\lambda}\right)$ and identity \eqref{eq:57}, we obtain \eqref{eq:53}.
This concludes the verification of assertion \eqref{eq:58}.

Finally, we construct $ \psi_{z_0} \in H^{-1/2}(\Gamma)$ such that \eqref{eq:11} holds. With the aid of Green's integral theorems, for each $g\in H^1(\Omega)$ solving the Helmholtz equation with the wave number $\lambda$, we have
\begin{align}
\left\langle \gamma(g), \partial_\nu w_l^D\right\rangle_{L^2(\Gamma)} = \int_\Gamma g(x) \partial_\nu w^D_l(x) - w^D_l(x) \partial_\nu g(x) d \sigma(x) = 0, \quad l=1,\ldots,N. \notag
\end{align}
In conjunction with \eqref{eq:56}, we obtain $\gamma(g) \in \left(\rm{Ker}\left(S_{\Gamma,\lambda}\right)\right)^{\perp}$. Then, by using \eqref{eq:58}, we can find a $\psi_{g} \in H^{-1/2}(\Gamma)$ such that $S_{\Gamma,\lambda} \psi_{g} = \gamma g$. 
Hence, the subtraction of $SL_{\Gamma,\lambda} \psi_{g}-g$ is a Dirichlet eigenfunction on $\Omega$ corresponding to the Dirichlet eigenvalue $\lambda$. Recall that the set $\left\{w_1^{D},\ldots, w_N^{D}\right\}$ forms the orthogonal basis of the Dirichlet eigenfunction space on $\Omega$ corresponding to $\lambda$. We arrive at 
\begin{align} \label{eq:4}
SL_{\Gamma,\lambda} \psi_{g} - g = \sum^N_{l=1}\left\langle SL_{\Gamma,\lambda} \psi_{g}-g, w^D_l \right\rangle_{L^2(\Omega)} w^D_l.
\end{align}
Combining \eqref{eq:14} and \eqref{eq:4} yields
\begin{align}\label{eq:8}
g = SL_{\Gamma,\lambda} \psi_{g} +  \sum^N_{l=1}\left\langle  SL_{\Gamma,\lambda} \psi_{g}-g, w^D_l \right\rangle_{L^2(\Omega)} SL_{\Gamma,\lambda} \partial_\nu w^D_l\quad\rm{in}\; \Omega.
\end{align}
Since
\begin{align*}
\left(\Delta + z^2_0/c^2_0\right)\left[ -\phi_{z_0} - \left({c_0^{-2}} - {c_1^{-2}}\right) z_0^2 N_{\Omega,z_0/c_0}\phi_{z_0} \right] = 0 \quad \mathrm{in}\; \Omega,
\end{align*}
setting 
\begin{align*}
\lambda = z_0/c_0\; \textrm{and}\; g = -\phi_{z_0} - \left({c_0^{-2}} - {c_1^{-2}}\right) z_0^2 N_{\Omega,z_0/c_0}\phi_{z_0}\; \textrm{in}\; \eqref{eq:8},
\end{align*}
then, we see that $\psi_{z_0}:=\psi_g + \sum^N_{l=1}\left\langle  SL_{\Gamma,\lambda} \psi_{g}-g, w^D_l \right\rangle_{L^2(\Omega)} \partial_\nu w^D_l$ solves \eqref{eq:11}.

\eqref{f2} We observe from \eqref{eq:40} that the matrix $\mathcal A(z,\tau)$ is a small perturbation of $\mathcal A_0(z)$ for each fixed point $z$. Using Gohberg and Sigal theory, we assert that the injectivity properties of $\mathcal A(z_0,\tau)$ can be determined by $\mathcal A_0(z_0)$ in vicinity of $z_0$. Below, we outline the proof of this statement.
\begin{enumerate}[(1)]
\item \label{b1} First, it can be seen that 
$\mathcal A_0(z)$ and $\mathcal A_1(z)$ are both analytic with respect to $z\in \mathbb C$.

\item \label{b2} Second, we claim that there exists $\eta$ such that $\mathcal A_0(z)$ is normal with respect to $\partial B_{\CC,\eta}(z_0)$ and its inverse exists when $z \in B_{\CC,\eta}(z_0) \backslash \{z_0\}$, where $B_{\CC,s}(z_0):=\{\lambda\in \CC: |\lambda-z_0|<s\}$ for any $s\in \R_+$.

\item \label{b3} Building upon the previous two steps, we can find $\delta_{z_0}> 0$ such that when $\tau \in (0,\delta_{z_0})$ 
\begin{align*}
\frac{\rho_1 \tau }{\rho_0}\|\mathcal A^{-1}_0(z) \mathcal A_1(z)\|_{L^2(\Omega) \times L^2(\Gamma)} < 1 \quad \rm{on}\; \partial  B_{\mathbb C, \eta}(z_0).
\end{align*}
Therefore, setting $B_1 = \mathcal A_0 $, $B_2 = \rho_1 \tau^2/\rho_0 \mathcal A_1$, $V  = B_{\mathbb C, \eta}(z_0)$ in Theorem \ref{th:a2}, we derive the assertion of this statement. 
\end{enumerate}

Finally, we prove \eqref{b2} in the outline of the proof mentioned above. By a straightforward calculation, we directly obtain
\begin{align} \label{eq:55}
\mathcal A_0(z) = \begin{bmatrix}
&\mathbb I &  0\\
& 0 &\frac{\mathbb I}2 
\end{bmatrix} 
+ \begin{bmatrix}
\left(\frac{1}{c^2_0}-\frac{1}{c^2_1}\right)z^2~ N_{\Omega, z/c_0} & SL_{\Gamma, z/c_0}\\
\left(\frac{1}{c^2_0}-\frac{1}{c^2_1}\right)z^2 \partial_\nu N_{\Omega, z/c_0} & K_{\Gamma, z/c_0}^*
\end{bmatrix}, \quad z\in \mathbb C.
\end{align}
It can be easily deduced from \eqref{eq:55} that $\mathcal A_0(z)$ is a Fredholm operator with index $0$. Therefore, in order to investigate if the inverse of $\mathcal A_0(z)$ exists, it suffices to consider its injectivity. When $z\in \CC$ is a point where $\mathcal A_0(z)$ fails to be injective, proceeding as in the derivation of equations \eqref{eq:67}--\eqref{eq:68}, we readily obtain that $z^2/c^2_1$ is a Neumann eigenvalue of the domain $\Omega$. 
In conjunct with the fact that the nonzero Neumann eigenvalues of the domain $\Omega$ are all positive, we can find $\eta$ such that $\mathcal A_0(z)$ is not injective at $z \in B_{\CC,\eta}(z_0) \backslash \{z_0\}$. This, together with the Fredholm property of $\mathcal A_0(z)$ yields the invertibility of $\mathcal A_0(z)$ in $B_{\CC,\eta}(z_0) \backslash \{z_0\}$. Hence, the assertion \eqref{b2} is derived.

\eqref{f3} Recall that each point $z\in \mathbb C$ where $\mathcal A_0(z)$ is not injective is a point at which $z^2/c^2_1$ is a Neumann eigenvalue of the domain $\Omega$. Then, we can find $\delta_I$ such that $\mathcal A_0(z)$ is bijective in $\mathcal I_I(\delta_I)$. From this, it can be seen that there exists $\tau_{\delta_I}$ such that when $\tau \in (0,\tau_{\delta_I})$, 
\begin{align*}
\frac{\rho_1 \tau }{\rho_0}\|\mathcal A^{-1}_0(z) \mathcal A_1(z)\|_{L^2(\Omega) \times L^2(\Gamma)} < 1\;  \textrm{on}\; \partial \mathcal I_{I}(\delta_I).
\end{align*}
Therefore, using generalized Rouché's Theorem (see Theorem \ref{th:a2}) again, we obtain that when $\tau\in (0,\tau_{\delta_I})$, $\mathcal A(z,\tau)$ is invertible in $\mathcal I_I(\delta_I)$.
\end{proof}

\begin{lemma} \label{le:4}
Let $z_0 \in \R $ be such that $z^2_0/c^2_1$ is a Neumann eigenvalue of the domain $\Omega$, and let $n_{|z_0|}$ denote the dimension of its corresponding space. The following arguments hold true.
\begin{enumerate}[(a)]
\item \label{c1} 
$\mathcal A_0(z)$ has the following expansion near $0$,
\begin{align}
& \mathcal A_0(z) = E_0(z) \left(\mathcal Q(0) +  z^2 \mathcal P(0)\right) F_0(z). \label{eq:85}
\end{align}

\item \label{c2} 
Assume $z_0 \ne 0$.
$\mathcal A_0(z)$ has the following expansion near $z_0$
\begin{align}\label{eq:49}
& \mathcal A_0(z) = E(z) \left(\mathcal P_0(z_0) + \sum_{l=1}^{n_{|z_0|}} (z - z_0) \mathcal P_l(z_0)\right) F(z).
\end{align}
Here,  
$\mathcal P_1(z_0),\ldots, \mathcal P_{n_{z_0}}(z_0)$ are mutually disjoint one-dimensional projections, $\sum_{l=0}^{n_{z_0}} \mathcal P_l(z_0) = \mathbb I$, and operators $E(z)$ and $F(z)$ are both holomorphic and invertible near $z_0$. 
\end{enumerate}

\end{lemma}

\begin{proof}
\eqref{c1} We note that $\mathcal A_0(z)$ is a Fredhlom operator with index 0. This, together with the assertion \eqref{b2} in the proof of statement \eqref{f2} of Lemma \ref{le:0} yields that 
\begin{align} \label{eq:44}
\mathcal A_0 \; \textrm{is normal at every point in}\; \Lambda_N.
\end{align}
Here, $\Lambda_N$ is specified in \eqref{eq:82}.
Then, using statement \eqref{h1} of Remark \ref{re:1}, we have,
\begin{align*}
\mathcal A_0(z) = E_{0}(z) \left(\mathcal P_0(0) + \sum_{l=1}^{n_0} z^{k_l(0)} \mathcal P_l(0)\right) F_0(z), 
\end{align*}
where $n_0:= {\rm{dim}}\left(\textrm{Ker}(\mathcal A_0(0))\right)$, $\mathcal P_1(0),\ldots, \mathcal P_{n_{0}}(0)$ are mutually disjoint one-dimensional projections, $k_1(0)\ldots, k_{n_{0}}(0)\in \mathbb N$, $\sum_{l=0}^{n_{0}} \mathcal P_l(0) = \mathbb I$, and $E_0(z)$ and $F_0(z)$ are both holomorphic and invertible near $0$. On the other hand, with the aid of \eqref{eq:3}, \eqref{eq:67}, \eqref{eq:68}, \eqref{eq:84} and the fact the Neumann eigenfunction corresponding to $0$ is $1$, we obtain 
\begin{align}\label{eq:167}
n_0 = 1 \quad \mathrm{and} \quad \textrm{Ker}(\mathcal A_0(0)) = F^{-1}_0(0)\mathrm{Ran}\left(P_1(0)\right).
\end{align}
Furthermore, since $K_{\Gamma,z/c_0}^* = K^*_{\Gamma,0} + O(z^2)$ and the function $\left(S^{-1}_01\right)(x)$ is an eigenfunction of the operator $K^*_{\Gamma,0}$ corresponding to the eigenvalue $-1/2$, by a straightforward calculation, we obtain 
\begin{align} \label{eq:168}
\mathcal A_0(z) 
\begin{bmatrix}
-\frac{\rho_0}{\rho_1} SL_{z/c_0} S_0^{-1} 1\\
S_0^{-1} 1 
\end{bmatrix} = z^2h(z),
\end{align}
where $h(z)$ is analytic in $L^2(\Omega)\times L^{2}(\Gamma)$ and $h(0) \ne (0,0)$. This, together with \eqref{eq:167} gives 
\begin{align*}
\begin{bmatrix}
-\frac{\rho_0}{\rho_1} SL_{0} S_0^{-1} 1\\
S_0^{-1} 1 
\end{bmatrix} = F^{-1}_0(0)e\quad \mathrm{for\; some}\; e\in \mathrm{Ran}\left(P_1(0)\right).
\end{align*}
Inserting $F^{-1}_0(0)e$ into \eqref{eq:168} gives $k_1(0) = 2$, implying \eqref{eq:85} holds.

\eqref{c2}
In the proof of this statement, we first claim that if
\begin{align} \label{eq:43}
\mathcal A_0(z) 
\begin{bmatrix}
\phi^{(0)}(z)\\
\psi^{(0)}(z)
\end{bmatrix}
= (z-z_0)^m
\begin{bmatrix}
\phi^{(1)}(z)\\
\psi^{(1)}(z)
\end{bmatrix}, \quad m \in \mathbb N,
\end{align}
we have $m=1$. Here, $\phi^{(j)}(z)$ and $\psi^{(j)}(z)$ are holomorphic near $z_0$ and $\left(\phi^{(j)}(z_0), \psi^{(j)}(z_0) \right) \ne \textbf{0}$ $(j\in \{0,1\})$. Then we have $m=1$. Based on this, \eqref{eq:44}, statement \eqref{h1} of Remark \ref{re:1}, we obtain \eqref{eq:49}.

Now we prove $m=1$ from the formula \eqref{eq:43}. We set
\begin{align}
-\phi^{(2)}(z) = \left(\frac{1}{c_0^{2}} - \frac{1}{c_1^{2}}\right) z^2 N_{\Omega,z/c_0}\phi^{(0)}(z) + SL_{\Gamma, z/c_0} \psi^{(0)}(z) \quad \textrm{in}\; \Omega. \label{eq:42} 
\end{align}
Clearly,
\begin{align} \label{eq:2}
\phi^{(2)}(z_0) = \phi^{(0)}(z_0) \quad \mathrm{on}\; \Omega \quad \partial_\nu \phi^{(2)}(z_0) = 0 \quad \mathrm{in}\; \Gamma
\end{align}
Furthermore, it follows from \eqref{eq:42} that 
\begin{align*}
&\Delta \phi^{(2)}(z) + z^2/c^2_1\phi^{(2)}(z) = (z-z_0)^m \left(\Delta + z^2/c_0^2 \right) \phi^{(1)} (z) \quad \rm{in}\; \Omega,\\
& \partial_\nu \phi^{(2)}(z) = (z-z_0)^m \psi^{(1)} (z) \quad \rm{on}\; \Gamma.\end{align*}
Thus, with the aid of \eqref{eq:2}, integrating by parts yields
\begin{align*}
&\int_{\Gamma} (z-z_0)^m (\psi^{(1)}(z))(y) \left(\overline{\phi^{(2)}(z_0)}\right)(y) d\sigma(y)\\ 
&= \int_{\Gamma} \frac{\partial \phi^{(2)}(z)(y)}{\partial \nu(y)} \left(\overline{\phi^{(2)}(z_0)}\right)(y) - \frac{\partial \left(\overline{\phi^{(2)}(z_0)}\right)(y)}{\partial \nu(y)} \phi^{(2)}(z)(y) d\sigma(y) \\
& = \int_{\Omega} \frac{{z^2_0} - z^2}{c_1^2} \left(\phi^{(2)}(z) \overline{\phi^{(2)}(z_0)}\right)(y)dy + \int_{\Omega} (z-z_0)^m \left(\overline{\phi^{(2)}(z_0)}\left(\Delta + z^2/c_0^2 \right)\phi^{(1)}(z)\right)(y) dy.
\end{align*}
If $m \ge 2$, let $z$ tend to $z_0$, we directly get
\begin{align*}
\int_{\Omega} \left(\left|\phi^{(2)}({z_0})\right|^2\right)(y) dy = 0
\end{align*}
which implies $\phi^{(0)}(z_0) = \phi^{(2)}(z_0) = 0$ in $\Omega$. This, together with \eqref{eq:42}, \eqref{eq:2}, jump relations of the single-layer potential and the well-posedness of the exterior scattering problem gives $\psi^{(0)}(z_0) = 0$, which is a contradiction to the nonzero assumption of $\left(\phi^{(0)}(z_0), \psi^{(0)}(z_0) \right)$. Therefore, $m$ must be equal to $1$.

\end{proof}

As a byproduct of Lemma \ref{le:0}, Lemma \ref{le:4} and Corollary \ref{co:1}, we establish the key result on the existence and uniqueness of the  Fabry-P\'erot resonances.

\begin{theorem} \label{th:1}
Let $\tau > 0$ and $I \subset \R$ be a bounded interval. There exists a constant $\delta_I> 0$ independent of $\tau$, and a constant $\tau_{\delta_I} > 0$ dependent on $\delta_I$, such that when $\tau \in (0, \tau_{\delta_I})$, for each $z_0 \in \Lambda_N \cap I$, $B_{\delta_I}(z_0)$ contains scattering resonances of the Hamiltonian $H_{\rho_\tau, k_\tau}$, all converging to $z_0$ as $\tau\rightarrow 0$. Here, $\Lambda_N$ is defined as \eqref{eq:82}. Furthermore, 
\begin{align}
&\sum |\{\lambda: \lambda \; \textrm{is a scattering resonance}\; \textrm{in}\; B_{\delta_I}(0)\}| = 2, \label{eq:166} \\
&\sum |\{\lambda: \lambda \; \textrm{is a scattering resonance}\; \textrm{in}\; B_{\delta_I}(z_0)\}| \le n_{|z_0|}, \quad \mathrm{if}\; z_0 \in \left(\Lambda_N \cap I\right)\backslash \{0\}, \label{eq:171}
\end{align}
where $n_{|z_0|}$ denotes the dimension of the Neumann eigenspace corresponding to $z^2_0/c^2_1$ .
Moreover, no resonances are located inside $\mathcal I_I(\delta_I)$, specified by \eqref{eq:74}.
\end{theorem}

\begin{proof}
The first assertion of this theorem easily follows from Lemma \ref{le:0} and Corollary \ref{co:1}.
Furthermore, based on assertions \eqref{b1}, \eqref{b2} and \eqref{b3} in the proof of statement \eqref{f3} of Lemma \ref{le:0}, we can use Lemma \ref{le:4} and Theorem \ref{th:a2} to obtain
\begin{align} 
& \sum_{\lambda \in B_{\delta_I(0)}} N_{\lambda}(\mathcal A) - N_{\lambda}(\mathcal A^{-1}) = 2, \label{eq:170}\\
& \sum_{\lambda \in B_{\delta_I(z_0)}}  N_{\lambda}(\mathcal A) - N_{\lambda}(\mathcal A^{-1}) = n_{|z_0|}. \label{eq:172}
\end{align}
Hence, \eqref{eq:170} together with Corollary \ref{co:1} and Lemma \ref{le:10} yields \eqref{eq:166}, while \eqref{eq:172} and Corollary \ref{co:1} give \eqref{eq:171}.
\end{proof}

\section{Asymptotic analysis for an extended resonator} \label{sec:4}

This section is devoted to deriving the asymptotic properties of the  Fabry-P\'erot resonances associated with the nonzero Neumann eigenvalues, as well as the resolvent and the scattered fields near these resonances, in terms of $\tau$, as introduced in sections \ref{sc:5.1}, \ref{sc:5.2} and \ref{sc:5.3}, respectively.

\subsection{Asymptotics of  Fabry-P\'erot resonances} \label{sc:5.1}

In this subsection, we focus on the asymptotic properties of the  Fabry-P\'erot resonances associated with the nonzero Neumann eigenvalues. To do so, we introduce several sesquilinear forms. For any $\phi_1, \phi_2 \in H^1(\Omega)$, we define
\begin{align}\label{eq:19}
a_{z_0}(\phi_1, \phi_2) = -\langle \nabla \phi_1, \nabla \phi_2 \rangle_{\mathbb L^2(\Omega)} + z_0^2c_1^{-2} \langle \phi_1, \phi_2 \rangle_{L^2(\Omega)} + \langle \mathcal P^N(|z_0|)\phi_1, \phi_2\rangle_{L^2(\Omega)},
\end{align}
where the operator $\mathcal P^N(|z_0|)$, defined by 
\begin{align} \label{eq:27}
\mathcal P^N(|z_0|) \phi:=\sum^{n_{|z_0|}}_{l=1} \left\langle \phi, e^{(l)}_{|z_0|}\right \rangle_{L^2(\Omega)}e^{(l)}_{|z_0|},
\end{align}
is a projection operator mapping $L^2(\Omega)$ onto the space spanned by $e^{(1)}_{|z_0|}, \ldots, e^{(n_{|z_0|})}_{|z_0|}$. Here, $z_0 \in \R \backslash \{0\}$ is a point at which $z^2_0/c^2_1$ is a Neumann eigenvalue of the domain $\Omega$, the corresponding orthogonal eigenfunctions with respect to the scalar product $\langle \cdot, \cdot\rangle_{L^2(\Omega)}$ are denoted by $e^{(1)}_{|z_0|},\ldots e^{(n_{|z_0|})}_{|z_0|}$, where $n_{|z_0|}\in \mathbb N$. 
Furthermore, given $\tau>0$, for any $\phi_1, \phi_2 \in H^1(\Omega)$, we define 
\begin{align} \label{eq:20}
a^{\rm{dom}}_{z,\tau}(\phi_1,\phi_2) &= a_{z_0}(\phi_1, \phi_2) + \bigg[\left(z^2-z_0^2\right)c_1^{-2} \langle \phi_1, \phi_2 \rangle_{L^2(\Omega)} \notag \\
& + \tau \frac{\rho_1}{\rho_0} 
\langle \left(D_N(z/c_0)\gamma\right)\phi_1, \phi_2 \rangle_{L^2(\Gamma)}\bigg] =: a_{z_0}(\phi_1, \phi_2) + a^{\rm{res}}_{z,\tau}(\phi_1,\phi_2).
\end{align}
It should be remarked that, owing to the meromorphic properties of the exterior Dirichlet to Neumann operator $D_N(z)$, for each $s \in \R_+$, there exists $ \delta_s \in \R_+$ such that $D_N(z)$ is analytic in the strip $\mathcal I_s$, which is defined by 
\begin{align*}
\mathcal I_{s}:= \{z\in \CC: |\textrm{Re}(z)| \le s, \; |\textrm{Im}(z)| \le \delta_s \}.
\end{align*}
Throughout this subsection, we assume that $z$ lies in the previously described type of strip, ensuring the analyticity of $D_N(z)$. Then, given $\tau > 0$, we define
\begin{align}
a_{z,\tau}(\phi_1,\phi_2) & = -\langle \nabla \phi_1, \nabla \phi_2 \rangle_{\mathbb L^2(\Omega)} + z^2c_1^{-2} \langle \phi_1, \phi_2 \rangle_{L^2(\Omega)} \notag \\
& + \tau\frac{\rho_1}{\rho_0}
\langle \left(D_N(z/c_0)\gamma\right)\phi_1, \phi_2 \rangle_{L^2(\Gamma)},  \quad \textrm{for}\; \phi_1,\phi_2 \in H^1(\Omega). \label{eq:66}
\end{align}

Let $\lambda_\tau$ be any Fabry-P\'erot resonance of the Hamiltonian $H_{\rho_\tau, k_\tau}$ (see Theorem \ref{th:1} for its existence). It follows from Lemma \ref{le:3} that $\lambda_\tau$ is a scattering resonance if and only if there exists a nonzero solution $\phi_{\lambda_\tau} \in H^1(\Omega)$ such that $a_{\lambda_{\tau}, \tau}(\phi_{\lambda_\tau}, \phi) = 0$ for every $\phi \in H^1(\Omega)$. Based on this, we aim to derive the asymptotics of $\lambda_\tau$ in term of $\tau$ by studying the solvability of the variational problem involving the sesquilinear form \eqref{eq:66}. It is worth mentioning that the variational method was first employed in \cite{FH-04} to characterize the asymptotics of the Minnaert resonance. Below, following the same spirit, we extend this approach to analyze the asymptotic behavior of Fabry-P\'erot resonances.

To do so, we first establish the well posedness of the variational problems corresponding to the sesquilinear forms \eqref{eq:19} and \eqref{eq:20}.

\begin{lemma} \label{le:1}
Let $z_0 \in \R \backslash \{0\}$ be such that $z^2_0/c^2_1$ is a Neumann eigenvalue of the domain $\Omega$. The following arguments hold true.
\begin{enumerate}[(a)]
\item \label{a1} Let the sesquilinear form $a_{z_0}\langle \cdot, \cdot\rangle$ be defined as \eqref{eq:19}. For each $h \in H^1(\Omega)$, there exists a unique $g^h \in H^1(\Omega)$ satisfying
\begin{align}\label{eq:71}
a_{z_0}(g^h, \phi) = \langle h, \phi \rangle_{H^1(\Omega)}, \quad \forall \phi \in H^1(\Omega).
\end{align}

\item \label{a2} Given $\tau>0$, we can find $\delta^{(1)}_{z_0}, \delta_{z_0}^{(2)} \in \R_+$ such that when $\tau \in (0,\delta^{(1)}_{z_0})$ and $z \in B_{\delta^{(2)}_{z_0}}(z_0)$, for each $h\in H^1(\Omega)$, there exists a unique $g^h_{{\rm{dom}},\tau,z} \in H^1(\Omega)$ satisfying
\begin{align} \label{eq:22}
a^{\rm{dom}}_{z,\tau}(g_{{\rm{dom}},\tau,z}^h, \phi) = \langle h, \phi \rangle_{H^1(\Omega)}, \quad \forall \phi \in H^1(\Omega).
\end{align}
Here, the sesquilinear form $a^{\rm{dom}}_{z,\tau}\langle \cdot, \cdot\rangle$ is defined as \eqref{eq:20}. Furthermore, $g^h_{{\rm{dom}},\tau,z}$ is analytic with respect to $\tau$ and $z-z_0$, i.e. 
\begin{align*}
g^h_{{\rm{dom}},\tau,z} = \sum^{\infty}_{p=1}\sum^{\infty}_{q=1} \tau^p(z-z_0)^q g^h_{p,q} \quad {\rm{for}}\; \tau \in (0,\delta^{(1)}_{z_0}) \; {\rm{and}}\; z \in B_{\delta^{(2)}_{z_0}}(z_0),
\end{align*}
where 
\begin{align*}
&\|g^h_{p,q}\|_{H^1({\Omega})} \le C_{p,q} \|h\|_{H^1(\Omega)},\\
&\quad {\rm{and}} \quad \sum^{\infty}_{p=1}\sum^{\infty}_{q=1}\tau^p (z-z_0)^q C_{p,q} < \infty,  \quad {\rm{for}}\; \tau \in (0,\delta^{(1)}_{z_0}) \; {\rm{and}}\; z \in B_{\delta^{(2)}_{z_0}}(z_0).
\end{align*}
Here, $C_{p,q}$ with $p,q\in \mathbb N$ is a positive constant independent of $h$.
\end{enumerate}

\end{lemma}

The proof of Lemma \ref{le:1} is provided in Appendix \ref{sec:a21}. Using this Lemma, we obtain the following result regarding solvability of the variational problem defined by the sesquilinear form \eqref{eq:66}.

\begin{lemma}\label{le:2}
Assume that $z_0 \in \R \backslash \{0\}$ is a point such that $z^2_0/c^2_1$ is a Neumann eigenvalue of the domain $\Omega$. 
Suppose that $\tau \in (0,\delta^{(1)}_{z_0})$ for some $\delta^{(1)}_{z_0} \in \R_+$ and that $z\in B_{\delta^{(2)}_{z_0}}(z_0)$ for some $\delta^{(2)}_{z_0} \in \R_+$, such that statement \eqref{a2} of Lemma \ref{le:1} holds. Let the sesquilinear forms $a^{\rm{dom}}_{z,\tau}\langle \cdot, \cdot\rangle$ and $a_{z,\tau}\langle \cdot, \cdot\rangle$ be defined as \eqref{eq:20} and \eqref{eq:66}, respectively.
For any $h\in H^1(\Omega)$, there exists a unique $g^h_{\tau,z}\in H^1(\Omega)$ satisfying 
\begin{align} \label{eq:23}
a_{z,\tau}(g^h_{\tau,z}, \phi) = \langle h, \phi \rangle_{H^1(\Omega)}, \quad \forall \phi \in H^1(\Omega),
\end{align}
if and only if the matrix $\mathbb I - \mathcal C(\tau,z)$ is invertible, where $\mathcal C(\tau,z)$ is a $n_{|z_0|} \times n_{|z_0|}$ matrix whose $(j,l)-$ element is given by
\begin{align} \label{eq:25}
\mathcal C_{j,l}(\tau,z) := \langle \eta_j(\tau, z), e^{(l)}_{|z_0|}\rangle_{L^2(\Omega)}, \quad j,l\in \left\{1,\ldots,n_{|z_0|} \right\}.
\end{align}
Here, $\eta_l(\tau, z)$ is the unique solution of 
\begin{align} \label{eq:26}
a^{{\rm{dom}}}_{z,\tau}(\eta_{l}(\tau, z),\phi) = \langle \phi, e^{(l)}_{|z_0|} \rangle_{L^2(\Omega)}, \quad \forall \phi \in H^1(\Omega).
\end{align}
Furthermore, when $\mathbb I-\mathcal C(\tau,z)$ is invertible, the solution of \eqref{eq:23} has the representation of 
\begin{align}
g^h_{\tau,z} & = g_{{\rm{dom}},\tau,z}^h + \left(\eta_1(\tau,z), \ldots, \eta_{n_{|z_0|}}(\tau,z)\right) \notag \\
&\left(\mathbb I-\mathcal C(\tau,z)\right)^{-1} \left( \left\langle g_{{\rm{dom}},\tau,z}^h , e^{(1)}_{|z_0|}\right\rangle_{L^2(\Omega)}, \ldots, \left\langle g_{{\rm{dom}},\tau,z}^h, e^{(n_{|z_0|})}_{|z_0|}\right\rangle_{L^2(\Omega)}\right)^{T}. \label{eq:24}
\end{align}
Here, $g_{\mathrm{dom},\tau,z}^h$ is the solution of \eqref{eq:22}. Moreover, when $\mathbb I-\mathcal C(\tau,z)$ is not invertible, for any $(b_1,\ldots b_{n_{|z_0|}}) \in \mathrm{Ker}(\mathbb I - \mathcal C(\tau,z))$,
\begin{align}
a_{z,\tau}\left(\sum^{n_{|z_0|}}_{l=1} b_l \eta_l(\tau , z), \phi\right) = 0, \quad \forall \phi \in H^1(\Omega). \label{eq:143}
\end{align}
\end{lemma}

\begin{proof}
By a straightforward calculation, we deduce from \eqref{eq:19} and \eqref{eq:20} that 
\begin{align*}
a_{z,\tau}(\phi_1, \phi_2) = a^{\textrm{dom}}_{z,\tau}(\phi_1, \phi_2) - \sum^{n_{|z_0|}}_{l=1} \left\langle \phi_1, e^{(l)}_{|z_0|} \right\rangle_{L^2(\Omega)}\left\langle \phi_2, e^{(l)}_{|z_0|} \right\rangle_{ L^2(\Omega)}, \quad \phi_1,\phi_2 \in H^1(\Omega).
\end{align*}
We note that $g_{\textrm{dom},\tau,z}^h$ satisfies \eqref{eq:22}. Therefore, with the aid of \eqref{eq:23}, we have 
\begin{align} \label{eq:21}
a^{\textrm{dom}}_{z,\tau}(g_{\tau,z}^h,\phi) - a^{\textrm{dom}}_{z,\tau}(g^h_{\textrm{dom},\tau,z},\phi) - \sum^{n_{|z_0|}}_{l=1} a^{\textrm{dom}}_{z,\tau}(\eta_l(\tau, z),\phi) \left\langle g_{\tau,z}^h, e^{(l)}_{|z_0|}\right\rangle_{L^2(\Omega)} = 0.
\end{align}
By statement \eqref{a2} of Lemma \ref{le:1}, \eqref{eq:21} implies 
\begin{align} \label{eq:126}
g_{\tau,z}^h = g_{\textrm{dom},\tau,z}^h + \sum^{n_{|z_0|}}_{l=1}\eta_l(\tau,z)\left\langle g_{\tau,z}^h,e^{(l)}_{|z_0|}\right\rangle_{L^2(\Omega)} \quad \textrm{in}\; \Omega,
\end{align}
which is equivalent to 
\begin{align*}
&\left\langle g_{\tau,z}^h, e^{(j)}_{|z_0|}\right\rangle_{L^2(\Omega)} = \left\langle g_{\mathrm{dom},\tau, z}^h, e^{(j)}_{|z_0|}\right\rangle_{L^2(\Omega)} + \sum^{n_{|z_0|}}_{l=1} \left\langle \eta_l(\tau, z), e^{(j)}_{|z_0|}\right\rangle_{L^2(\Omega)} \left\langle g_{\tau,z}^h, e^{(l)}_{|z_0|}\right\rangle_{L^2(\Omega)}
\end{align*}
for each $j \in \left\{1,\ldots,n_{|z_0|} \right\}$. With the aid of \eqref{eq:126}, we obtain that \eqref{eq:23} is solvable if and only if $\mathbb I - \mathcal C(\tau,z) $ is invertible and that $g^h_{\tau,z}$ satisfies \eqref{eq:24} when $\mathbb I - \mathcal C(\tau,z)$ is invertible. 

Furthermore, when $\mathbb I - \mathcal C(\tau,z)$ is not invertible, utilizing  \eqref{eq:126} again, we obtain that for every $(b_1,\ldots, b_{n_{|z_0|}}) \in \mathrm{Ker}(\mathbb I - \mathcal C(\tau,z))$, \eqref{eq:143} holds. 
\end{proof}

In the sequel, we derive the asymptotic expansion of $\mathcal C(\tau,z)$, where $\mathcal C(\tau,z)$ is specified in \eqref{eq:25}. Before proceeding with this objective, we define a $n_{|z_0|} \times n_{|z_0|}$ matrix $\mathcal M(z_0)$ whose $(j,l)-$th element is given by
\begin{align} \label{eq:17}
\mathcal M_{jl}(z_0) := 
-\left\langle D_N({z_0/c_0}) \gamma e^{(j)}_{|z_0|}, e^{(l)}_{|z_0|}\right\rangle_{L^2(\Gamma)}, 
\quad j,l \in \left\{1,\ldots,n_{|z_0|}\right\}.
\end{align} 

\begin{lemma} \label{le:7} 
Under the same assumptions as in Lemma \ref{le:2}, the matrix $\mathcal C(\tau,z)$ has the asymptotic expansion near $z_0$, 
\begin{align}\label{eq:16}
& \mathcal C(\tau,z) = 
\mathbb I - 2(z-z_0)\frac{z_0}{c_1^2} \mathbb I + \tau \frac{\rho_1}{\rho_0}  \mathcal M(z_0) + O(|z-z_0|^2 + \tau^2)\quad \mathrm{as}\; \tau \rightarrow 0.
\end{align}
Here, $\mathcal M(z_0)$ is given by \eqref{eq:17}.
\end{lemma}

\begin{proof}
Due to \eqref{eq:25}, we focus on analyzing the asymptotic behavior of $\eta_l(\tau, z)$, where $\eta_l(\tau,z)$ solves \eqref{eq:26} in the neighborhood of $z_0$. It follows from statement \eqref{a2} of Lemma \ref{le:1} that $\eta_l(\tau, z)$ satisfies
\begin{align}
&\left[\Delta + z_0^2 c^{-2}_1  + (z-z_0)(2z_0 + (z-z_0))c^{-2}_1 + \mathcal P^N(|z_0|) \right] \eta_l(\tau, z) = e^{(l)}_{|z_0|} \quad  \textrm{in}\; \Omega, \label{eq:28}\\
&\partial_{\nu} \eta_l(\tau, z) = \tau \frac{\rho_1}{\rho_0} D_N(z/c_0)\gamma\eta_l(\tau, z) \quad \textrm{on}\; \Gamma. \label{eq:29}
\end{align}
Here, the projection operator $\mathcal P^N(|z_0|)$ is specified by \eqref{eq:27} and $\eta_l(\tau, z)$ has the expansion 
\begin{align}
&\eta_l(\tau, z) = \sum^{\infty}_{p=0}\sum^{\infty}_{q=0} \tau^p(z-z_0)^q \eta_{l,p,q} \quad \textrm{in some neighborhood of}\; (0,z_0) \label{eq:30}.
\end{align}
Furthermore, due to the analyticity of $D_N(z)$, we have 
\begin{align}
D_N(z/c_0) = \sum^{+\infty}_{p=1} \frac{{(z-z_0)}^{p}}{c^{p}_0 p!} \partial^p_z D_N(z)|_{z_0/c_0} \quad \textrm{in some neighborhood of}\; z_0. \label{eq:31}
\end{align}
Inserting \eqref{eq:30} and \eqref{eq:31} into equations \eqref{eq:28}--\eqref{eq:29}, and equating the coefficients for $\tau^p(z-z_0)^q$ with $(p,q) = (0,0), (0,1), (1,0)$, we obtain that
\begin{align}
&\Delta \eta_{l,0,0} + z_0^2 c^{-2}_1 \eta_{l,0,0} + \mathcal P^N({|z_0|}) \eta_{l,0,0} = e^{(l)}_{|z_0|} \quad \rm{in}\; \Omega, \label{eq:32} \\
&\partial_\nu \eta_{l,0,0} = 0 \quad \rm{on}\; \Gamma, \label{eq:33}
\end{align}
\begin{align}
&\Delta \eta_{l,0,1} + z_0^2 c^{-2}_1 \eta_{l,0,1} + 2{z_0}{c^{-2}_1} \eta_{l,0,0} + \mathcal P^N({|z_0|}) \eta_{l,0,1} = 0 \quad \rm{in}\; \Omega, \label{eq:34} \\
&\partial_\nu \eta_{l,0,1} = 0 \quad \rm{on}\; \Gamma, \label{eq:35}
\end{align}
and
\begin{align}
&\Delta \eta_{l,1,0} + z_0^2 c^{-2}_1 \eta_{l,1,0} + \mathcal P^N({|z_0|}) \eta_{l,1,0} = 0 \quad \rm{in}\; \Omega, \label{eq:36}\\
&\partial_\nu \eta_{l,1,0} = \frac{\rho_1}{\rho_0} D_N(z_0/c_0)\gamma\eta_{l,0,0}\quad \rm{on}\; \Gamma. \label{eq:37}
\end{align}
In conjunction with \eqref{eq:32}, \eqref{eq:33}, \eqref{eq:34} and \eqref{eq:35}, we find
\begin{align} \label{eq:7}
\eta_{l,0,0} = e^{(l)}_{|z_0|}, \quad \eta_{l,0,1} = -\frac{2z_0}{c^2_1} e^{(l)}_{|z_0|}.
\end{align}
Moreover, utilizing \eqref{eq:36} and \eqref{eq:37}, we have
\begin{align*}
&-\left\langle \nabla \eta_{l,1,0},\nabla \phi \right \rangle_{\mathbb L^2(\Omega)} + z_0^2c_1^{-2} \left\langle \eta_{l,1,0}, \phi \right \rangle_{L^2(\Omega)}  \\
&+ \left\langle \frac{\rho_1}{\rho_0} D_N(z_0/c_0) \gamma \eta_{l,0,0}, \phi \right \rangle_{L^2(\Gamma)} + \left\langle \mathcal P^N(|z_0|)\eta_{l,1,0}, \phi \right\rangle_{L^2(\Omega)} = 0, \quad \rm{for}\; \phi\in H^1(\Omega).
\end{align*}
Setting $\phi = e^{(j)}_{|z_0|}$ with $j \in \left\{1,\ldots,n_{|z_0|} \right\} $ in the above equation, we arrive at
\begin{align*}
 \left\langle \eta_{l,1,0}, e^{(j)}_{|z_0|}\right\rangle_{L^2(\Omega)} = \left\langle \mathcal P^N(|z_0|) \eta_{l,1,0}, e^{(j)}_{|z_0|} \right\rangle_{L^2(\Omega)} = - \left\langle \frac{\rho_1}{\rho_0} D_N(z_0/c_0) \gamma \eta_{l,0,0}, e^{(j)}_{|z_0|} \right \rangle_{L^2(\Gamma)}. 
\end{align*}
This, together with \eqref{eq:25}, \eqref{eq:30} and \eqref{eq:7} gives the asymptotic expansion \eqref{eq:16}.
\end{proof}

We also need the following lemma.

\begin{lemma} \label{le:8}
Let $z_0 \in \R \backslash \{0\}$ be such that $z^2_0/c^2_1$ is a Neumann eigenvalue of the domain $\Omega$. Assume that $\lambda_{\mathcal M_{z_0}} \in \mathbb C$ is any eigenvalue of the matrix $\mathcal M(z_0)$, where the matrix $\mathcal M(z_0)$ is given by \eqref{eq:17}.
Then, we have 
\begin{align} \label{eq:86}
{\rm{Im}}(\lambda_{\mathcal M({z_0})}) < 0.
\end{align}
\end{lemma}

\begin{proof}
We denote by $V_{\lambda_{\mathcal M{(z_0)}}}$ the eigenspace corresponding to $\lambda_{\mathcal M{(z_0)}}$. For any nonzero vector $(b_1,\ldots, b_{n_{|z_0|}}) \in V_{\lambda_{\mathcal M{(z_0)}}}$, then 
\begin{align}
\mathcal M({z_0})(b_1,\ldots, b_{n_{|z_0|}}) = \lambda_{\mathcal M({z_0})} (b_1,\ldots, b_{n_{|z_0|}}). \notag
\end{align}
This directly implies 
\begin{align} \label{eq:18}
(\overline{b_1},\ldots, \overline{b_{n_{|z_0|}}}) \mathcal M(z_0) (b_1,\ldots, b_{n_{|z_0|}})^{T} = \lambda_{\mathcal M({z_0})}\sum^{n_{|z_0|}}_{l=1}|b_l|^2.
\end{align}
We set $w = \sum^{n_{|z_0|}}_{l=1} b_l e^{(l)}_{|z_0|}$ and consider the subsequent scattering problem
\begin{align*}
&\Delta w^{sc} + z_0^2/c_0^2 w^{sc} = 0 \quad \textrm{in}\; \Omega,\\
& w^{sc} = \gamma w \quad \textrm{on}\; \Gamma,\\
& w^{sc}\; \textrm{satisfies}\; z_0/c_0-\textrm{outgoing}.
\end{align*}
If ${\rm{Im}}(\lambda_{\mathcal M({z_0})}) \ge 0$, we deduce from \eqref{eq:18} that
\begin{align*}
\textrm{Im} \left(\int_{\Gamma} \partial_\nu w^{sc}(y) \overline{w^{sc}}(y) d\sigma(y)\right) \le 0.
\end{align*}
Thus, by Rellich lemma, we have $w^{sc}(y)=0$ in $\R^3\backslash \Omega$. This directly yields $w = 0$ on $\Gamma$. Therefore, by Holmgren's theorem, we obtain $w=0$ in $\Omega$, which implies 
\begin{align*}
b_1 =\cdots= b_{n_{|z_0|}} = 0.
\end{align*}
This leads to a contradiction. Hence, we obtain \eqref{eq:86}.
\end{proof}

As a consequence of Lemmas \ref{le:2}, \ref{le:7} and \ref{le:8}, we obtain the asymptotic expansion of the scattering resonance associated with the nonzero Neumann eigenvalue.

\begin{theorem} \label{th:2}
Let $\tau>0$. Assume that $z_0 \in \mathbb R \backslash \{0\}$ is a point at which $z^2_0/c^2_1$ is a Neumann eigenvalue of the domain $\Omega$. The scattering resonances of the Hamiltonian $H_{\rho_\tau,k_{\tau}}$ in the neighborhood of $z_0$ should satisfy one of the following asymptotic expansion 
\begin{align}
z^{(j)}_0(\tau) = z_0 + \tau \frac{c^2_1}{2z_0}\frac{\rho_1}{\rho_0}{\lambda^{(j)}_{\mathcal M(z_0)}}
+ o(\tau), \quad \textrm{as}\; \tau \rightarrow 0, \quad j\in\{1,\ldots,n_{|z_0|}\},\label{eq:38}
\end{align}
where $\lambda^{(j)}_{\mathcal M(z_0)}$ is the eigenvalue of the matrix $\mathcal M{(z_0)}$ which is given by \eqref{eq:17}.
\end{theorem}

\begin{proof}
We recall that $z$ is a scattering resonance of the Hamiltonian $H_{\rho_\tau, k_\tau}$ if and only if there exists a nonzero solution $\phi_z \in H^1(\Omega)$ such that $a_{z,\tau}(\phi_z, \phi) = 0$ for every $\phi \in H^1(\Omega)$.
Therefore, with the aid Lemma \ref{le:2}, we obtain that $z$ is a scattering resonance in the neighborhood of $z_0$ if and only if $\mathbb I- \mathcal C(\tau,z)$ is invertible. Moreover, it follows from Lemma \ref{le:8} that $\mathcal M(z_0)$ is invertible. This, together with the asymptotic expansion \eqref{eq:16} yields that every point where $\mathbb I- \mathcal C(\tau,z)$ is not injective, has the asymptotic expansion \eqref{eq:38} as $\tau \rightarrow 0$.
\end{proof}

\subsection{Asymptotics of the resolvents near  Fabry-P\'erot resonances} \label{sc:5.2}

This subsection is devoted to deriving the asymptotic behavior of the resolvent $R_{H_{\rho_\tau,k_\tau}}(\kappa)$ when $\kappa \in \R$ is near the  Fabry-P\'erot resonances associated with the nonzero Neumann eigenvalue.  We begin with the following observation. For every $f\in L_{\textrm{comp}}^2(\R^3)$ and $\kappa \in \R$, let $u^f_{\kappa,\tau}$ and $v^f_{\kappa}$ be given by \eqref{eq:72} and \eqref{eq:81}, respectively. With the aid of \eqref{eq:75}, we obtain that the subtraction of $u^f_{\kappa,\tau}$ and $v^f_{\kappa}$
, which is denoted by $w^f_{\kappa,\tau}$, satisfies
\begin{align} \label{eq:97}
&w^f_{\kappa,\tau}(x) = \left(\frac{1}{c^2_1} - \frac{1}{c^2_0} \right)\left[\kappa^2\left(N_{\Omega,\kappa/c_0}w^f_{\kappa,\tau}\right)(x) + \kappa^2 \left(N_{\Omega,\kappa/c_0}v^f_{\kappa}\right)(x) + \left(N_{\Omega,\kappa/c_0}f\right)(x) \right]\notag\\
&-\left(1-\frac{\rho_1 \tau}{\rho_0}\right)\left[\left(SL_{\Gamma, \kappa/c_0}(D_N(\kappa/c_0) \gamma w^f_{\kappa,\tau}\right)(x) + SL_{\Gamma, \kappa/c_0}\left(\partial_\nu v^f_\kappa\right)(x)\right], \quad x\in \R^3 \backslash \Gamma.
\end{align}
Our next aim is to derive the asymptotic properties of the resolvent based on the above equation. Before this, we introduce several novel notations. Given $g_1 \in X$ and $g_2 \in Y$, where $X$ and $Y$ are two Banach spaces, we write $g_1 = O_{\|\cdot\|_{X}}\left(a\|g_2\|_{Y}\right)$ if $\|g_1\|_{X}\le C a\|g_2\|_{Y} $ as $a \rightarrow 0$, where $C$ is a positive constant independent of $g_1$ and $g_2$. Given $\beta \in \R$, define
\begin{align*}
L_\beta^2(\R^3) := \left\{u\in L^2_{\textrm{loc}}\left(\R^3\right): (1+|x|^2)^{\frac\beta 2} u(x) \in L^2\left(\R^3\right)\right\}.
\end{align*}
Given $\psi \in H^{1/2}(\Gamma)$ and $\kappa \in \R$, define 
$\mathcal S_{\kappa}^{\mathrm{ex},D} \psi = g^{\mathrm{sc}}$, where $g^{\mathrm{sc}}$ solves the following scattering problem 
\begin{align*}
&\Delta g^{\mathrm{sc}} + \kappa^2 c^{-2}_0 g^{\mathrm{sc}} = 0  \quad \mathrm{in}\;  \R^3 \backslash \overline{\Omega},\\
& g^{\mathrm{sc}} = \psi \quad \mathrm{on}\; \Gamma,\\
& g^{\mathrm{sc}}\;\mathrm{is}\; \kappa/c_0-\mathrm{outgoing}.
\end{align*}
It is well known that (see \cite[Theorem 2.6]{DK19})
\begin{align*}
g^{\mathrm{sc}}(x) = \frac{e^{i\kappa|x-y_0|/c_0}}{4\pi|x-y_0|} g^{\infty}(\hat x) + O\left(\frac 1{|x|^2}\right),\quad \textrm{as}\; |x| \rightarrow \infty,
\end{align*}
uniformly in all directions $\hat x:= x/|x| \in \mathbb S^2$, where the function $g^{\infty}$ is called a far-field pattern of $g^{\mathrm{sc}}$. To characterize the relation between $\psi$ and $g^{\infty}$, we define
\begin{align} \label{eq:162}
\mathcal S_{\kappa}^{\infty,D} \psi = g^{\infty} \quad \mathrm{on}\; \mathbb S^2.
\end{align}
Furthermore, given $ h \in H^1(\Omega)$ and $\kappa \in \R$, we define 
\begin{align} \label{eq:73}
R^{\mathrm{ex}}_{\kappa} h :=  
\begin{cases}
h, &\quad \mathrm{in} \; \Omega,\\
\mathcal S_{\kappa}^{\mathrm{ex},D} \gamma h, &\quad \mathrm{in}\; {\R^3 \backslash \Omega}.
\end{cases}
\end{align}

\begin{theorem} \label{th:3}
Let $\kappa \in \R$, $\tau \in \R_+$ and $\beta > 1/2$. Assume that $z_0 \in \R \backslash \{0\}$ is a point at which $z^2_0/c^2_1$ is a Neumann eigenvalue of the domain $\Omega$, and that the corresponding orthogonal eigenfunctions with respect to the scalar product $\langle \cdot, \cdot\rangle_{L^2(\Omega)}$ are denoted by $e^{(1)}_{|z_0|},\ldots e^{(n_{|z_0|})}_{|z_0|}$, where $n_{|z_0|}\in \mathbb N$.  Given $f \in L^2_{{\rm{comp}}}(\R^3)$, we have the asymptotic expansion
\begin{align}
R_{H_{\rho_\tau,k_\tau}}(\kappa) f  = - \frac{1}{c_0^2} R_{\kappa/c_0} f + \sum^4_{l=1}R^{(l)}_{H_{\rho_\tau,k_\tau}}(\kappa) f \notag
\end{align}
with
\begin{align}
&R^{(1)}_{H_{\rho_\tau,k_\tau}}(\kappa) f =  \left[R^{\mathrm{ex}}_{\kappa} - N_{\Omega,\kappa/c_0} \mathcal P^N(|z_0|)\right]w^f_{0,0}, \label{eq:78}\\
&R^{(2)}_{H_{\rho_\tau,k_\tau}}(\kappa) f = R^{\mathrm{ex}}_{\kappa} e_{|z_0|}^{f,1}, \quad
R^{(3)}_{H_{\rho_\tau,k_\tau}}(\kappa) f = R^{\mathrm{ex}}_{\kappa} e_{|z_0|}^{f,2} \label{eq:140}
\end{align}
and $R^{(4)}_{H_{\rho_\tau,k_\tau}}(\kappa) f$ satisfying
\begin{align*}
&\left\|R^{(4)}_{H_{\rho_\tau,k_\tau}}(\kappa) f\right\|_{L^2_{-\beta}(\R^3)} \le C (\tau+ |\kappa - z_0|)\left(\left\|v_{\kappa}^f\right\|_{H^1(\Omega)} + \|f\|_{L^2(\R^3)} + \left\|e_{|z_0|}^{f,1}\right\|_{H^1(\Omega)}\right),
\end{align*} 
{as} $\tau \rightarrow 0 $ and $\kappa \rightarrow z_0$, where $C$ is a positive constant independent of $f$, $\tau$ and $\kappa$, $R^{\mathrm{ex}}_{\kappa}$ is defined in \eqref{eq:73}, $\mathcal P^N(|z_0|)$ is specified in \eqref{eq:27}, $w^f_{0,0}$ solves
\begin{align}
&\Delta w^f_{0,0} + z_0^2 c^{-2}_1 w^f_{0,0} + \mathcal P^N({|z_0|}) w^f_{0,0} = \left(\frac{1}{c_0^{2}}-\frac{1}{c^{2}_1}\right)\left(f + \kappa^2 v^f_{\kappa}\right) \quad \rm{in}\; \Omega, \label{eq:89}\\
&\partial_\nu w^f_{0,0} = -\partial_\nu v^f_{\kappa} \quad \rm{on}\; \Gamma,\label{eq:90}
\end{align}
and $e_{|z_0|}^{f,j}$ $(j=1,2)$ are defined by
\begin{align*}
& e_{|z_0|}^{f,j}:= \mathbf e_{|z_0|}\left( 2(\kappa-z_0)z_0 \mathbb I + \tau \frac {\rho_1}{\rho_0} \mathcal M(z_0)\right)^{-1} {\mathbf b}^{f,j}  \quad \mathrm{in}\; \Omega.
\end{align*}
Here, $\mathbf e_{|z_0|}$ is denoted by
\begin{align} \label{eq:96} 
\mathbf e_{|z_0|}(x):=\left(e^{(1)}_{|z_0|}(x), \ldots, e^{(n_{|z_0|})}_{|z_0|}(x)\right)\quad \mathrm{for}\; x \in \Omega,
\end{align}
and ${\mathbf b}^{f,1}$ and $\mathbf b^{f,2}$ are two column vectors whose elements are given by 
\begin{align*}
&{\mathbf b}_l^{f,1}:=-c^{-2}_1\left\langle f, e^{(l)}_{|z_0|}\right\rangle_{L^2(\Omega)}, \\
&{\mathbf b}_l^{f,2}:=-\frac{2z_0}{c^2_1}(\kappa-z_0) \left\langle v^f_{\kappa}, e^{(1)}_{|z_0|}\right\rangle_{L^2(\Omega)} - \frac{\rho_1 \tau}{\rho_0}\left\langle \partial^+_\nu R^{\mathrm{ex}}_{z_0}w^f_{0,0}, e^{(l)}_{|z_0|}\right\rangle_{L^2(\Gamma)}
\end{align*}
for  $l=1,\ldots,n_{|z_0|}$.
\end{theorem}

\begin{proof}
It is easy to verify that $w^f_{\kappa,\tau}$ satisfies  
\begin{align}
&\Delta w^f_{\kappa,\tau} + \kappa^2 c_1^{-2} w^f_{\kappa,\tau} = \left(\frac{1}{c_1^{2}}-\frac{1}{c^{2}_0}\right)\left(-f - \kappa^2 v^f_{\kappa}\right)\quad \textrm{in}\; \Omega, \label{eq:1} \\
&\partial_{\nu} w_{\kappa,\tau}^f  = \frac{\rho_1}{\rho_0}\tau \left[D_N(\kappa/c_0) \gamma  w_{\kappa,\tau}^f + \partial_\nu v^f_\kappa \right] - \partial_\nu v^f_\kappa \quad\; \textrm{on}\; \Gamma. \label{eq:52}
\end{align}
Since $w^f_{\kappa,\tau}$ solves \eqref{eq:1}--\eqref{eq:52}, we find
\begin{align*}
&a_{\kappa,\tau}({w^f_{\kappa,\tau}}{|_\Omega}, \phi) = \left(v_{\kappa,1}^{f}, \phi \right)_{H^1(\Omega)} \quad \textrm{for}\; \phi \in H^1(\Omega).
\end{align*}
Here, the sesquilinear form $a_{\kappa,\tau}\langle \cdot, \cdot\rangle$ is defined as \eqref{eq:66} and $v_{\kappa,1}^f \in H^1(\Omega)$ is determined by 
\begin{align}
\left \langle v_{\kappa,1}^f, \phi \right \rangle_{H^1(\Omega)} = \left \langle \left(c_1^{-2}-c^{-2}_0\right \rangle \left(-f - \kappa^2v^f_{\kappa}\right), \phi \right\rangle_{L^2(\Omega)} + \left(1 - \frac{\rho_1 \tau}{\rho_0}\right)\left\langle \partial_\nu v^f_\kappa, \gamma \phi \right\rangle_{L^2(\Gamma)} \label{eq:79}
\end{align}
for $\phi \in H^1(\Omega)$. Therefore, it follows from Lemma \ref{le:2} that 
\begin{align}
w^f_{\kappa,\tau} &= w^f_{\mathrm{dom}} + \left(\eta_1(\tau,\kappa), \ldots, \eta_{n_{|z_0|}}(\tau, \kappa )\right) \notag \\
&\left(\mathbb I-\mathcal C(\tau, \kappa)\right)^{-1} \left( \left\langle w_{{\rm{dom}}}^f , e^{(1)}_{|z_0|}\right\rangle_{L^2(\Omega)}, \ldots, \left\langle w_{{\rm{dom}}}^f, e^{(n_{|z_0|})}_{|z_0|}\right\rangle_{L^2(\Omega)}\right)^{T}.\label{eq:145} 
\end{align}
Here, $\eta_l$ solves \eqref{eq:26} and $w_{\textrm{dom}}^f$ is the solution of 
\begin{align} \label{eq:87}
a^{\rm{dom}}_{\kappa,\tau}(w_{{\rm{dom}}}^f, \phi) = \langle v^f_{\kappa,1}, \phi \rangle_{H^1(\Omega)}, \quad \forall \phi \in H^1(\Omega),
\end{align}
where the sesquilinear form $a^{\rm{dom}}_{\kappa,\tau}$ is specified by \eqref{eq:20}.
Using statement \eqref{a2} of Lemma \ref{le:1}, we rewrite $w^f_{\mathrm{dom}}$ as
\begin{align}
w_{\textrm{dom}}^{f} = w^f_{0,0} + \tau w^f_{1,0} + (\kappa - z_0) w^f_{0,1} + \sum^{\infty}_{p=1}\sum^{\infty}_{q=1} \tau^p(\kappa-z_0)^q w^f_{p,q} \label{eq:88}
\end{align}
\textrm{in some neighborhood of} $(0,z_0)$, which are uniformly in the $H^1(\R^3)-$ norm of $v^f_{\kappa,1}$. We note that
\begin{align} \label{eq:146}
\|v^f_{\kappa,1}\|_{H^1(\Omega)} \le C\|v^f_{\kappa}\|_{H^1(\Omega)} + C\|f\|_{L^2(\Omega)}.
\end{align}
Inserting \eqref{eq:79} into equation \eqref{eq:87}, and equating the coefficients for $\tau^p(z-z_0)^q$ with $(p,q) = (0,0), (0,1), (1,0)$, where $l \in \mathbb N_0$,
we find that $w^f_{0,0}$ satisfies equations \eqref{eq:89}--\eqref{eq:90}, and that $w^f_{0,1}$ and $w^f_{1,0}$ solve
\begin{align}
&\Delta w^f_{0,1} + z_0^2 c^{-2}_1 w^f_{0,1} + {2z_0}{c^{-2}_1} w^f_{0,0} + \mathcal P^N({|z_0|}) w^f_{0,1} = 0 \quad \rm{in}\; \Omega, \label{eq:93}\\
&\partial_\nu w^f_{0,1} = 0 \quad \rm{on}\; \Gamma \label{eq:94}
\end{align}
and
\begin{align}
&\Delta w^f_{1,0} + z_0^2 c^{-2}_1 w^f_{1,0} + \mathcal P^N({|z_0|}) w^f_{1,0} = 0 \quad \rm{in}\; \Omega, \label{eq:91}\\
&\partial_\nu w^f_{1,0} = \frac{\rho_1}{\rho_0}\left[\partial_\nu v^f_{\kappa} + D_N(z_0/c_0) \gamma w^f_{0,0}\right] \quad \rm{on}\; \Gamma, \label{eq:92}
\end{align}
respectively. 

In the sequel, we calculate $\left\langle w^f_{0,0}, e^{(l)}_{|z_0|} \right\rangle_{L^2(\Omega)}$, $\left\langle w^f_{0,1}, e^{(l)}_{|z_0|} \right\rangle_{L^2(\Omega)}$ and $\left\langle w^f_{1,0}, e^{(l)}_{|z_0|} \right\rangle_{L^2(\Omega)}$, where $l=1,\ldots,n_{|z_0|}$. For the estimate of $\left\langle w^f_{0,0}, e^{(l)}_{|z_0|} \right\rangle_{L^2(\Omega)}$, it follows from \eqref{eq:89}--\eqref{eq:90} that 
\begin{align}
&-\left\langle \partial_\nu v^f_{\kappa}, e^{(l)}_{|z_0|} \right\rangle_{L^2(\Gamma)} + \left\langle v^f_{\kappa}, \partial_\nu e^{(l)}_{|z_0|} \right\rangle_{L^2(\Gamma)}= \left\langle \partial_\nu w^f_{0,0}, e^{(l)}_{|z_0|} \right\rangle_{L^2(\Gamma)} - \left\langle  w^f_{0,0}, \partial_\nu e^{(l)}_{|z_0|} \right\rangle_{L^2(\Gamma)} \notag\\
&\qquad \qquad \qquad \qquad \qquad =\left(\frac{1}{c_1^{2}}-\frac{1}{c^{2}_0}\right)\langle -f- \kappa^2 v^f_{\kappa}, e^{(l)}_{|z_0|}\rangle_{L^2(\Omega)} -\langle w^f_{0,0}, e^{(l)}_{|z_0|} \rangle_{L^2(\Omega)}. \notag
\end{align}
From this, we can use \eqref{eq:81} to obtain 
\begin{align} \label{eq:95}
& \left\langle w^f_{0,0}, e^{(l)}_{|z_0|}\right\rangle_{L^2(\Omega)} = \frac{1}{c^2_1} \left\langle -f, e^{(l)}_{|z_0|} \right\rangle_{L^2(\Omega)} + \frac{z^2_0 -\kappa^2 }{c^2_1} \left\langle v^f_{\kappa}, e^{(l)}_{|z_0|} \right\rangle_{L^2(\Omega)}.
\end{align}
For the estimate of $\left\langle w^f_{0,1}, e^{(l)}_{|z_0|} \right\rangle_{L^2(\Omega)}$, 
with the aid of \eqref{eq:93} and \eqref{eq:94}, we have
\begin{align}
&0 = \left\langle \partial_\nu w^f_{0,1}, e^{(l)}_{|z_0|} \right\rangle_{L^2(\Gamma)} - \left\langle  w^f_{0,1}, \partial_\nu e^{(l)}_{|z_0|} \right\rangle_{L^2(\Gamma)} = \left\langle -\frac{2z_0}{c^2_1}w^f_{0,0} - w^f_{0,1} , e^{(l)}_{|z_0|} \right\rangle_{L^2(\Omega)}. \notag
\end{align}
This, together with \eqref{eq:95} implies that 
\begin{align} \label{eq:100}
\left\langle w^f_{0,1}, e^{(l)}_{|z_0|} \right\rangle_{L^2(\Omega)} = -\frac{2z_0}{c^2_1}\left[ \frac{1}{c^2_1} \left\langle -f, e^{(l)}_{|z_0|} \right\rangle_{L^2(\Omega)} + \frac{z^2_0 -\kappa^2 }{c^2_1} \left\langle v^f_{\kappa}, e^{(l)}_{|z_0|} \right\rangle_{L^2(\Omega)} \right].
\end{align}
For the estimate of $\left\langle w^f_{1,0}, e^{(l)}_{|z_0|} \right\rangle_{L^2(\Omega)}$, by utilizing \eqref{eq:91} and \eqref{eq:92}, we have
\begin{align}
&-\left\langle w^f_{1,0}, e^{(l)}_{|z_0|} \right\rangle_{L^2(\Omega)} - \frac{\rho_1}{\rho_0}\left\langle D_N(z_0/c_0) \gamma w^f_{0,0}, e^{(l)}_{|z_0|}\right\rangle_{L^2(\Gamma)} \notag\\
&= \frac{\rho_1}{\rho_0} \left\langle \partial_\nu v^f_{\kappa}, e^{(l)}_{|z_0|} \right\rangle_{L^2(\Gamma)} - \frac{\rho_1}{\rho_0} \left\langle v^f_{\kappa}, \partial_\nu e^{(l)}_{|z_0|} \right\rangle_{L^2(\Gamma)}\notag\\
&= \frac{\rho_1}{\rho_0}\left[\left(\frac{z_0^2}{c_1^{2}}-\frac{\kappa^2}{c^{2}_0}\right) \left\langle \kappa^2 v^f_{\kappa}, e^{(l)}_{|z_0|}\right\rangle_{L^2(\Omega)} - \frac{1}{c^2_0} \left\langle f, e^{(l)}_{|z_0|} \right\rangle_{L^2(\Omega)}\right]\notag.
\end{align}

Combining this with \eqref{eq:30}, \eqref{eq:7}, \eqref{eq:73}, \eqref{eq:145}, \eqref{eq:88}, \eqref{eq:146}, \eqref{eq:95} and \eqref{eq:100} gives 
\begin{align}
 w^f_{\kappa,\tau} & = \left[w_{0,0}^f + O_{H^1(\Omega)}\left((\tau + |\kappa - z_0|)\left(\|v^f_{\kappa}\|_{H^1(\Omega)} + \|f\|_{L^2(\Omega)}\right)\right)\right] \notag\\
& + \left[\mathbf e_{|z_0|} + O_{H^1(\Omega)}(\tau + |\kappa - z_0|)\right]\left(2(\kappa - z_0)z_0 \mathbb I + \tau \frac{\rho_1}{\rho_0} \mathcal M(z_0) + O\left(|\kappa-z_0| + \tau\right)^2\right)^{-1} \notag\\ 
&\quad\left[{\mathbf b}^{f} + O\left((|\kappa - z_0| + \tau)^2 \left(\|v^f_{\kappa}\|_{H^1(\Omega)} + \|f\|_{L^2(\Omega)}\right)\right)\right] 
\quad {\rm{in}}\; \Omega, \quad \textrm{as}\; \tau \rightarrow 0,\label{eq:149}
\end{align}
uniformly in some neighborhood of $z_0$, where the $l-$th entry of the column vector ${\mathbf b}^{f}$ satisfy  
\begin{align}
&{\mathbf b}_{l}^{f} = \left[1- \frac{2z_0}{c^2_1}(\kappa-z_0)\right]\left(\frac{1}{c^2_1} \left\langle -f, e^{(l)}_{|z_0|} \right\rangle_{L^2(\Omega)} + \frac{z^2_0 -\kappa^2 }{c^2_1} \left\langle v^f_{\kappa}, e^{(l)}_{|z_0|} \right\rangle_{L^2(\Omega)}\right) - \frac{\rho_1 \tau}{\rho_0} \notag\\
&\left[\left(\frac{z_0^2}{c_1^{2}}-\frac{\kappa^2}{c^{2}_0}\right) \left\langle \kappa^2 v^f_{\kappa}, e^{(l)}_{|z_0|}\right\rangle_{L^2(\Omega)} - \frac{1}{c^2_0} \left\langle f, e^{(l)}_{|z_0|} \right\rangle_{L^2(\Omega)} +  \left\langle\partial^+_\nu R^{\mathrm{ex}}_{z_0}w^f_{0,0},e^{(l)}_{|z_0|}\right\rangle_{L^2(\Gamma)}\right]. \label{eq:150}
\end{align}
Since $\mathcal M(z_0)$ is invertible, it can be deduced that 
\begin{align} \label{eq:109}
(\tau+ |\kappa - z_0|)^s\left(2(\kappa - z_0)z_0 \mathbb I + \tau \frac{\rho_1}{\rho_0} \mathcal M(z_0)\right)^{-1} \le C (\tau + |\kappa - z_0|)^{s-1},
\end{align}
{as} $\tau \rightarrow 0$ {and} $\kappa \rightarrow z_0$, where $s \in \{1,2\}$. With the aid of \eqref{eq:149}, \eqref{eq:150} and \eqref{eq:109}, we have 
\begin{align} \label{eq:98}
w^f_{\kappa,\tau} = w^f_{0,0} + e^{f,1}_{|z_0|}  + e^{f,2}_{|z_0|} + w^f_{\mathrm{Res}}\quad \mathrm {in}\; \Omega,
\end{align}
where 
\begin{align}\label{eq:114}
\|w^f_{\mathrm{Res}}\|_{H^1(\Omega)} \le C (\tau+ |\kappa - z_0|)\left(\left\|v_\kappa^f\right\|_{H^1(\Omega)} + \|f\|_{L^2(\R^3)} + \left\|e_{|z_0|}^{f,1}\right\|_{H^1(\Omega)}\right), 
\end{align}
{as} $\tau \rightarrow 0$ {and} $\kappa \rightarrow z_0$.
Furthermore, it can be seen from \eqref{eq:89} and \eqref{eq:90} that 
\begin{align}
&DL_{\Gamma, \kappa/c_0} \gamma(w^f_{0,0} + v^f_{\kappa}) -  SL_{\Gamma, \kappa/c_0}\left[D_N(\kappa/c_0) \gamma w^f_{0,0} + \partial_\nu v^f_{\kappa}\right] - DL_{\Gamma, \kappa/c_0} \gamma(w^f_{0,0} + v^f_{\kappa}) \notag \\
&= DL_{\Gamma, \kappa/c_0} \gamma w^f_{0,0}   - SL_{\Gamma, \kappa/c_0} D_N(\kappa/c_0) \gamma w^f_{0,0}  + SL_{\Gamma, \kappa/c_0}\partial_\nu w^f_{0,0} - DL_{\Gamma, \kappa/c_0} \gamma w^f_{0,0}\notag \\
& = -\left(\frac{1}{c^2_1} - \frac{1}{c^2_0} \right) N_{\Omega, \kappa/c_0}f -\left(\frac{\kappa^2}{c^2_1} - \frac{\kappa^2}{c^2_0} \right)N_{\Omega, \kappa/c_0}(w^f_{0,0} + v^f_{\kappa}) + \frac{\kappa^2 - z^2_0}{c^2_1} N_{\Omega, \kappa/c_0}w^f_{0,0} \notag\\
&\quad - N_{\Omega,\kappa/c_0} \mathcal P^N(|z_0|) w^f_{0,0} + R^{\mathrm{ex}}_{\kappa} w^f_{0,0}, \quad \mathrm{in}\;\R^3. \label{eq:151}
\end{align}
Similarly, for any Neumann eigenfunction $e(x)$ associated with the eigenvalue $z^2_0/c^2_1$, we have 
\begin{align}
& \left(\frac{1}{c^2_1} - \frac{1}{c^2_0} \right) \kappa^2 N_{\Omega, \kappa/c_0}e - SL_{\Gamma, \kappa/c_0}D_N(\kappa/c_0) \gamma e = \frac{\kappa^2 - z^2_0}{c^2_1} N_{\Omega, \kappa/c_0}e + R^{\mathrm{ex}}_{\kappa} \gamma e \quad \mathrm{in}\;\R^3. \label{eq:152}
\end{align}
Using \eqref{eq:97}, \eqref{eq:98}, \eqref{eq:151} and \eqref{eq:152}, we arrive at 
\begin{align*}
w^f_{\kappa,\tau} = \sum_{l=1}^{4} R^{(l)}_{H_{\rho_\tau,k_\tau}}(\kappa) f,
\end{align*}
where
\begin{align}
R^{(4)}_{H_{\rho_\tau,k_\tau}}(\kappa) f = \frac{\rho_1 \tau}{\rho_0}SL_{\Gamma, \kappa/c_0}(D_N(\kappa/c_0) \gamma w^f_{\kappa,\tau} + \left(\frac{1}{c^2_1} - \frac{1}{c^2_0} \right)\kappa^2 N_{\Omega,\kappa/c_0}w^f_{\mathrm{Res}} \notag\\
-SL_{\Gamma, \kappa/c_0}(D_N(\kappa/c_0) \gamma w^f_{\mathrm{Res}} + \frac{\kappa^2 - z^2_0}{c^2_1} N_{\Omega, \kappa/c_0}\left(w^f_{0,0} + e_{|z_0|}^{f,1} + e^{f,2}_{|z_0|}\right) \quad \mathrm{in}\; \R^3 \backslash \Gamma.\label{eq:157}
\end{align} 
This, together with \eqref{eq:98}, \eqref{eq:114} and the fact that 
\begin{align}
SL_{\Gamma, \kappa/c_0} \in \mathcal L(H^{-1/2}(\Gamma), L_{-\beta}^2(\R^3))\; \mathrm{and}\; N_{\Omega, \kappa/c_0} \in \mathcal L(L^{2}(\Gamma), L_{-\beta}^2(\R^3)) \label{eq:158}
\end{align}
yields the assertion of this theorem.
\end{proof}

We conclude this section by investigating the asymptotic behavior of the scattering problem associated with the Hamiltonian $H_{\rho_\tau,k_\tau}$ as $\tau$ tends to $0$.
    
\subsection{Asymptotics of the scattered fields} \label{sc:5.3}

We consider a time-harmonic plane wave $u_\kappa^{\mathrm{in}}:= e^{i\kappa x\cdot d/c_0}$ propagating in the direction $d\in \mathbb S^2$ and impinging on $\Omega$, 
where $\kappa >0$ is a given incident frequency.
Then the scattering by the object $\Omega$ can be mathematically formulated as the problem of finding the total field $u^{\mathrm{tot}}_{\kappa} $ such that 
\begin{align*}
&\nabla \cdot \frac 1 {\rho_\tau} \nabla u^{\mathrm{tot}}_{\kappa} +\kappa^2 \frac{1}{k_\tau}u^{\mathrm{tot}}_{\kappa} = 0\; \quad\quad\quad\;\;\;\;\;\;\;\;\;\;\;\; \;\; \; \;\; \;\;\; \text{in}\; \R^3, \\
& u^{\mathrm{tot}}_{\kappa} = u^{\mathrm{sc}}_{\kappa} + u^{\mathrm{in}}_\kappa \;\;\;\;\;\;\; \;\;\;\; \qquad \qquad \qquad \qquad \qquad\; \;\;\;\; \text{in}\; \mathbb R^3,\\
&\lim_{|x|\rightarrow +\infty}\left(\frac{x}{|x|}\cdot\nabla - i\frac{\kappa}{c_0}\right)u^{\mathrm{sc}}_{\kappa} = 0.
\end{align*}
Here, $\rho_\tau$ and $k_\tau$, specified as \eqref{eq:0}, represent the mass density and the bulk modulus of the acoustic medium in the presence of the object $\Omega$, respectively. The mathematical model described above is used to describe  the wave propagation in bubbly media within the linar regime, see \cite{C-M-P-T-1, C-M-P-T-2} for more details.

The following theorem describes the asymptotic behavior of the above scattering problem as $\tau$ tends to $0$.

\begin{theorem}\label{th:5}
Let $\kappa,\tau \in \R_+$ and $\beta > 1/2$. Assume that $z_0 \in \R_+$ is a point at which $z^2_0/c^2_1$ is a Neumann eigenvalue of the domain $\Omega$, and that the corresponding orthogonal eigenfunctions with respect to the scalar product $\langle \cdot,\cdot\rangle_{L^2(\Omega)}$ are denoted by $e^{(1)}_{|z_0|},\ldots e^{(n_{|z_0|})}_{|z_0|}$, where $n_{|z_0|}\in \mathbb N$. We have 
\begin{align}
u_{\kappa}^{\mathrm{sc}} = R^{\mathrm{ex}}_{\kappa} e^{\mathrm{tot}}_{|z_0|}- R^{\mathrm{ex}}_{\kappa} u_\kappa^{\mathrm{in}} + u^{\mathrm{sc}}_{\kappa, \mathrm{Res}} \label{eq:115}
\end{align}
with
\begin{align}
&\left\|u^{\mathrm{sc}}_{\kappa, \mathrm{Res}}\right\|_{L^2_{-\beta}(\R^3)} \le C \left(\tau + |\kappa -z_0|\right)\left\|u_\kappa^{\mathrm{in}}\right\|_{H^1(\Omega)}, \label{eq:116}
\end{align}
{as} $\tau \rightarrow 0$ {and} $\kappa \rightarrow z_0$,
where $C$ is positive constant independent of $\kappa$ and $\tau$, $R^{\mathrm{ex}}_{\kappa}$ is given by \eqref{eq:73} and $e^{\mathrm{tot}}_{|z_0|}$ is defined by
\begin{align*}
e^{\mathrm{tot}}_{|z_0|}:= \mathbf e_{|z_0|} \left( 2(\kappa-z_0)z_0 \mathbb I + \tau \frac {\rho_1} {\rho_0} \mathcal M(z_0)\right)^{-1} \mathbf{b}^{\mathrm{tot}}.
\end{align*}
Here, $\mathbf e_{|z_0|}$ is specified in \eqref{eq:96} and the $l-$th entry of the column vector ${\mathbf b}^{\mathrm{tot}}$ satisfy
\begin{align*}
&{\mathbf b}_{l}^{\mathrm{tot}} = -\frac{\rho_1 \tau}{\rho_0}\bigg[\left\langle \partial_\nu u^{\mathrm{in}}_{\kappa} - \partial^+_\nu R^{\mathrm{ex}}_{\kappa}u^{\mathrm{in}}_{\kappa} , e^{(l)}_{|z_0|}\right\rangle_{L^2(\Gamma)}\bigg], \;\; l=1,\ldots, n_{|z_0|}.
\end{align*}
\end{theorem}

\begin{proof}
By a straightforward calculation, for $x\in \R^3 \backslash \Gamma$, we have
\begin{align}
& u^{\mathrm{tot}}_{\kappa}(x)= u_\kappa^{\mathrm{in}}(x) + \left(\frac{1}{c^2_1} - \frac{1}{c^2_0} \right)\kappa^2\left(N_{\Omega,\kappa/c_0}u^{\mathrm{tot}}_{\kappa}\right)(x)-\left(\frac{\rho_0}{\rho_1\tau}-1\right) \left(SL_{\Omega,\kappa/c_0} \partial_\nu u^{\mathrm{tot}}_{\kappa}\right)(x). \label{eq:46}
 \end{align}
Here, the value $u^{\mathrm{tot}}_{\kappa}$ within $\Omega$ is determined by 
\begin{align}
&\Delta u^{\mathrm{tot}}_{\kappa} + \kappa^2 c_1^{-2} u^{\mathrm{tot}}_{\kappa} = 0 \quad \textrm{in}\; \Omega, \label{eq:147} \\
&\partial_{\nu} u_{\kappa}^{\mathrm{tot}} = \frac{\rho_1}{\rho_0} \tau \left[D_N(\kappa/c_0)\gamma u^{\mathrm{tot}}_{\kappa} - D_N(\kappa/c_0)\gamma u^{\mathrm{in}}_\kappa + \partial_\nu u^{\mathrm{in}}_\kappa\right] \quad \textrm{on}\; \Gamma. \label{eq:148}
\end{align}
The above equations \eqref{eq:147}--\eqref{eq:148} is equivalent to 
\begin{align*}
&a_{\kappa,\tau}(u^{\mathrm{tot}}_{\kappa}{|_\Omega}, \phi) = \left(v_{\kappa}^{\mathrm{in}}, \phi \right)_{H^1(\Omega)} \quad \textrm{for}\; \phi \in H^1(\Omega).
\end{align*}
Here, the sesquilinear form $a_{\kappa,\tau}\langle \cdot, \cdot\rangle$ is defined as \eqref{eq:66} and $v_{\kappa}^{\mathrm{in}} \in H^1(\Omega)$ is determined by 
\begin{align}
\left \langle v_{\kappa}^{\mathrm{in}}, \phi \right \rangle_{H^1(\Omega)} =  -\frac{\rho_1 \tau}{\rho_0} \left\langle \partial_\nu u^{\mathrm{in}}_\kappa - D_N(\kappa/c_0)\gamma u^{\mathrm{in}}_\kappa, \gamma \phi \right\rangle_{L^2(\Gamma)}, \quad \phi \in H^1(\Omega). \notag
\end{align}
Clearly, 
\begin{align} \label{eq:161}
\left\|v_\kappa^{\mathrm{in}}\right\|_{H^1(\Omega)} \le  C\left\|u_\kappa^{\mathrm{in}}\right\|_{H^1(\Omega)}.
\end{align}
Proceeding similarly to the derivation of \eqref{eq:149}, and using Lemmas \ref{le:1} and \ref{le:2}, we have
\begin{align}
u^{\mathrm{tot}}_{\kappa} & = \left[\mathbf e_{|z_0|} + O_{H^1(\Omega)}(\tau + |\kappa - z_0|)\right]\left(2(\kappa - z_0)z_0 \mathbb I + \tau \frac{\rho_1}{\rho_0} \mathcal M(z_0) + O\left(|\kappa-z_0| + \tau\right)^2\right)^{-1} \notag\\ 
&\quad\left[\left({\mathbf b}_{1}^{\mathrm{tot}}, \ldots, {\mathbf b}_{n_{|z_0|}}^{\mathrm{tot}}\right)^{T} + \tau O\left((|\kappa - z_0| + \tau)\|v^{\mathrm{in}}_\kappa\|_{H^1(\Omega)}\right)\right] + O_{H^1(\Omega)}\left(\tau\|v^{\mathrm{in}}_{\kappa}\|_{H^1(\Omega)}\right)\quad {\rm{in}}\; \Omega, \notag
\end{align}
\textrm{as} $\tau \rightarrow 0$ and $\kappa \rightarrow z_0$. Therefore, with the aid of \eqref{eq:109} and \eqref{eq:161}, we find
\begin{align} \label{eq:159}
u^{\mathrm{tot}}_{\kappa} = e^{\mathrm{tot}}_{|z_0|} + u_{\kappa,\mathrm{Res}}^{\mathrm{tot}} \quad \mathrm{in}\; \Omega,
\end{align}
where 
\begin{align}\label{eq:160}
&\left\|u^{\mathrm{tot}}_{\kappa, \mathrm{Res}}\right\|_{H^1(\Omega)} \le C \left(\tau + |\kappa -z_0|\right)\left\|u_\kappa^{\mathrm{in}}\right\|_{H^1(\Omega)}. 
\end{align}
Furthermore, we observe that 
\begin{align*}
&-DL_{\Omega,\kappa/c_0} u^{\mathrm{in}}_{\kappa}+ SL_{\Omega,\kappa/c_0}\left(\partial_\nu u^{\mathrm{in}}_{\kappa} - D_N(\kappa/c_0)\gamma u^{\mathrm{in}}_\kappa\right) + DL_{\Omega,\kappa/c_0}u^{\mathrm{in}}_{\kappa} =  R^N_\kappa u^{\mathrm{in}}_{\kappa}.
\end{align*}
From this, we can use \eqref{eq:152}, \eqref{eq:46}, \eqref{eq:148} and
\eqref{eq:159} to obtain expansion \eqref{eq:115}, where $u^{\mathrm{sc}}_{\kappa, \mathrm{Res}}$ is given by
\begin{align*}
u^{\mathrm{sc}}_{\kappa, \mathrm{Res}} &= \left(\frac{1}{c^2_1} - \frac{1}{c^2_0} \right)\kappa^2 N_{\Omega,\kappa/c_0}u^{\mathrm{tot}}_{\kappa, \mathrm{Res}} - SL_{\Omega,\kappa/c_0}  D_N(\kappa/c_0) \gamma u^{\mathrm{tot}}_{\kappa,\mathrm{Res}} + SL_{\Omega,\kappa/c_0} \partial_\nu u^{\mathrm{tot}}_{\kappa} \\
&+ \frac{\kappa^2 - z^2_0}{c^2_1} N_{\Omega, \kappa/c_0}e^{\mathrm{tot}}_{|z_0|} \;\; \mathrm{in}\; \Omega,
\end{align*}
whence \eqref{eq:116} follows from \eqref{eq:158}, \eqref{eq:148} and \eqref{eq:160}.
\end{proof}

\section{Asymptotic analysis for a microresonator} \label{sec:5}

This section is devoted to investigating the microresonator regime, where the size of $\Omega$ and the parameter $\tau$ are both very small. To facilitate its presentation, we introduce some new notations. Let $y_0$ be any fixed point in $\R^3$. For any $\vep>0$, define 
\begin{align} \label{eq:12}
\Omega_{\vep}(y_0):= \{x: x=y_0+\vep(y-y_0), \; y\in \Omega\} \; \textrm{and}\; \Gamma_{\vep}(y_0):= \partial \Omega_{\vep}(y_0),
\end{align}
and the corresponding Hamiltonian $H_{\rho_\tau, k_\tau}(\Omega_\vep(y_0))$ associated with $\Omega_\vep(y_0)$ is denoted by $H_{\tau,\vep}(y_0)$. Define
\begin{align*}
B_1(y_0):= \{y \in \R^3: |y-y_0|<1\}.
\end{align*}
For every $g\in L_{\mathrm{loc}}^2(\R^3)$, we define its scaled function as
\begin{align*}
\widetilde g(x):= g(y_0 + \vep(x-y_0)), \quad \mathrm{for}\; x\in \R^3.
\end{align*}
 Furthermore, we present the following lemma, whose proof is provided in section \ref{sec:a22}.

\begin{lemma} \label{le:12}
Let $\vep >0$ and $\beta > 1/2$. Assume that $y_0$ is any point in $\R^3$ and $\omega$ is a nonzero real number. The following arguments hold true.

\begin{enumerate}[(a)]

\item \label{z1} Let $\alpha>0$. For any $f \in L^2_{\beta}(\R^3)$, we have
\begin{align} \label{eq:110}
&\left\|\int_{\R^3}\left[\frac{e^{i(\omega\vep^{-1} + \omega\vep^{\alpha-1})|y_0 + \vep(x-y_0)- y|}}{|y_0 + \vep(x-y_0) - y|} - \frac{e^{i{\omega}|y_0-y|/\vep}}{4\pi|y_0-y|}e^{i{\omega}\hat y_{y_0}\cdot(y_0-x)}\right]f(y) dy \right\|_{H^1(\Omega)}\notag\\
&\qquad \qquad \qquad \qquad \qquad \qquad \qquad \qquad \qquad  \le C\max\left(\vep^{\alpha}, \vep^{1/3}\ln \vep\right)\|f\|_{L_{\beta}^2(\R^3)}.
\end{align}

\item \label{z2} Let $\alpha>1$. For $\phi \in L^2(\Omega)$, we have
\begin{align} 
&\left\|\int_{\Omega}\left[\frac{e^{i(\omega + \omega\vep^{\alpha})|y_0+ \vep^{-1}(x-y_0)- y|}}{4\pi|y_0+ \vep^{-1}(x-y_0)- y|}- \vep \frac{e^{i{\omega}|x-y_0|/\vep}}{4\pi|x-y_0|}e^{i\omega\hat x_{y_0}\cdot(y_0-y)}\right]\phi (y)d y\right\|_{L_{-\beta}^2(\R^3)}\notag\\
&\qquad \qquad \qquad \qquad \qquad \quad \quad \qquad \quad \qquad \le C \max\left(\vep^{\alpha + 1}, \vep^{4/3}\ln\vep \right)\|\phi\|_{L^2(\Omega)}. \label{eq:105}
\end{align}
Furthermore, for $\psi \in H^{-1/2}(\Gamma)$, we have
\begin{align}
&\left\|\int_{\Gamma}\left[\frac{e^{i(\omega + \omega\vep^{{\alpha}})|y_0 + \vep^{-1}(x-y_0)- y|}}{4\pi|y_0+ \vep^{-1}(x-y_0)- y|}-\vep\frac{e^{i{\omega}|x-y_0|/\vep}}{4\pi|x-y_0|}e^{i\omega\hat x_{y_0}\cdot(y_0-y)}\right]\psi(y)d\sigma(y)\right\|_{L_{-\beta}^2(\R^3)}\notag\\
&\qquad \qquad \qquad \qquad \qquad \qquad \quad \qquad \quad \quad \le C\max\left(\vep^{\alpha + 1}, \vep^{4/3}\ln\vep \right)\|\psi\|_{H^{-1/2}(\Gamma)}. \label{eq:108}
\end{align}
\end{enumerate}
Here, for each $x\in \R^3\setminus\{y_0\}$, the direction from $x$ to $y_0$ is denoted by
\begin{align} \label{eq:173}
\hat x_{y_0}:=
\frac{x-y_0}{|x-y_0|},
\end{align}
and $C$ is a positive constant independent of $f$, $\vep$ and $\alpha$.
\end{lemma}

\subsection {Asymptotics of the resolvents in the presence of a microresonator}

We observe from \eqref{eq:141}--\eqref{eq:142} that for each resonant state $u^{y_0}_{\lambda_{\tau,\vep}}$ of the Hamiltonian $H_{\tau,\vep}(y_0)$ corresponding to the resonance $\lambda^{y_0}_{\tau,\vep}$, its scaled function $\widetilde u^{y_0}_{\lambda_{\tau,\vep}}$ is a resonant state of $H_{\rho_\tau,\rho_\tau}(y_0)$ associated with the resonance $\vep \lambda^{y_0}_{\tau,\vep}$.
Therefore, based on Theorem \ref{th:3}, we can derive the asymptotic behavior of the resolvent $R_{H_{\tau,\vep}(y_0)}(\omega)$ as $\tau$ and $\vep$ both tend to $0$, in the regime where the scaled frequency $\vep \omega$ is close to $z_0$, with $z^2_0/c^2_1$ being a Neumann eigenvalue of the domain $\Omega$.

\begin{theorem} \label{th:6}
Let $\omega \in \R$, $\tau, \vep \in \R_+$ and $\beta> 1/2$. Assume that $z_0 \in \R \backslash \{0\}$ is a point at which $z^2_0/c^2_1$ is a Neumann eigenvalue of the domain $\Omega$, and that the corresponding orthogonal eigenfunctions with respect to the scalar product $\langle \cdot, \cdot\rangle_{L^2(\Omega)}$ are denoted by $e^{(1)}_{|z_0|},\ldots e^{(n_{|z_0|})}_{|z_0|}$, where $n_{|z_0|}\in \mathbb N$. 
Given a point $y_0\in \R^3$, if $|\vep \omega - z_0| = \vep^\alpha$ with $\alpha \in \R_+$, then for any $f\in L_{\mathrm{comp}}^2(\R^3) \cap H^2(B_1(y_0))$, we have 
\begin{align}
&\left(R_{H_{\tau,\vep}(y_0)}(\omega) f\right)(x) = -\frac{1}{c_0^2} \left(R_{\omega/c_0} f\right)(x) \notag\\ 
&+ \frac{\vep e^{i\frac{z_0}{\vep c_0}|x-y_0|}}{4\pi|x-y_0|} e^{i\frac{z_0}{c_0}\hat x_{y_0} \cdot y_0} \left(\mathcal S_{z_0}^{\infty,D}\left[\gamma e^{f,0}_{|z_0|} + \gamma e^{f,1}_{|z_0|} + \gamma e^{f,2}_{|z_0|}- \gamma \hat v_{z_0}^f\right]\right)(\hat x_{y_0}) + R_{\mathrm{Res}}(\omega) f(x) \label{eq:45}
\end{align}
with
\begin{align*}
&\|R_{\mathrm{Res}}(\omega) f\|_{L^2_{-\beta}(\R^3)} \le C \max\left(\tau\vep, \vep^{\alpha+1}, \vep^{4/3}\ln \vep\right)\left(\|f\|_{L^2{(\R^3)}} + \|f\|_{H^2(B_1(y_0))} + \|e_{|z_0|}^{f,1}\|_{H^1(\Omega)}\right),
\end{align*}
as $\vep,\tau\rightarrow 0$,
where $C$ is a positive constant independent of $f$, $\omega$, $\vep$ and $\tau$, $S_{z_0}^{\infty,D}$ and $\hat x_{y_0}$ are specified in \eqref{eq:162} and \eqref{eq:173}, respectively, and $e^{f,j}_{|z_0|}$ $(j = 0,1,2)$ are defined as 
\begin{align}
& e^{f,j}_{|z_0|}:= \mathbf e_{|z_0|}\left(\mathrm{sgn}(\vep\omega-z_0)2 z_0\vep^\alpha \mathbb I + \tau \frac{\rho_1}{\rho_0} \mathcal M(z_0)\right)^{-1} {\mathbf b}^{f,j} \quad \mathrm{in}\; \Omega. \notag
\end{align}
Here, $\mathbf e_{|z_0|}$ is defined in \eqref{eq:96} and ${\mathbf b}^{f,j}$ are  column vectors whose $l-$th entry ($l=1,\ldots,n_{|z_0|}$) given by 
\begin{align*}
&{\mathbf b}_l^{f,0}:= \mathcal P^N(|z_0|)\hat v_{z_0}^f, \\
&{\mathbf b}_l^{f,1}:= \frac{-\vep^2}{c^{2}_1}\left\langle f(y_0), e^{(l)}_{|z_0|}\right\rangle_{L^2(\Omega)},\\
&{\mathbf b}_l^{f,2}= -\mathrm{sgn}(\vep\omega-z_0)\frac{2z_0\vep^\alpha }{c^2_1}\left\langle \hat v_{z_0}^f, e^{(1)}_{|z_0|}\right\rangle_{L^2(\Omega)}\\
& \qquad\quad - \frac{\rho_1 \tau}{\rho_0}\left\langle \partial^+_\nu R^{\mathrm{ex}}_{z_0}\left(\mathcal P^{N}(|z_0|) \hat v_{z_0}^f - \hat v_{z_0}^f\right),e^{(l)}_{|z_0|}\right\rangle_{L^2(\Gamma)},
\end{align*}
where $\mathcal P^N(|z_0|)$ is specified in \eqref{eq:27} and $\hat v_{z_0}^{f}$ is defined as 
\begin{align}
\hat v_{z_0}^f(x):= \frac{1}{c^2_0}\int_{\R^3} \frac{e^{i\frac{z_0}{\vep c_0}|y_0-y|}}{4\pi|y_0-y|}e^{i\frac{z_0}{c_0}\hat y_{y_0}\cdot(y_0-x)}f(y) dy, \quad x\in \R^3. \notag
\end{align}
\end{theorem}

\begin{proof}
Let $\vep$ be sufficiently small throughout the proof. We denote $u^f_{\omega,\vep}:= R_{H_{\tau,\vep}}(\omega)f$ and $v^f_{\omega}:= -{c_0^{-2}}R_{\omega/c_0} f$.
It is easy to verify that 
\begin{align}
\widetilde u^f_{\omega,\vep} = u^{\vep^2 f}_{\vep\omega,\tau} \;\; \textrm{and} \;\; \widetilde v^f_{\omega} = v^{\vep^2 f}_{\vep\omega}, \notag
\end{align}
where $u^{\vep^2 f}_{\vep\omega,\tau} $ and $v^{\vep^2 f}_{\vep\omega}$ are given by \eqref{eq:72} and \eqref{eq:81}, respectively.
Therefore, using Theorem \ref{th:3} for the case $f = \vep^2 \widetilde f$, we have
\begin{align}
\widetilde u^f_{\omega,\vep} - \widetilde v^f_\omega = \sum^4_{j=1}R^{(j)}_{H_{\rho_\tau,k_\tau}}(\vep\omega)\left(\vep^2 \widetilde f\right)=: \sum^4_{j=1}\widetilde u^{f,j}_{\omega,\vep}. \notag
\end{align}
When $|\vep \omega - z_0| = \vep^{\alpha}$, utilizing statement \eqref{z1} of Lemma \ref{le:12}, we have 
\begin{align}\label{eq:111}
&\left\|v^{\vep^2 f}_{\vep\omega} - \hat v_{z_0}^f\right\|_{H^1(\Omega)} \le C\max\left(\vep^{\alpha}, \vep^{1/3}\ln \vep\right)\|f\|_{L^2\left(\R^3\right)} \quad \textrm{for}\; f \in L_{\mathrm{comp}}^2(\R^3).
\end{align}

In the sequel, we estimate the asymptotic behavior of $\widetilde u^{f,j}_{\omega,\vep}$ with $j=1,2,3,4$. For the estimate of $\widetilde u^{f,1}_{\omega,\vep}$, combining \eqref{eq:89}, \eqref{eq:90}, \eqref{eq:111} and the fact that 
\begin{align*}
\|\widetilde f - f(y_0)\| \le C\vep^{1/2}\|f\|_{H^2(B_1(y_0))}
\end{align*}
gives
\begin{align} \label{eq:112}
w^{\vep^2 \widetilde f}_{0,0} = \mathcal P^{N}(|z_0|) \hat v_{z_0}^f - \hat v_{z_0}^f + w^{\mathrm{Res}}_{0,0}
\end{align}
with 
\begin{align}\label{eq:113}
\left\|w^{\mathrm{Res}}_{0,0}\right\|_{H^1(\Omega)} \le C\max\left(\vep^{\alpha}, \vep^{1/3}\ln \vep\right)\left[\|f\|_{L^2\left(\R^3\right)} + \|f\|_{H^2\left(B_1(y_0)\right)}\right].
\end{align}
We note that 
\begin{align*} 
\Delta \hat v^f_{z_0} + z_0^2/c_0^2 \hat v^f_{z_0} = 0\quad \mathrm{in}\;\R^3.
\end{align*}
In conjunction with \eqref{eq:78}, \eqref{eq:151}, \eqref{eq:111}, \eqref{eq:112}, \eqref{eq:113} and Lemma \ref{le:12}, we obtain
\begin{align*}
\widetilde u^{f,1}_{\omega,\vep}(y_0 + \vep^{-1}(x-y_0)) & = 
\frac{\vep e^{i\frac{z_0}{\vep c_0}|x-y_0|}}{4\pi|x-y_0|} \bigg[\int_{\Gamma} \partial_{\nu(y)} e^{\frac{iz_0}{c_0} \hat x_{y_0}\cdot (y_0-y)}\bigg[\gamma(e^{f,0}_{|z_0|} -\hat v^f_{z_0})\bigg](y) \\
&- e^{\frac{iz_0}{c_0} \hat x_{y_0}\cdot (y_0-y)}\big[D_N(z_0/c_0)\gamma(e^{f,0}_{|z_0|} - \hat v^f_{z_0})(y)\big] d\sigma(y)\bigg] + \widetilde u_{\mathrm{Res}}^{f,1}(x) \\ 
&=\frac{\vep e^{i\frac{z_0}{\vep c_0}|x-y_0|}}{4\pi|x-y_0|}  e^{\frac{iz_0}{c_0}\hat x_{y_0} \cdot y_0}  \mathcal S_{z_0}^{\infty,D}\left[\gamma e^{f,0}_{|z_0|} - \gamma \hat v^f_{z_0}\right](\hat x_{y_0}) + \widetilde u_{\mathrm{Res}}^{f,1}(x)
\end{align*}
with $\widetilde u_{\mathrm{Res}}^{f,1}$ satisfying
\begin{align*}
\left\|\widetilde u_{\mathrm{Res}}^{f,1}\right\|_{L^2_{-\beta}(\R^3)} \le C \max\left(\vep^{\alpha+1}, \vep^{4/3}\ln \vep\right)\left[\|f\|_{L^2\left(\R^3\right)} + \|f\|_{H^2\left(B_1(y_0)\right)}\right].
\end{align*}
Similarly, utilizing \eqref{eq:140}, \eqref{eq:109}, \eqref{eq:152}, \eqref{eq:111}, \eqref{eq:112}, \eqref{eq:113} and Lemma \ref{le:12}, we have 
\begin{align*}
\sum^3_{j=2}\widetilde u^{f,j}_{\omega,\vep}(y_0 + \vep^{-1}(x-y_0)) = \frac{\vep e^{i\frac{z_0}{\vep c_0}|x-y_0|}}{4\pi|x-y_0|} e^{\frac{iz_0}{c_0}\hat x_{y_0} \cdot y_0} \mathcal S_{z_0}^{\infty,D}\left[\gamma e^{f,1}_{|z_0|} + \gamma e^{f,2}_{|z_0|}\right] (\hat x_{y_0})+ \widetilde u_{\mathrm{Res}}^{f,(2,3)}(x),
\end{align*}
where $\widetilde u_{\mathrm{Res}}^{f,(2,3)}$ satisfies
\begin{align*}
\left\|\widetilde u_{\mathrm{Res}}^{f,(2,3)}\right\|_{L^2_{-\beta}(\R^3)} &\le C \max\left(\tau\vep, \vep^{\alpha+1}, \vep^{4/3}\ln \vep\right)\left(\|f\|_{L^2{(\R^3)}} + \|f\|_{H^1(B_1(y_0))} + 
\left\|e_{|z_0|}^{f,1}\right\|_{H^1(\Omega)}\right).
\end{align*}
Moreover, based on \eqref{eq:157}, applying Lemma \ref{le:12} and using \eqref{eq:111}, \eqref{eq:112}, \eqref{eq:113} again, we arrive at 
\begin{align*}
\left\|\widetilde u_{\mathrm{Res}}^{f,4}\right\|_{L^2_{-\beta}(\R^3)} &\le  C \max\left(\tau\vep, \vep^{\alpha+1}\right)\left(\|f\|_{L^2{(\R^3)}} + \|f\|_{H^1(B_1(y_0))} + \left\|e_{|z_0|}^{f,1}\right\|_{H^1(\Omega)}\right).
\end{align*}
Building upon the above asymptotic estimates of $\widetilde u^{f,j}_{\omega,\vep}$ with $j=1,2,3,4$, we obtain the assertion of this theorem. 
\end{proof}

\subsection {Asymptotics of the scattered fields in the presence of a microresonator}

Following the approach adopted in the proof of Theorem \ref{th:6} and making use of Theorem \ref{th:5}, we can establish the asymptotics of the scattered field in the presence of the microresonator $\Omega_\vep(y_0)$.

\begin{theorem} \label{th:4}
Let $\omega, \tau, \vep \in \R_+$, $\beta> 1/2$, and let $d$ denote the propagation direction of the incident plane wave $u^{\mathrm{in}}$. Assume that $z_0 \in \R_+$ is a point at which $z^2_0/c^2_1$ is a Neumann eigenvalue of the domain $\Omega$, and that the corresponding orthogonal eigenfunctions with respect to the scalar product $\langle \cdot, \cdot\rangle_{L^2(\Omega)}$ are denoted by $e^{(1)}_{|z_0|},\ldots e^{(n_{|z_0|})}_{|z_0|}$, where $n_{|z_0|}\in \mathbb N$. Given a point $y_0 \in \R^3$, if $|\vep \omega - z_0| = \vep^\alpha$ with $\alpha \in \R_+$, then we have
\begin{align}
u_{\omega}^{\mathrm{sc}}(x) = \frac{\vep e^{i\frac{z_0}{\vep c_0}|x-y_0|}}{4\pi|x-y_0|}u_{z_0/\vep}^{\mathrm{in}}(y_0) e^{i\frac{z_0}{c_0}(\hat x_{y_0} - d)\cdot y_0} \mathcal S_{z_0}^{\infty,D}\left(\gamma e^{\mathrm{tot}}_{|z_0|} - \gamma u_{z_0}^{\mathrm{in}}\right)(\hat x_{y_0}) + u^{\mathrm{sc}}_{\kappa, \mathrm{Res}}(x) \label{eq:5}
\end{align}
with
\begin{align}
&\left\|u^{\mathrm{sc}}_{\kappa, \mathrm{Res}}\right\|_{L^2_{-\beta}(\R^3)} \le C \max\left(\tau\vep, \vep^{\alpha+1}, \vep^{4/3}\ln \vep\right) \left\|u_{z_0}^{\mathrm{in}}\right\|_{H^1(\Omega)}, \notag
\end{align}
{as} $\vep, \tau \rightarrow 0$, where $C$ is positive constant independent of $\omega$, $\vep$ and $\tau$, $S_{z_0}^{\infty,D}$ and $\hat x_{y_0}$ are specified in \eqref{eq:162} and \eqref{eq:173}, respectively, and
\begin{align*}
e^{\mathrm{tot}}_{|z_0|}:= \mathbf e_{|z_0|}\left(\mathrm{sgn}(\vep\omega-z_0)2 z_0\vep^\alpha \mathbb I + \tau \frac{\rho_1}{\rho_0} \mathcal M(z_0)\right)^{-1} {\mathbf b}^{\mathrm{tot}} \quad \mathrm{in}\; \Omega. \notag
\end{align*}
Here, $\mathbf e_{|z_0|}$ is specified in \eqref{eq:96} and the $l-$th entry of the column vector ${\mathbf b}^{\mathrm{tot}}$ are given by
\begin{align*}
&{\mathbf b}_{l}^{\mathrm{tot}} = -\frac{\rho_1 \tau}{\rho_0}\left\langle \partial_\nu u_{z_0}^{\mathrm{in}} - \partial^+_\nu R^{\mathrm{ex}}_{z_0} u_{z_0}^{\mathrm{in}}, e^{(l)}_{|z_0|}\right\rangle_{L^2(\Gamma)},\quad l=1,\ldots, n_{|z_0|}.
\end{align*}
\end{theorem}

\begin{remark}\label{Anisotropic-small-scatterers} We state the following two comments.

\begin{enumerate}
\item Small scatterers are called isotropic point-scatterers if the generated dominant field (which is related to the one of the equivalent point-like scatterer) has an amplitude which is independent on the directions of incidence or propagation. If not, we call them anisotropic small scatterers. To the best of our knowledge, the low-frequency microresonators are all isotropic. The expansion in (\ref{eq:5}) shows that at high frequencies, the microresonator is anisotropic, via the terms $e^{i\frac{z_0}{c_0}(\hat x_{y_0} - d)\cdot y_0}$ and $\mathcal S_{z_0}^{\infty,D}\left(\gamma e^{\mathrm{tot}}_{|z_0|} - \gamma u_{z_0}^{\mathrm{in}}\right)(\hat x_{y_0})$. Near the Minnaert resonance (which is a low-frequency resonance), the scatterer behaves as an isotropic one, see e.g., \cite{AZ18, DGS-21}. Therefore, the anisotropy is created at high frequency resonances. Such a feature may open new possibilities for designing anisotropic metamaterials from simple configurations of a single microresonator. In recent years, the study of high frequency homogenization has attracted significant interest due to its practical potential implications. We believe that the expansion derived in (\ref{eq:5}) is a good first step in achieving such a goal.
\item Observe that the expansion (\ref{eq:5}) relates the field $u_{\omega}^{\mathrm{sc}}(x)$, generated by the small scaled scatterer $\Omega_{\vep}(y_0)$, to the far-filed pattern $\mathcal S_{z_0}^{\infty,D}\left(\gamma e^{\mathrm{tot}}_{|z_0|} - \gamma u_{z_0}^{\mathrm{in}}\right)(\hat x_{y_0})$ of the extended obstacle $\Omega$, see (\ref{eq:12}). Such a relation suggests that by measuring the field $u_{\omega}^{\mathrm{sc}}(x)$ away from the small scatterer $\Omega_{\vep}(y_0)$, we can have access to the far-field pattern of $\Omega$. Therefore, this relation is expected to be useful in designing inversion strategies, not only to localize the small scatterer and estimate its size (as it is done in the traditional inverse problems community using MUSIC type algorithms for instance), but also recover the details of the scatterer's shape $\partial \Omega$. Indeed, as we can recover the far-field pattern, then we can us the machinery already developed in the literature (e.g., \cite{DK19}) to recover shapes from far-field measurements. 

\end{enumerate}
\end{remark}
\subsection{Large time behavior of a microresonator}

In this subsection, we concern the wave evolution of a acoustic microresonator in the time domain. Specifically, given a causal source $F(x,t)$, we consider the following model
\begin{align}
& \frac{1}{k_{\tau,\vep}}\partial_{tt} u_{\tau,\vep} -  \nabla \cdot \frac{1}{\rho_{\tau,\vep}} \nabla u_{\tau,\vep} = F  \quad \textrm{in}\; \R^3 \times \R,  \label{eq:101}\\
& u(x,0) = 0, \quad \partial_t u(x,0) = 0, \quad \textrm{for}\; x\in \R^3 \label{eq:102}
\end{align}
describing the wave propagation in the presence of the microresonator. Here, for each $\tau, \vep >0$, the mass density $\rho_{\tau,\vep}$ and the bulk modulus $k_{\tau,\vep}$ satisfy
\begin{align*}
\rho_{\tau,\vep}(x) := 
\begin{cases}
\rho_0,        &x \in \R^3 \backslash \Omega_\vep(y_0), \\
{\rho_1}\tau,  & x \in \Omega_\vep(y_0),
\end{cases}
\quad \textrm{and} \quad
k_{\tau,\vep}(x):=
\begin{cases}
k_0,      & x \in \R^3 \backslash \Omega_\vep(y_0), \\
k_1\tau,  & x \in \Omega_\vep(y_0),
\end{cases}
\end{align*}
where $\Omega_\vep(y_0)$ is specified in \eqref{eq:12}. We are interested in the large time behavior of $u_{\tau,\vep}$ when $\tau$ and $\vep$ are both small enough. To do so, we first establish a representation of $u_{\tau,\vep}$ in terms of $F$. 

\begin{lemma} \label{le:9}
Let $M\in \mathbb N$. Given $\tau,\vep >0$, let $u_{\tau,\vep}$ be the solution of equations \eqref{eq:101}--\eqref{eq:102} and $H_{\tau,\vep}(y_0)$ be the Hamiltonian $H_{\rho_\tau, k_\tau}(\Omega_\vep(y_0))$ associated with $\Omega_\vep(y_0)$. Suppose that $\mathcal K$ is any compact set in $\R^3$. For every $F \in C_c^{M}\left(\R_+;L_{\mathrm{comp}}^2(\R^3)\right)$, with $M \in \mathbb N$ and $\mathrm{supp}_x F \in \mathcal K$, we have
\begin{align*}
u_{{\tau,\vep}}(x,t) = \int_{\Gamma_\sigma}\frac{e^{-i\lambda t}}{2\pi}\left(R_{H_{\tau,\vep}(y_0)} (\lambda) k_{\tau,\vep}\hat F (\cdot, -i\lambda)\right)(x) d\lambda, \quad t>0, \; x\in \mathcal K.
\end{align*}
Here, 
\begin{align}\label{eq:163}
\Gamma_\sigma:=\{z\in \CC: {\rm{Im}}(z) = \sigma\}\; \mathrm{with}\; \sigma >0,
\end{align}
and $\hat F$ is a Fourier-Laplace transform of $F$, which is defined by \eqref{eq:125}. 
\end{lemma}

The proof of Lemma \ref{le:9} can be established by using similar arguments as in \cite[Theorem 2.2]{SS-2024}.

In the sequel, we provide an asymptotic characterization of the resonant projection corresponding to the Hamiltonian $H_{\rho_\tau, k_\tau}(\Omega)$ at the Minnaert resonances.

\begin{lemma} \label{le:6} 
Let $z_{\pm}(\tau)$ be two scattering resonances of $H_{\rho_\tau, k_\tau}$ around $0$, satisfying the asymptotic expansion \eqref{eq:65}. Then, we have 
\begin{align} \label{eq:169}
m_R({z_{\pm}(\tau)}) = 1,
\end{align}
namely, $z_{\pm}(\tau)$ are simple, 
where the definition of $m_R({z_{\pm}(\tau)})$ is specified in \eqref{eq:118}. Moreover, for any $f\in L^2_{{\rm{comp}}}(\R^3)$, we have 
\begin{align} \label{eq:134}
\left(\Pi_{H_{\rho_\tau,k_\tau}}({z_{\pm}(\tau))} f\right)(x) = |c_{{\mathrm{norm}}}|^2 u_\pm(x) \int_{\R^3} \overline{u_{\pm}(x)} \frac{f(x)}{k_\tau(x)} dx, \quad x\in \R^3,
\end{align}
where $c_{\mathrm{norm}} \in \CC$ is a normalized constant satisfying
\begin{align} \label{eq:130}
|c_{{\mathrm{norm}}}|^2 = \frac{\tau k_1}{|\Omega|} + O(\tau^{3/2}),
\end{align}
and 
\begin{align}
&u_{{\pm}}(x) = 
-\left(SL_{\Gamma,z_{\pm}(\tau)/c_0}(D_N(0)1 + {\mathrm{Rem}}^{\pm})\right)(x), \quad x\in \R^3 \backslash \Gamma, \label{eq:132}
\end{align}
with 
\begin{align}
\left\|{\mathrm{Rem}}^{\pm}\right\|_{L^2(\Gamma)} \le C\tau^{1/2}, \quad \mathrm{as}\; \tau\rightarrow 0. \label{eq:133}
\end{align} 
Here, $C$ is a positive constant independent of $\tau$ and $f$.

\end{lemma}

\begin{proof}

We note that $\mathcal A(z,\tau)$ is also a Fredholm operator with index $0$ with respect to $z-$ variable (see formula \eqref{eq:55}). This, together with \eqref{eq:170}, Lemma \ref{le:10}, Remark \ref{re:1} and Theorem \ref{th:a2} yields that there exists $\delta_0 > 0$ such that when $\tau \in (0,\delta_0)$, $\mathcal A(z,\tau)$ has the following expansions near $z_{\pm}(\tau)$,
\begin{align} \label{eq:69}
& \mathcal A(z,\tau) = E_{z_{\pm}(\tau)}(z) \left(\mathcal Q(z_{\pm}(\tau)) + (z- z_{\pm}(\tau))\mathcal P(z_{\pm}(\tau))\right) F_{z_{\pm}(\tau)}(z).
\end{align}
Here, $\mathcal P(z_{\pm}(\tau))$ are one-dimensional projections, $\mathcal Q(z_{\pm}(\tau)) = \mathbb I - \mathcal P(z_{\pm}(\tau))$, and $E_{z_{\pm}(\tau)}(z)$ and $F_{z_{\pm}(\tau)}(z)$ are both holomorphic and invertible near $z_{\pm}(\tau)$, respectively. With the aid of \eqref{eq:62} and \eqref{eq:69}, we obtain 
\begin{align}
R_{H_{\rho_\tau, k_\tau}}(\lambda)  = -\frac{\Pi_{H_{\rho_\tau,k_\tau}}({z_{\pm}(\tau)})}{\lambda^2-z_{\pm}^2(\tau)} + B(\lambda,{z_{\pm}(\tau)})\quad \textrm{in the neighborhood of}\; {z_{\pm}(\tau)}, \label{eq:122}
\end{align}
where $\lambda \rightarrow B(\lambda, z_{\pm}(\tau))$ is holomorphic near ${z_{\pm}(\tau)}$.  Therefore, we obtain \eqref{eq:169}.

In the sequel, we aim to characterize the projection $\Pi_{H_{\rho_\tau,k_\tau}}({z_{\pm}(\tau))}$. To achieve this, we divide the rest of the proof into the following two steps.

\textbf{Step 1}: In this step, we aim to find the basis in the range of $\Pi_{H_{\rho_\tau,k_\tau}}({z_{\pm}(\tau))}$. It follows from \eqref{eq:169} that each element in ${\rm{Ran}}\left(\Pi_{H_{\rho_\tau,k_\tau}}({z_{\pm}(\tau))}\right)$ is a resonant state. From this, utilizing Remark \ref{re:0}, for each element $u_{z_{\pm}(\tau)} \in {\rm{Ran}}\left(\Pi_{H_{\rho_\tau,k_\tau}}({z_{\pm}(\tau))}\right)$, we have that the value of $u_{z_{\pm}(\tau)}$ within $\Omega$ and the multiplication of the constant $\rho_0/(\rho_1\tau)$ and its normal derivative on $\Gamma$ constitute a nonzero solution of $\mathcal A(z_{\pm}(\tau),\tau)$. Here, we note that
\begin{align*} 
\textrm{dim}\left(\textrm{Ker}\left(\mathcal A(z_{\pm}(\tau),\tau)\right)\right) = 1.
\end{align*}
Let $h(z_{\pm}(\tau))$ be the element of $\textrm{Ker}\left(\mathcal A(z_{\pm}(\tau), \tau)\right)$, i.e., 
\begin{align}\label{eq:60}
&\left(\mathcal A_0(z_{\pm}(\tau)) + \frac{\rho_1\tau}{\rho_0}\mathcal A_1(z_{\pm}(\tau))\right) h(z_{\pm}(\tau)) = \left(\mathcal A_0(0) + \mathcal A_{\rm{res}}(z_{\pm}(\tau)) \right) h(z_{\pm}(\tau)) = 0,
\end{align}
where 
\begin{align*}
\mathcal A_{\rm{res}}(z_{\pm}(\tau)):= \mathcal A_0(z_{\pm}(\tau)) - A_0(0) + \frac{\rho_1\tau}{\rho_0} \mathcal A_1(z_{\pm}(\tau)).
\end{align*}
In conjunction with inequality \eqref{eq:65}, \eqref{eq:70} and the analyticity of $\mathcal A_0(z)$ and $A_1(z)$ with respect to $z$, we have 
\begin{align} \label{eq:61}
\left\|\mathcal A_{\rm{res}}(z_\pm(\tau))\right\|_{L^2(\Omega)\times L^2(\Gamma)} \le C \sqrt{\tau}.
\end{align}
From \eqref{eq:60}, we find
\begin{align}
&\mathcal Q^O(0) \left[\mathcal A_0(0) + \mathcal A_{\rm{res}}(z_{\pm}(\tau))\right]\mathcal Q^O(0)\left(h(z_{\pm}(\tau))\right) = -\mathcal Q^O(0) \mathcal A_{\rm{res}}(z_{\pm}(\tau)) \mathcal P^O(0)\left(h(z_{\pm}(\tau))\right). \notag
\end{align}
Here, $\mathcal P^O(0)$ is an orthogonal projection onto ${\rm{Ker}}\left(\mathcal A_0(0)\right)$, $\mathcal Q^O(0):= \mathbb I - \mathcal P^O(0)$. We note that, due to \eqref{eq:61} and the self-adjointness of $\mathcal A_0(0)$, the inverse of $\mathcal Q^O(0)\left(\mathcal A_0(0) +  \mathcal A_{\rm{res}}(z_{\pm}(\tau))\right)$ exists in $\mathcal L\left(\rm{Ran}(\mathcal Q^O(0))\right)$ for sufficiently small $\tau>0$. Therefore, we obtain 
\begin{align}
\left\|h(z_{\pm}(\tau)) - \mathcal P^O(0)\left(h(z_{\pm}(\tau))\right)\right\|_{L^2(\Omega)\times L^2(\Gamma)} \le C \sqrt \tau \left\|\mathcal P^O(0)\left(h(z_{\pm}(\tau))\right)\right\|_{L^2(\Omega)\times L^2(\Gamma)}. \label{eq:76}
\end{align}
Recall that, by \eqref{eq:67}, \eqref{eq:68} and \eqref{eq:84}, ${\rm{Ker}}\left(\mathcal A_0(0)\right) = \textrm{Span}\{(1,D_N(0) 1)\}$. This, together with \eqref{eq:76} implies that
\begin{align}
{\rm{Ker}}\left(\mathcal A(z_{\pm}(\tau),\tau)\right) = \textrm{Span}\{(1,D_N(0) 1)+ \left(\textrm{Res}_{1}^{\pm}(\tau), \textrm{Res}_{2}^{\pm}(\tau)\right)\} \notag
\end{align}
with
\begin{align} \label{eq:127}
\left\|{\rm{Res}_1}^{\pm}\right\|_{L^2(\Omega)} \le C\sqrt{\tau}\; {\rm{and}}\; \left\| {\rm{Res}}_2^{\pm}\right\|_{L^2(\Gamma)} \le C\sqrt{\tau}.
\end{align}
Therefore, with the aid of \eqref{eq:77}, the range of $\Pi_{z_{\pm}(\tau)}$ is spanned by the following element
\begin{align}
&u^b_{{\pm}}(x) = \left(\frac{1}{c^2_1} - \frac{1}{c^2_0} \right)\left(z_{\pm}(\tau)\right)^2 \left(N_{\Omega,z_{\pm}(\tau)/c_0} (1 + {\rm{Res}}_1^{\pm})\right)(x) \notag \\
&\quad\quad \quad -\left(1-\frac{\rho_1\tau}{\rho_0}\right)\left(SL_{\Gamma,z_{\pm}(\tau)/c_0}(D_N(0)1 + {\rm{Res}}_2^{\pm})\right)(x), \quad x\in \R^3 \backslash \Gamma,  \label{eq:63}
\end{align}
with $\textrm{Res}^\pm_1$ and $\textrm{Res}^\pm_2$ satisfying \eqref{eq:127}.

\textbf{Step 2}. In this step, we prove \eqref{eq:134} holds with $c_{\mathrm{norm}}$ and $u_\pm$ satisfying \eqref{eq:130} and \eqref{eq:132}, respectively. Due to the self-adjointness of $H_{\rho_\tau, k_\tau}$ in $\mathcal H$, the meromorphic properties of $R_{\rho_\tau,k_\tau}$ and $m_R({z_{\pm}(\tau)}) = 1$, for any $f\in L^2_{{\rm{comp}}}(\R^3)$, we have
\begin{align} \label{eq:129}
\left(\Pi_{H_{\rho_\tau,k_\tau}}({z_{\pm}(\tau))}f\right)(x) = |c_{{\textrm{norm}}}|^2u^b_\pm(x) \int_{\R^3} \overline{u^b_{\pm}(x)} \frac{f(x)}{k_\tau(x)} dx, \quad x\in \R^3.
\end{align}
We proceed to determine the asymptotics of this constant. From Remark \eqref{re:0}, there exist $g_{\pm}\in L^2_{\textrm{comp}}(\R^3)$ and $r_0>0$ such that 
\begin{align*}
u^b_{\pm}|_{\R^3 \backslash B_{r_0}} = R_{z_\pm(\tau)/c_0}g_{\pm}|_{\R^3\backslash B_{r_0}}.
\end{align*}
We choose a function $\chi \in \CC^{\infty}(\R^3)$ such that $\chi(x) = 1$ for $x \in \R^3\backslash B_{r_2}$, $\chi(x) = 0$ for $x \in B_{r_1}$ and $u^b_\pm = R_{z_\pm(\tau)/c_0}g_{\pm}$ on the support of $\chi$. Here, $r_1 < r_0 < r_2$ and $\Omega \subset B_{r_1}$. Then, we define 
\begin{align*}
h^*_{\pm} = u^b_{\pm} - \chi R_{z_\pm(\tau)/c_0}g_{\pm}\;\;\textrm{and}\;\; h_{\pm}(\lambda) = h^*_{\pm} + \chi R_{\lambda/c_0}g_{\pm}\;\textrm{for}\; \lambda \in \CC.
\end{align*}
Clearly, $h^*_{\pm} \in H_{\textrm{comp}}^2(\R^3)$ and 
\begin{align}\label{eq:128}
h_{\pm}(z_{\pm}(\tau)) = u^b_{\pm}.
\end{align}
A straightforward calculation gives
\begin{align}
\left(H_{\rho_\tau, k_\tau} - \lambda^2\right) h_{\pm}(\lambda) &=  (z^2_\pm(\tau) - \lambda^2) h^*_\pm - \left(H_{\rho_\tau, k_\tau} - z^2_\pm(\tau)\right)\chi R_{z_\pm(\tau)/c_0}g_{\pm} \notag \\
&+ \left(H_{\rho_\tau, k_\tau} - \lambda^2\right)\chi R_{\lambda/c_0}g_{\pm}. \label{eq:119}
\end{align}
By the definition of $\chi$ and $H_{\rho_\tau, k_\tau}$, we easily find 
\begin{align}
&\left[\left(H_{\rho_\tau, k_\tau} - z^2_\pm(\tau)\right)\chi R_{z_\pm(\tau)/c_0}g_{\pm} - \left(H_{\rho_\tau, k_\tau} - \lambda^2\right)\chi R_{\lambda/c_0}g_{\pm}\right](x) \notag\\
& = \begin{cases}
\; 0,  \quad\quad\qquad\qquad\qquad\qquad\qquad\qquad\qquad\qquad\qquad \qquad \qquad  \textrm{if}\; x\in \R^3 \backslash B_{r_2}\; \textrm{or}\; x \in B_{r_1}, \\
\big(z_{\pm}(\tau) - \lambda\big) \left[\left(\Delta - z^2_{\pm}(\tau)\right)\chi R'_{z_{\pm}(\tau)}g_\pm - 2z_{\pm}(\tau)\chi R_{z_{\pm}(\tau)}g_\pm \right](x) + \textrm{Res}_3^\pm(x), \quad \textrm{else},
\end{cases} \label{eq:120}
\end{align}
where 
\begin{align} \label{eq:121}
\left\|\textrm{Res}_3^\pm(x)\right\|_{L^2(B_{r_2}\backslash B_{r_1})} \le C|z_\pm(\tau) -\lambda|^2
\end{align}
and for $z\in \mathbb C$,
\begin{align*}
R'_z: L^2_{\textrm{comp}}(\R^3) \rightarrow L^2_{\textrm{loc}}(\R^3), \quad \left(R'_z \phi\right)(x):=\frac{i}{4\pi c_0}\int_{\R^3} e^{iz|x-y|/c_0}\phi(y)dy, \quad  x\in \R^3.
\end{align*}
Since, for any outgoing solution $f\in H^2_{\textrm{loc}}(\R^3)$, $R_{H_{\rho_\tau, k_\tau}}(\lambda)(H_{\rho_\tau, k_\tau} - \lambda^2) f = f$, we have 
\begin{align*}
R_{H_{\rho_\tau, k_\tau}}(\lambda)(H_{\rho_\tau, k_\tau} - \lambda^2)  h_\pm(\lambda) = h_\pm(\lambda).
\end{align*}
Taking $\lambda \rightarrow z_\pm(\tau)$ and using \eqref{eq:122}, \eqref{eq:129}, \eqref{eq:128}, \eqref{eq:119}, \eqref{eq:120} and \eqref{eq:121}, we arrive at 
\begin{align*}
u^b_\pm &=  |c_{{\textrm{norm}}}|^2 u^b_\pm \int_{B_{r_2}} \frac{\overline {u^b_\pm(x)} h^*_\pm}{k_{\tau}(x)} dx \\
&+ \frac{u^b_\pm |c_{{\textrm{norm}}}|^2}{2z_\pm(\tau)k_0}\int_{B_{r_2}\backslash B_{r_1}} \left[\left(\Delta - z^2_{\pm}(\tau)\right)\chi R'_{z_{\pm}(\tau)}g_\pm - 2z_{\pm}(\tau)\chi R_{z_{\pm}(\tau)}g_\pm \right](x) \overline{u^b_\pm(x)} dx. 
\end{align*}
This, together with \eqref{eq:65} and \eqref{eq:70} yields that 
\begin{align*}
1 = \frac{|c_{{\textrm{norm}}}|^2 }{\tau}\left(\int_{\Omega}\frac{|u^b_{\pm}(x)|^2}{k_1} dx + O(\sqrt{\tau})\right)\quad \textrm{for sufficiently small}\; \tau.
\end{align*}
Therefore, with the aid of \eqref{eq:127}, we have \eqref{eq:130}.
Building upon \eqref{eq:130}, \eqref{eq:127}, \eqref{eq:63} and \eqref{eq:129}, we obtain the last assertion of this theorem. 
\end{proof}

Building upon Lemma \ref{le:9} and Lemma \ref{le:6}, we obtain the following result regarding the large time behavior of the wave field, solving \eqref{eq:101} and \eqref{eq:102}.

\begin{theorem} \label{th:7}
Given $\tau,\vep >0$, let $u_{\tau,\vep}$ be the solution of equations \eqref{eq:101}--\eqref{eq:102}. Suppose that $\mathcal K$ is a compact set in $\R^3$. For every $F \in C_c^{M}\left(\R_+;L_{\mathrm{comp}}^2(\R^3)\right)$, with $M \in \mathbb N$ and $\mathrm{supp}_x F \in \mathcal K$, and $\zeta \in(0,M)$, then there exists a fixed $T$, only depending on the diameter of $\mathrm{supp}_t F$, such that when $\vep^{M-\zeta} = o(\tau)$, the following asymptotic expansion holds:
\begin{align*}
u_{{\tau,\vep}}(x,t) & =  \frac{\tau}{\vep}\frac{\mathcal C^2_\Omega k_1}{|\Omega|}\sum_{\pm}\frac{-ie^{-it\omega^{\pm}_M}}{2\omega^{\pm}_M} \frac{e^{i{\omega^{\pm}_M}|x-y_0|/c_0}}{4\pi|x-y_0|} \int_{\R^3} \frac{e^{-i \omega^{\pm}_M|y_0-y|/c_0}}{4\pi |y_0-y|} \hat F (y; -i\omega^{\pm}_M)dy + u_{\tau,\vep}^{res}(x,t)
\end{align*}
with $u_{\tau,\vep}^{\mathrm{res}}$ satisfying that
\begin{align*}
\left\|u_{\tau,\vep}^{\mathrm{res}}(\cdot, t)\right\|_{L^2(\mathcal K)} \le C \left[\frac{\tau \max(1,(\tau/\vep)^{1/2})}{\vep^{1/2}}e^{\mathrm{Im}(w_M^{\pm})t}\|F\|_{L^2(\R^3)} + \vep^{M-1-\zeta}\left\|\partial^M F\right\|_{L^2(\R^3)}\right]\\
\qquad\qquad\qquad\qquad\qquad\qquad\qquad \mathrm{for}\; t \in (T, \vep^{-\zeta}], \quad \mathrm{as}\; \vep \rightarrow 0.
\end{align*}
Here, $\omega^\pm_M:= z_{\pm}(\tau)/\vep$, $\hat F$ is the Fourier-Laplace transform of $F$ and $C$ is a positive constant depending on $\mathcal K$, but independent of $\tau$ and $\vep$.

\end{theorem}

\begin{proof}
For any $f\in L^2_{\textrm{comp}}(\R^3)$, it can be seen that 
\begin{align}
(R_{H_{\tau,\vep}(y_0)}(z) k_{\tau,\vep}f)(x) = \vep^2 \left(R_{H_{\rho_\tau, k_\tau}}(\vep z) \widetilde {k_{\tau,\vep}f}\right)(y_0 + (x-y_0)/\vep), \quad x \in \R^3. \label{eq:83}
\end{align}
Here, 
\begin{align*}
\widetilde {k_{\tau,\vep} f}(x) := \left(k_{\tau,\vep}f\right)(y_0 + \vep(x-y_0))\quad \mathrm{for}\; x\in \R^3.
\end{align*}
Since $F\in C_c^{\infty}\left(\R_+;L_{\mathrm{comp}}^2(\R^3)\right)$, we assume that $F(x,t) = 0$ for almost every $x\in \R^3$ when $t\in [ T,+ \infty) \cup (-\infty,0]$.  With the aid of Lemma \ref{le:9}, given each $\sigma>0$, we can rewrite $u_{\tau,\vep}$ as
\begin{align}
u_{\tau,\vep}(x,t) &= \left[\int_{I_{\sigma,\vep}} + \int_{\Gamma_\sigma \backslash I^\vep_{\sigma}}\right]\frac{e^{-i\lambda t}}{2\pi}\left(R_{H_{\tau,\vep}(y_0)}({\lambda}) k_{\tau,\vep} \hat F (\cdot, -i\lambda)\right)(x) d\lambda \notag\\
& = u^{(1)}_{\tau,\vep}(x,t) + u^{(2)}_{\tau,\vep}(x,t), \quad t > 0, \; x \in \R^3. \label{eq:80}
\end{align}
Here, $I \subset \R$ is a bounded interval, which is symmetric about the origin and excludes any nonzero element in $\Lambda_N$ with $\Lambda_N$ specified by \eqref{eq:82}, $\Gamma_\sigma$ is given by \eqref{eq:163}, and 
\begin{align*}
I^\vep_{\sigma} := \{ z\in \mathbb C: \vep \textrm{Re}(z) \in I,\; \textrm{Im}(z) = \sigma\}.
\end{align*}
The rest of the proof consists of two steps. The first step involves estimating $u_{\tau,\vep}^{(2)}$ and the second step addresses the estimate of $u_{\tau,\vep}^{(1)}$.

\textbf{Step 1}: In this step, we estimate $u_{\tau,\vep}^{(2)}$. 
By integration by parts, there exists a positive constant $C$ such that for all $z\in \mathbb C$ and $x\in \mathcal K$
\begin{align} \label{eq:136}
\left|\hat F(x, -iz)\right| \le C (1+|z|)^{-M} e^{-T\mathrm{Im}(z)}\left\|\partial^M_t F\right\|_{L^2(\R^3)}.
\end{align}
From this, together with the following inequality (see its derivation in the appendix \ref{sec:a23})
\begin{align} \label{eq:131}
\|R_{H_{\tau,\vep}(y_0)}(\lambda) k_{\tau,\vep} f\|_{L^2(\R^3)} \le \frac{C}{\sigma} \|f\|_{L^2(\R^3)}, \quad  \textrm{for}\; \lambda \in \Gamma_\sigma  \;\textrm{and}\; f \in L^2(\R^3),
\end{align}
we arrive at: for $t\in \R_+$,
\begin{align*}
\left\|u_{\tau,\vep}^{(2)}(\cdot,t)\right\|_{L^2(\mathcal K)} \le C \frac{e^{\sigma t}}\sigma \int^{+\infty}_{|I|/2\vep} \frac 1{(1+|s|)^{M}} ds \left\|\partial^M_t F\right\|_{L^2(\R^3)} \le C \vep^{M-1}\frac{e^{\sigma t}}{\sigma} \left\|\partial^M_t F\right\|_{L^2(\R^3)}.
\end{align*}

\begin{figure}[htbp]
\centering
\includegraphics{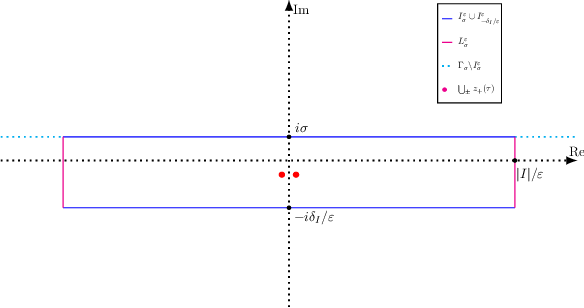}
\caption{The integral curves $I_{\sigma}^{\vep}$, $L_{\sigma}^{\vep}$ and $I_{-\delta_I/\vep}^{\vep}$} \label{fi2}
\end{figure}

\textbf{Step 2}: In this step, we establish an asymptotic estimate of $w_{\tau,\vep}^{(1)}$.
Building upon Theorem \ref{th:1}, Lemma \ref{le:6} and formula \eqref{eq:83}, we obtain that for sufficiently small $\tau$, apart from $z_{\pm}(\tau)$, specified in \eqref{eq:65}, no resonances of $H_{\tau,\vep}$ lie in the following strip 
\begin{align*}
\left\{z \in \mathbb C: \vep \textrm{Re}(z) \in \overline I,\; \textrm{Im}(z)\in [-\delta_I/\vep, \sigma]\right\}.
\end{align*}
Here, $\delta_I \in \R_+$ is specified in Theorem \ref{th:1}. Therefore, with the aid of \eqref{eq:122} and Cauchy integral theorem (see Figure \ref{fi2} for the transformation of the contour), we have 
\begin{align} \label{eq:137}
u_{\tau,\vep}^{(1)}(x,t) = -i 
\sum_{\pm}\frac{e^{-it\omega^\pm_M}}{2\omega^\pm_M} \Pi_{H_{\tau,\vep}(y_0)}\left[\left(\omega^\pm_M\right)k_{\tau,\vep} \hat F\left(\cdot, -i\omega^\pm_M\right)\right](x) + u_{\tau,\vep}^{(1,1)}(x,t) + u_{\tau,\vep}^{(1,2)}(x,t),
\end{align}
where
\begin{align*}
&u_{\tau,\vep}^{(1,1)}(x,t) = \int_{L_{\sigma}^{\vep}}\frac{e^{-i\lambda t}}{2\pi}\left(R_{H_{\tau,\vep}(y_0)}({\lambda}) k_{\tau,\vep} \hat F (\cdot, -i\lambda)\right)(x) d\lambda,\\
&u_{\tau,\vep}^{(1,2)}(x,t) = \int_{I_{-\delta_I/\vep}^{\vep}}\frac{e^{-i\lambda t}}{2\pi}\left(R_{H_{\tau,\vep}(y_0)}({\lambda}) k_{\tau,\vep} \hat F (\cdot, -i\lambda)\right)(x) d\lambda.
\end{align*}
Here, 
\begin{align*}
L_{\sigma}^{\vep}:= \left\{z \in \mathbb C: \textrm{Re}(z) = \pm|I|/\vep,\; \textrm{Im}(z)\in [-\delta_I/\vep, \sigma]\right\}.
\end{align*}
For the estimate of $u_{\tau,\vep}^{(1,1)}$ and $u_{\tau,\vep}^{(1,2)}$, due the choice of $I$, which guarantees the uniform boundedness of $R_{H_{\tau,\vep}(y_0)}({\lambda})$ in $L_{\sigma}^{\vep} \cup I_{-\delta_I/\vep}^{\vep} $, we can use \eqref{eq:136} to obtain
\begin{align}
&\|u_{\tau,\vep}^{(1,1)}(\cdot,t)\|_{L^2(\mathcal K)} \le C\vep^{M-1} e^{\sigma t} e^{-(t-T)\delta_I/\vep}\left\|\partial^M_t F\right\|_{L^2(\R^3)}, \label{eq:138}\\
&\|u_{\tau,\vep}^{(1,2)}(\cdot,t)\|_{L^2(\mathcal K)} \le C \vep^{M-1} e^{-(t-T)\delta_I/\vep}\left\|\partial^M_t F\right\|_{L^2(\R^3)}, \quad t > T. \label{eq:139}
\end{align}

We proceed to estimate the first term in \eqref{eq:137}. By \eqref{eq:64} and \eqref{eq:83}, it is easy to verify that
\begin{align} \label{eq:135}
\left(\Pi_{H_{\tau,\vep}(y_0)}(z_-(\tau)/\vep) k_{\tau,\vep}f\right) (x) = \left(\Pi_{H_{\rho_\tau, k_\tau}}(z_-(\tau)) \widetilde {k_{\tau,\vep}f} \right)(y_0 + (x-y_0)/\vep), \quad f\in L^2_{\textrm{comp}}(\R^3).
\end{align}
Let $u_\pm$ be given by Lemma \ref{le:6}. Since $D_N(0) = -(1/2-K^*_{\Gamma,0})S^{-1}_{\Gamma,0}$ and $-1/2$ is the simple eigenvalue of $K^*_{\Gamma,0}$ with the corresponding eigenfunction being $\left(S^{-1}_{\Gamma,0}\right)(x)$, utilizing \cite[Lemma 2.3]{LS-04}, \eqref{eq:132} and \eqref{eq:133}, we find
\begin{align} \label{eq:124}
u_\pm(y_0 + (x-y_0)/\vep) =  \sqrt{\frac{\tau k_1}{|\Omega|}} \left[\vep\frac{e^{i{w_M^{\pm}}|x-y_0|/c_0}}{4\pi|x-y_0|} \mathcal C_\Omega + O_{L^2(\mathcal K)}(\vep^{3/2})\right].
\end{align}
From \eqref{eq:124}, a straightforward calculation gives 
\begin{align*}
& \int_{\R^3} \overline{u_{\pm}(x)} \frac{\widetilde {k_{\tau,\vep}f}(x)}{k_\tau(x)} dx = \frac{1}{\vep^3}\int_{\R^3} \overline{u_\pm(y_0 + (x-y_0)/\vep)} f(x) dx \\
& = \frac{\mathcal C_\Omega}{\vep^{3}}\sqrt{\frac{\tau k_1}{|\Omega|}} \left[\vep \int_{\R^3} \frac{e^{-i \omega^{\pm}_M|y_0-y|/c_0}}{4\pi |y_0-y|} f (y)dy  + O_{L^2(\mathcal K)}(\vep^{3/2}\|f\|_{L^2(\R^3)}) \right] \quad \textrm{for}\; f \in L^2_{\textrm{comp}}(\R^3). 
\end{align*}
Combining this with \eqref{eq:134}, \eqref{eq:135} and \eqref{eq:124} yields that
\begin{align}
\left(\Pi_{H_{\tau,\vep}(y_0)}(z_\pm(\tau)/\vep) k_{\tau,\vep}f\right)(x) = \frac{\tau}{\vep}\frac{\mathcal C^2_\Omega k_1}{|\Omega|} \frac{e^{i{\omega^{\pm}_M}|x-y_0|/c_0}}{4\pi|x-y_0|}\int_{\R^3} \frac{e^{-i \omega^{\pm}_M|y_0-y|/c_0}}{4\pi |y_0-y|} f (y)dy + \mathrm{Rem}^f_1(x), \notag
\end{align}
where $\mathrm{Rem}^f_1$ satisfies
\begin{align*}
\|\mathrm{Rem}^f_1\|_{L^2(\mathcal K)} \le C \frac{\tau \max(1,(\tau/\vep)^{1/2})}{\vep^{1/2}}\|f\|_{L^2(\R^3)} \quad \textrm{for}\; f \in L^2_{\textrm{comp}}(\R^3). 
\end{align*}
This, together with \eqref{eq:137}, \eqref{eq:138} and \eqref{eq:139} we obtain 
\begin{align*}
u_{\tau,\vep}^{(1)}(x,t) =  \frac{\tau}{\vep}\frac{\mathcal C^2_\Omega k_1}{|\Omega|}\sum_{\pm}\frac{-ie^{-it \omega^{\pm}_M}}{2\omega^{\pm}_M} \frac{e^{i{\omega^{\pm}_M}|x-y_0|/c_0}}{4\pi|x-y_0|} \int_{\R^3} \frac{e^{-i \omega^{\pm}_M|y_0-y|/c_0}}{4\pi |y_0-y|} \hat F (y; -i\omega^{\pm}_M)dy +  \mathrm{Rem}^F_2(x,t),
\end{align*}
where $\mathrm{Rem}^F_2(\cdot,t)$ satisfies
\begin{align*}
\|\mathrm{Rem}^F_2(\cdot,t)\|_{L^2(\mathcal K)} \le C \left[\frac{\tau \max(1,(\tau/\vep)^{1/2})}{\vep^{1/2}}e^{\mathrm{Im}(w_M^{\pm})t}\|F\|_{L^2(\R^3)} + \vep^{M-1} e^{\sigma t}\left\|\partial^M_t F\right\|_{L^2(\R^3)}\right]
\end{align*}
for $t>T$.

In conjunction with the formula \eqref{eq:80} and the results of two previous steps for the case when $\sigma = \vep^\zeta$, we obtain the assertion of this theorem.
\end{proof}

\begin{appendices}
\renewcommand{\theequation}{\Alph{section}.\arabic{equation}}

\section{} \label{sec:a}

\refstepcounter{subsection}
\subsection*{\thesubsection\quad Gohberg and Sigal theory}

This section provides a brief introduction to the Gohberg and Sigal theory used in this paper, for more details, we refer to \cite{AK-09, DM}.

\begin{definition}
Let $X$ and $Y$ be two Banach spaces. Suppose that $V \subset \mathbb C$ is a connected open set. We say that $\lambda \mapsto B(\lambda)\in \mathcal L(X,Y)$ is a meromorphic family of operators in $V$ if for any $\mu \in V$, there exist operators $B_j,\; 1\le j\le J$, with a finite range, and a family of operators $\lambda \mapsto B_0(\lambda)$, holomorphic near $z$, such that 
\begin{align}\label{eq:47}
B(\lambda) = B_0(\lambda) + \frac{B_1}{\lambda-\mu} + \cdots \frac{B_J}{(\lambda-\mu)^{J}} \quad \textrm{near}\; \mu.
\end{align}
Furthermore, we say that $B(\lambda)$ is a meromorphic family of Fredholm operators if for every $\mu \in V$, $B_0(\lambda)$ is a Fredholm operator for $\lambda$ near $\mu$. In particular, for nonsingular $\mu$, $B_0(\lambda) = B(\lambda)$. 
\end{definition}

\begin{definition}
Let $X$ be a Banach space. We say that $\lambda \mapsto B(\lambda)\in \mathcal L(X)$ is normal at $\mu$ if $B(\mu)$ has an expansion \eqref{eq:47} with $B_0(\mu)$ being a Fredholm operator of index $0$, and $B(\lambda)$ is holomorphic and invertible in a punctured neighborhood of $\mu$. Furthermore, we call that $B(\lambda)$ is normal with respect to $\partial V$ if the operator $B(\lambda)$ is invertible in $\overline V$, except for a finite number of points  which are normal points of $B(\lambda)$. Here, $V$ is a simply bounded domain with rectifiable boundary $\partial V$.
\end{definition}

\begin{theorem} \label{le:a2}
Let $X$ be a Banach space. Suppose that $\lambda \mapsto B(\lambda)\in \mathcal L(X)$ with $\lambda \in V$ is a meromorphic family of Fredholm operators. Here, $V \subset \mathbb C$ is a simply bounded domain with rectifiable boundary $\partial V$. If  $B_0(\mu)$ in \eqref{eq:47} has index $0$, then there exist families of operators $\lambda \mapsto G_j(\lambda),\;l=1,2$, holomorphic and invertible near $\mu$, and operators $P_l, 1\le l \le J$, such that near $\mu$,
\begin{align}
& B(\lambda) = G_1(\lambda) \left(\mathcal P_0 + \sum_{l=1}^{J} (\lambda - \mu)^{k_l} \mathcal P_l\right)G_2(\lambda), \quad k_l \in \mathbb Z \backslash \{0\}. \label{eq:51}
\end{align}
Here, $\mathcal P_1,\ldots, \mathcal P_J$ are one-dimensional operators, $\mathcal P_{l_1}\mathcal P_{l_2} = \delta_{l_1l_2} \mathcal P_{l_2}$ with $0\le l_1,l_2 \le J$, and $\mathbb I - \sum^J_{l=0}\mathcal P_l$ is a finite dimensional operator. 
\end{theorem}

\begin{remark} \label{re:1}
Let $B(\lambda)$ be the family of operators considered in Theorem \ref{le:a2}. We have the following interpretations.
\begin{enumerate}[(a)]
\item \label{h1} $B(\lambda)$ is normal at $\mu$ if and only if $\sum^J_{l=0}\mathcal P_l = \mathbb I$, in which case,
\begin{align*}
 \left(B(\lambda)\right)^{-1} = \left(G_2(\lambda)\right)^{-1} \left(\mathcal P_0 + \sum_{l=1}^{J} (\lambda - \mu)^{-k_l} \mathcal P_l\right)\left(G_1(\lambda)\right)^{-1}, \quad k_l \in \mathbb N \backslash \{0\}.
\end{align*}
\item The factorization \eqref{eq:51} of $B(\lambda)$ provides a definition of a null multiplicity of $B(\lambda)$ at $\mu$, defined by 
\begin{align*}
N_{\mu}(B):= \begin{cases}
\sum_{k_l > 0} k_l & {\rm{if}}\;  J = {\rm{dim}}({\rm{Ran}}(\mathbb I-P_0)),\\
\infty  &{\rm{if}}\; J < {\rm{dim}}({\rm{Ran}}(\mathbb I-P_0)).
\end{cases}
\end{align*}
When $N_{\mu}(B) < \infty$, $B(\lambda)^{-1}$ is meromorphic and 
\begin{align*}
N_{\mu}(B^{-1}):=\sum_{k_l < 0} k_l.
\end{align*}
\end{enumerate}
\end{remark}

\begin{theorem}[Generalized Rouché's Theorem] \label{th:a2}
Let $X$ be a Banach space. Suppose that $\lambda \mapsto B_1(\lambda)\in \mathcal L(X)$ with $\lambda \in V$ is normal with respect to $\partial V$. Here, $V \subset \mathbb C$ is a simply bounded domain with rectifiable boundary $\partial V$. Assume that $\lambda \mapsto B_2(\lambda)\in \mathcal L(X)$ with $\lambda \in V$ is a meromorphic family of Fredholm operators.
If 
\begin{align*}
\|B^{-1}_1(\lambda) B_2(\lambda)\|_{\mathcal L(X)} < 1, \lambda \in \partial V,
\end{align*}
we have that $B_1 + B_2$ is also normal with respect to $\partial V$ and that
\begin{align*}
\sum_{\mu \in V} N_{\mu}(B_1)- N_{\mu}(B_1^{-1}) = \sum_{\mu \in V} N_{\mu}(B_1+ B_2)- N_{\mu}((B_1+B_2)^{-1}).
\end{align*}
\end{theorem}

\refstepcounter{subsection}
\subsection*{\thesubsection\quad Proofs of Lemmas \ref{le:1} and \ref{le:12}, and inequality (\ref{eq:131})}\label{sec:a2}

\refstepcounter{subsubsection}
\subsubsection*{\thesubsubsection\quad Proof of Lemma \ref{le:1}} \label{sec:a21}

\begin{proof}[Proof of Lemma \ref{le:1}]
Introducing two bilinear forms
\begin{align*}
&b_1(\phi_1, \phi_2) = \langle \phi_1, \phi_2 \rangle_{L^2(\Omega)} + \langle \nabla \phi_1, {\nabla \phi_2}\rangle_{\mathbb L^2(\Omega)} \quad \textrm{for}\; \phi_1, \phi_2 \in H^1(\Omega),\\
&\textrm{and}\; b_2(\phi_1,\phi_2) = \langle \phi_1, \phi_2\rangle_{L^2(\Omega)}, \quad \textrm{for}\; \phi_1 \in L^2(\Omega)\; \textrm{and}\; \phi_2 \in H^1(\Omega).
\end{align*}
By Riesz representation theorem, we obtain that
\begin{align*}
& b_1(\phi_1, \phi_2) = \left(\mathcal B_1 \phi_1, \phi_2 \right)_{H^1(\Omega)} \quad \textrm{for}\; \phi_1, \phi_2 \in H^1(\Omega),\\
&\textrm{and}\quad b_2(\phi_1 , \phi_2) = \left(\mathcal B_2 \phi_1, \phi_2\right)_{H^1(\Omega)} \quad \textrm{for}\; \phi_1 \in L^2(\Omega)\; \textrm{and}\; \phi_2 \in H^1(\Omega).
\end{align*}
Here, $\mathcal B_1 \in \mathcal L(H^1(\Omega))$ and $\mathcal B_2 \in  \mathcal L(L^2(\Omega), H^1(\Omega))$. It easily follows from Lax-Milgram theorem that $\mathcal B_1$ is invertible from $H^1(\Omega)$ to $H^1(\Omega)$. Then, due to the regularity assumption of $\Gamma$, it can be seen that for each $\lambda \in \mathbb C \backslash \{0\}$, $\lambda \mathbb I + \mathcal B^{-1}_1 \mathcal B_2$ is a Fredholm operator with index $0$ in $\mathcal L(L^2(\Omega))$. We note that $\lambda \mathbb I + \mathcal B^{-1}_1 \mathcal B_2$ is invertible if and only if $-\lambda \in \R_+$ is an eigenvalue of $\mathcal B^{-1}_1 \mathcal B_2$. Furthermore, it can be seen that 
\begin{align*}
\langle \mathcal B^{-1}_1 \mathcal B_2 \phi_1, \phi_2 \rangle_{L^2(\Omega)} = \langle \mathcal B^{-1}_1 \mathcal B_2 \phi_1, \mathcal B_2 \phi_2\rangle_{H^1(\Omega)} = \langle \mathcal B_2 \phi_1, \mathcal B^{-1}_1 \mathcal B_2 \phi_2\rangle_{H^1(\Omega)} = \langle \phi_1, \mathcal B^{-1}_1 \mathcal B_2 \phi_2\rangle_{L^2(\Omega)}.
\end{align*}
Therefore, when $-\lambda$
is an eigenvalue of $\mathcal B^{-1}_1 \mathcal B_2$, we have that 
\begin{align} \label{eq:99}
\lambda \mathbb I + \mathcal B^{-1}_1 \mathcal B_2 + \mathcal P_{\textrm{Ker}(\lambda\mathbb I + \mathcal B_1^{-1}\mathcal B_2)}\; \textrm{is invertible in} \; \mathcal L(L^2(\Omega)).
\end{align}
Here, $\mathcal P_{\textrm{Ker}(\lambda\mathbb I + \mathcal B_1^{-1} \mathcal B_2)}$ is an orthogonal projection onto ${\textrm{Ker}(\lambda\mathbb I + \mathcal B_1^{-1}\mathcal B_2)}$. 

On the other hand, for any $\phi_1, \phi_2 \in H^1(\Omega)$, we easily find
\begin{align*}
-\langle \nabla \phi_1, {\nabla \phi_2}\rangle_{\mathbb L^2(\Omega)} + \kappa^2 \langle \phi_1 , \phi_2 \rangle_{L^2(\Omega)} 
&= - \left[b_1(\phi_1 , \phi_2) - (1 + \kappa^2) b_2(\phi_1,\phi_2)\right], \\
& = - \left[\left(\mathcal B_1 \phi_1, \phi_2\right)_{H^1(\Omega)} - (1 + \kappa^2) (\mathcal B_2\phi_1, \phi_2)_{H^1(\Omega)}\right].
\end{align*}
This implies that each Neumann eigenvalue $\kappa$ is a point at which $1/(1+\kappa^2)$ is the eigenvalue of $\mathcal B_1^{-1} \mathcal B_2$. Therefore, we arrive at 
\begin{align*}
 a_{z_0}(\phi_1, \phi_2) &= -\left(\mathcal B_1 \phi_1, \phi_2\right)_{H^1(\Omega)} +  \left(1 + \frac{z^2_0}{c^2_1}\right) (\mathcal B_2\phi_1,\phi_2)_{H^1(\Omega)} \notag\\
&+ \left(\mathcal P_{{\rm{Ker}}((1+z_0^2/c^2_1)^{-1}\mathbb I + \mathcal B_1^{-1} \mathcal B_2)} \phi_1, \phi_2 \right)_{L^2(\Omega)}.
\end{align*}
Then, \eqref{eq:71} can be equivalently represented by 
\begin{align*}
\left[\mathbb I - \left(1 + \frac{z^2_0}{c^2_1}\right) B_1^{-1}\mathcal B_2 + \mathcal B_1^{-1} \mathcal B_2 \mathcal P_{{\rm{Ker}}((1+z_0^2/c^2_1)^{-1}\mathbb I + \mathcal B_1^{-1} \mathcal B_2)}\right] v^f = B_1^{-1}f.
\end{align*}
Combining this and \eqref{eq:99} yields the solvability of \eqref{eq:71}.

Moreover, from \eqref{eq:20}, we observe that the sesquilinear form $a^{\rm{dom}}_{z,\tau}$ is subject to a small analytic perturbation of $a_{z_0}$ in terms of $\tau$ and $z-z_0$, represented by $a^{\rm{res}}_{z,\tau}$. Therefore, by using the Born series inversion method and applying statement \eqref{a1} of this lemma, we obtain the assertion of statement \eqref{a2} of this lemma. 
\end{proof}

\refstepcounter{subsubsection}
\subsubsection*{\thesubsubsection\quad Proof of Lemma \ref{le:12}} \label{sec:a22}

\begin{proof}[Proof of Lemma \ref{le:12}]
\eqref{z1} 
A straightforward calculation gives
\begin{align}
&\int_{\R^3}\frac{e^{i(\omega\vep^{-1} + \omega\vep^{\alpha -1})|y_0 + \vep(x-y_0)- y|}}{|y_0 + \vep(x-y_0) - y|} f(y)dy \notag\\
&= \left\{\int_{U_1(\vep)}+ \int_{U_2(\vep^\sigma)} + \int_{U_3(\vep^\sigma)}\right\} \frac{e^{i(\omega\vep^{-1} + \omega\vep^{\alpha -1})|y_0 + \vep(x-y_0)- y|}}{|y_0 + \vep(x-y_0) - y|} f(y)dy =: g_1(x) + g_2(x) + g_3(x), \label{eq:103}
\end{align}
where $\sigma \in (0,1)$, $U_1(\vep)$, $U_2(\vep^\sigma)$ and $U_3(\vep^\sigma)$ are defined by 
\begin{align}
&U_1(\vep):=\{y \in \R^3: |y-y_0| \le 2\vep \max_{x\in\overline \Omega}|x-y_0|\}, \notag \\
&U_2(\vep^\sigma):=\{y \in \R^3: 2\vep \max_{x\in\overline \Omega}|x-y_0| <|y-y_0| <\vep^\sigma\}, \notag\\
&U_3(\vep^\sigma):=\{y \in \R^3: |y-y_0| \ge \vep^\sigma)\}. \notag
\end{align}

For the estimate of $g_1$, we observe that 
\begin{align*}
g_1(x) = \frac{1}{\vep} \int_{U_1(\vep)}\frac{e^{i(\omega + \omega\vep^{\alpha})|(x-y_0) -(y-y_0)/\vep|}}{|(x-y_0) - (y-y_0)/\vep|} f(y)dy \\
= \vep^2 \int_{(U_1(\vep)-y_0)/\vep}\frac{e^{i(\omega + \omega\vep^{\alpha})|x-y|}}{|x-y|}f(y_0 + \vep(y-y_0))dy 
\end{align*}
Thus, we have 
\begin{align}
\|g_1\|_{H^1(\Omega)} \le \vep^{1/2} \|f\|_{L_{\beta}^2(\R^3)}. \label{eq:153}
\end{align} 
In the sequel, we estimate $g_2$ and $g_3$. It is easy to verify that 
\begin{align} \label{eq:104}
& e^{i(\omega\vep^{-1} + \omega\vep^{\alpha -1}) a} = e^{i\omega\vep^{-1}a}\left(1 + i\omega a \vep^{\alpha}\sum^{\infty}_{l=1}\frac{1}{l!}\left(i\omega \vep^{\alpha}a\right)^{l-1}\right), \quad a \in \R, 
\end{align}
and that 
\begin{align} 
& |y_0 + \vep(x-y_0) - y| = |y_0-y|\sqrt{1 + 2\vep \frac{(x-y_0)\cdot (y_0 - y)}{|y_0-y|^2} + {\frac{\vep^2|x-y_0|^2}{|y_0-y|^2}}}, \label{eq:107}\\
& {\left|\frac{y-y_0}\vep + (y_0-x)\right|} = \frac{|y_0-y|}{\vep}\sqrt{1 + 2\vep\frac{(y_0-x)\cdot(y-y_0)}{|y_0-y|^2} + {\frac{\vep^2|y_0 - x|^2}{|y_0-y|^2}}}. \label{eq:144}
\end{align}
Using \eqref{eq:107} and applying Cauchy-Schwartz inequality, we find
\begin{align}
|g_2(x)| &\le C\int_{U_2(\vep^\sigma)}\frac 1 {|y_0-y|}|f(y)|dy \notag \\
&\le C\vep^{\frac{3\sigma(q-2)}{2q}} \||y_0-y|^{-1}\|_{L^q(\R^3)}\|f\|_{L_{\beta}^2(\R^3)} \le C\vep^{\frac{\sigma}2}\ln\vep\|f\|_{L_{\beta}^2(\R^3)}. \label{eq:154} 
\end{align}
Here, the last inequality is obtained by setting $q = 3$.
Similarly, we obtain
\begin{align}
|\nabla g_2(x)| \le C \int_{U_2(\vep^\sigma)}\frac 1 {|y_0-y|}|f(y)|dy \le C\vep^{\frac{\sigma}2}\ln\vep\|f\|_{L_{\beta}^2(\R^3)}, \quad x\in \Omega. \label{eq:155}
\end{align}
Moreover, with the aid of \eqref{eq:104}, \eqref{eq:107},\eqref{eq:144} and the inequality
\begin{align*}
 \int_{\R^3} \frac{1}{|y_0-y|} |f(y)|dy \le C\|f\|_{L^2_{\beta}(\R^3)},   
\end{align*}
we arrive at 
\begin{align} \label{eq:156}
\left\| g_3(x) - \int_{\R^3} \frac{e^{i{\omega}|y_0-y|/\vep}}{4\pi|y_0-y|}e^{i{\omega}\hat y_{y_0}\cdot(y_0-x)}f(y) dy \right\|_{H^1(\Omega)} \le \vep^{\min(\alpha -1,1-\sigma)}\|f\|_{L_{\beta}^2(\R^3)}.
\end{align}
Setting $\sigma = 2/3$
and combining \eqref{eq:103} \eqref{eq:153}, \eqref{eq:154}, \eqref{eq:155} and \eqref{eq:156}, we obtain \eqref{eq:110}.

\eqref{z2} For each $\phi\in L^2(\Omega)$ and $f \in L_\beta^2(\R^3)$, a straightforward calculation yields that
\begin{align}
&\int_{\R^3}\int_{\Omega}\frac{e^{i(\omega + \omega\vep^{\alpha})|y_0+ \vep^{-1}(x-y_0)- y|}}{4\pi|y_0+ \vep^{-1}(x-y_0)- y|}\phi(y) f(x) dy dx \notag\\
&= \vep \int_{\Omega}\int_{\R^3}\frac{e^{i(\omega\vep^{-1} + \omega\vep^{\alpha-1})|(x-y_0) + \vep(y_0-y)}|}{4\pi|(x-y_0) + \vep(y_0-y)|} f(x) \phi(y)dx dy \label{eq:106}
\end{align}
and 
\begin{align}
&\int_{\R^3} \int_\Omega \frac{e^{i{\omega}|x-y_0|/\vep}}{4\pi|x-y_0|}e^{i\omega\hat x_{y_0}\cdot(y_0-y)}\phi(y) f(x) dy dx \notag \\
&= \int_{\Omega}\int_{\R^3}\frac{e^{i{\omega}|x-y_0|/\vep}}{4\pi|x-y_0|}e^{i\omega\hat x_{y_0}\cdot(y_0-y)} f(x) \phi(y) dx dy. \notag
\end{align}
This, together with \eqref{eq:110} and \eqref{eq:106} yields \eqref{eq:105}. Furthermore, \eqref{eq:108} can be derived by proceeding as in the derivation of \eqref{eq:105}.
\end{proof}

\refstepcounter{subsubsection}
\subsubsection*{\thesubsubsection\quad Proof of inequality (\ref{eq:131})} \label{sec:a23}

\begin{proof}[Proof of inequality \eqref{eq:131}]
Given $\lambda \in \mathbb C_+$, consider a sesqulinear form 
\begin{align*}
b_\lambda(h_1,h_2):= \overline \lambda\int_{\R^3}\frac{\lambda^2}{k_{\tau,\vep}(x)}(h_1\overline{h_2})(x)dx + \overline \lambda \frac{1}{\rho_{\tau,\vep}(x)}(\nabla h_1\cdot \nabla \overline{h_2})(x) dx, \quad h_1,h_2\in H^1(\R^3).
\end{align*}
Clearly, $b_\lambda$ is coercive in $H^1(\R^3)$. Therefore, given $\lambda \in \Gamma_{\sigma}$. For any $f \in L^2(\R^3)$, there exists a unique $h^f(x,\lambda)$ satisfying
\begin{align}\label{eq:123}
-\frac{\lambda^2}{k_{\tau,\vep}(x)} h^f(x,\lambda) - \nabla \cdot \frac{1}{\rho_{\tau,\vep}(x)} \nabla h^f(x,\lambda) = f(x), \quad x\in \R^3.
\end{align}
We note that $ h^f = R_{H_{\tau,\vep}}(\lambda) k_{\tau,\vep} f$. It follows from \eqref{eq:123} that 
\begin{align}
\overline{\lambda}\int_{\R^3}\frac{|\lambda|^{2}}{k_{\tau,\vep}(x)}|h^f(x,\lambda)|^2 & + \overline \lambda \frac{1}{\rho_{\tau,\vep}(x)}\left|\nabla h^f(x,\lambda)\right|^2 dx = \int_{\R^3} \overline \lambda f(x) \overline{h^f(x,\lambda)}dx, \notag
\end{align}
whence \eqref{eq:131} follows.
\end{proof}

\refstepcounter{subsection}
\subsection*{\thesubsection\quad Fourier-Laplace transform}

For any $\varphi \in C(\R; L_{\textrm{comp}}^2(\R^3))$, its Fourier-Laplace transform  $\hat \varphi$ is defined by
\begin{align} \label{eq:125}
\left(\mathcal F \varphi(x, \cdot) \right)(s) := \hat \varphi (x,s)  = \int^{+\infty}_{0} e^{-st} \varphi(x,t) dt, \quad s \in \{z\in \CC: \textrm{Re}(z) > 0\}, 
\end{align}
It is well known that 
\begin{align} 
\Theta(t)\varphi(x,t) = \frac{1}{2\pi}\int_{-\infty}^{\infty} e^{(\textrm{Re}(s)+i\xi)t} \hat \varphi(x,\textrm{Re}(s) + i\xi) d\xi. \notag
\end{align}
Here, $\Theta(t)$ is a heaviside function, which is defined as
\begin{align}
    \Theta(t):= 
    \begin{cases}
    0, & t<0,\\
    1, & t\ge 0.
    \end{cases} \notag
\end{align}

\end{appendices}

\section*{Acknowledgment}
This work is supported by the Austrian Science Fund (FWF) grant P: 36942.

\section*{Disclosure Statement}
The authors declare no conflict of interest.

\end{document}